\documentclass[12pt,reqno]{amsart}
\usepackage[utf8]{inputenc}
\usepackage[margin=2.9cm,headheight=1000pt,textheight=30cm,marginpar=2cm]{geometry}
\usepackage{color, amssymb}
\usepackage{comment}
\usepackage{stmaryrd}
\usepackage{enumerate}

\usepackage{hyperref}
\usepackage{fancyhdr}

\newtheorem{theorem}{Theorem}
\newtheorem{definition}{Definition}
\newtheorem{lemma}{Lemma}
\newtheorem*{theorem*}{Theorem}
\newtheorem*{lemma*}{Lemma}
\newtheorem{prop}{Proposition}

\newtheorem{conjecture}{Conjecture}

\renewcommand{\implies}{\Rightarrow}
\newcommand{\ii}{\mathrm{i}} 
\newcommand{\cc}{\mathrm{c}} 
\newcommand{\re}{{\rm Re } }
\newcommand{\E}{\mathbb E} 
\newcommand{\PP}{\mathbb P}
\newcommand{\1}{{\mathbf 1}}
\newcommand{\OO}{{\rm O}}
\newcommand{\oo}{{\rm o}}
\newcommand{\dd}{{\rm d}}

\makeatletter
\def\@tocline#1#2#3#4#5#6#7{\relax
  \ifnum #1>\c@tocdepth 
  \else
    \par \addpenalty\@secpenalty\addvspace{#2}%
    \begingroup \hyphenpenalty\@M
    \@ifempty{#4}{%
      \@tempdima\csname r@tocindent\number#1\endcsname\relax
    }{%
      \@tempdima#4\relax
    }%
    \parindent\z@ \leftskip#3\relax \advance\leftskip\@tempdima\relax
    \rightskip\@pnumwidth plus4em \parfillskip-\@pnumwidth
    #5\leavevmode\hskip-\@tempdima
      \ifcase #1
       \or\or \hskip 2em \or \hskip 3em \else \hskip 4em \fi%
      #6\nobreak\relax
    \dotfill\hbox to\@pnumwidth{\@tocpagenum{#7}}\par
    \nobreak
    \endgroup
  \fi}
\makeatother

\usepackage[normalem]{ulem} 

\begin{document}

\begin{abstract} 
By analogy with conjectures for random matrices, Fyodorov-Hiary-Keating and Fyodorov-Keating  proposed precise asymptotics for the maximum of the Riemann  zeta function in a typical short interval on the critical line. In this paper, we settle  the upper bound part of their conjecture in a strong form. More precisely, we show that the measure of those $T \leq t \leq 2T$ for which 
$$
\max_{|h| \leq 1} |\zeta(\tfrac 12 + \ii t + \ii h)| > e^y \frac{\log T }{(\log\log T)^{3/4}}
$$
is bounded by $Cy e^{-2y} T$ uniformly in $y \geq 1$ with $C > 0$ an absolute constant. 
This is expected to be optimal for $y= \OO(\sqrt{\log\log T})$. This upper bound is sharper than what is known in the context of random matrices, since it gives (uniform) decay rates in $y$. In a subsequent paper we will obtain matching  lower bounds. 
\vspace{-0.5cm}
\end{abstract}

\title{The Fyodorov-Hiary-Keating Conjecture. I.}

\author{Louis-Pierre Arguin}
\address{Department of Mathematics, Baruch College and Graduate Center, City University of New York, USA}
\email{louis-pierre.arguin@baruch.cuny.edu}
\author{Paul Bourgade}
\address{Courant Institute, New York University, USA}
\email{bourgade@cims.nyu.edu}
\author{Maksym Radziwi\l\l}
\address{Department of Mathematics, Caltech, USA}
\email{maksym.radziwill@gmail.com}

\begingroup
\def\uppercasenonmath#1{} 
\let\MakeUppercase\relax 
~\vspace{-0.1cm}
\maketitle
\endgroup

\thispagestyle{empty}

~\vspace{-0.9cm}

\setcounter{tocdepth}{2}
\tableofcontents

\section{Introduction}
\fancyhead{}
\pagestyle{fancy}
\renewcommand{\headrulewidth}{0pt}
\fancyhead[LE,RO]{\thepage}
\fancyhead[CO]{L.-P. Arguin, P. Bourgade and M. Radziwi\l\l}
\fancyhead[CE]{The Fyodorov-Hiary-Keating Conjecture. I.}
 \fancyfoot{}

Motivated by the problem of understanding the global maximum of the Riemann zeta function on the critical line, Fyodorov-Keating \cite{FyodorovKeating} and Fyodorov-Hiary-Keating \cite{FyodorovHiaryKeating} raised the question of understanding the distribution of the local maxima of the Riemann zeta function  on the  critical line. They made the following conjecture.

\begin{conjecture}[Fyodorov-Hiary-Keating] \label{eq:keating}
  There exists a cumulative distribution function $F$ such that, for any $y$, as $T \rightarrow \infty$,  
  $$
  \frac{1}{T}\, \text{\rm meas} \Big \{ T \leq t \leq 2T : \max_{0 \leq h \leq 1} |\zeta(\tfrac 12 + \ii t + \ii h)| \leq e^y \frac{ \log T}{(\log\log T)^{3/4}} \Big \} \sim F(y).
  $$
  Moreover, as $y \rightarrow \infty$, the right-tail decay is $1-F(y) \sim Cy e^{-2y}$ for some $C>0$.
\end{conjecture}

The striking aspect of this conjecture is the exponent $\tfrac 34$ on the $\log\log T$ and the decay rate $1 - F(y) \ll y e^{-2y}$. This suggests that around the local maximum there is a significant degree of interaction between nearby shifts of the Riemann zeta function (on the scale $1/\log T$). If there were no interactions, one would expect an exponent of $\tfrac 14$ on the $\log\log T$ and a decay rate $e^{-2y}$ (see \cite{HarperSurvey}).

This paper settles the upper bound part of the Fyodorov-Hiary-Keating conjecture in a strong from, with uniform and sharp decay in $y$.
\begin{theorem} \label{thm:main}
There exists $C>0$ such that for any $T \geq 3$ and $y\geq 1$, we have
  $$
\frac{1}{T} \, \text{\rm meas} \Big \{ T \leq t \leq 2T : \max_{|h| \leq 1} |\zeta(\tfrac 12 + \ii t + \ii h)| > e^y\frac{ \log T}{(\log\log T)^{3/4}} \Big \} \leq Cy e^{-2y}.
$$
\end{theorem}

Theorem \ref{thm:main} is expected to be sharp in the range $y =\OO(\sqrt{\log\log T})$. For larger $y$ in the range $y \in [1, \log\log T]$, it is expected that the sharp decay rate is
$$
\ll y e^{-2y} \exp \Big ( - \frac{y^2}{\log\log T} \Big ). 
$$

Conjecture \ref{eq:keating} emerges in \cite{FyodorovHiaryKeating,FyodorovKeating} from the analogous prediction for random matrices, according to which
\begin{equation} \label{eq:conj}
\sup_{|z|=1} \log |X_n(z)| = \log n - \frac{3}{4} \log \log n + M_n,
\end{equation}
with $X_n(z)$ the characteristic polynomial of   a Haar-distributed $n\times n$  unitary matrix, and with $M_n$ converging to a random variable $M$ in distribution. 
Progress on \eqref{eq:conj} was accomplished by Arguin-Belius-Bourgade \cite{ArgBelBou15} and Paquette-Zeitouni \cite{PaqZei16}, culminating in the work of Chhaibi-Madaule-Najnudel \cite{ChaMadNaj16}. In \cite{ChaMadNaj16} it was established for the circular beta ensemble that the sequence of random variables $M_n$ is tight. The convergence of $M_n$ in distribution to a limiting random variable $M$ and the decay rate of $\mathbb{P}(M > y)$ as $y$ increases remain open.
In this regard, Theorem \ref{thm:main} is a rare instance of a result obtained for the Riemann zeta function prior to the analogue for random matrices.
  This type of decay is expected by analogy with branching random walks, but has only been proved for a few processes, notably for the two-dimensional Gaussian free field \cite{DinZei2014, BraDinZei2016}.

Previous results in the direction of Conjecture \ref{eq:keating}  were more limited than for unitary matrices. The first order, that is, 
 $$
 \max_{|h| \leq 1} \log |\zeta(\tfrac 12 + \ii t + \ii h)| \sim \log \log T \ ,  \ T \rightarrow \infty,
 $$
for all $t \in [T, 2T]$ outside of an exceptional set of measure $\oo(T)$, was established conditionally on the Riemann Hypothesis by Najnudel \cite{Naj18}, and unconditionally by the authors with Belius and Soundararajan \cite{ArgBelBouRadSou19}. Harper \cite{Harper2nd} subsequently obtained the upper bound up to second order. More precisely, Harper showed that for $t \in [T, 2T]$ outside of an exceptional subset of measure $\oo(T)$, and for any $g(T)\to\infty$,
\begin{equation}
\label{eqn: harper bound}
\max_{|h| \leq 1} \log |\zeta(\tfrac 12 + \ii t + \ii h)| \leq \log\log T - \frac{3}{4} \log\log\log T + \frac{3}{2}\log\log\log\log T+g(T).
\end{equation}

Progress towards Conjecture \ref{eq:keating} has been made by observing that the large values of $\log |\zeta(\tfrac 12 + \ii t + \ii h)|$ on  a short interval indexed by $h \in [-1,1]$ are akin to the ones of an approximate branching random walk, see for example \cite{arguin2016}. 
This is because, the average of $\log|\zeta(\tfrac 12 + \ii t + \ii h)|$ over a neighborhood of $h$ of width $e^{-k}$ for $k\leq \log\log T$ can be thought of as a Dirichlet sum $S_k$ of $p^{-1/2+\ii t+\ii h}$ up to $p\leq \exp e^k$, see Equation (\ref{eqn: Sk}) below.
The partial sums $S_k$, $k\leq \log\log T$, for different $h$'s have a correlation structure that is approximately the one of a branching random walk. 

For branching random walks, the identification of the maximum up to an error of order one relies on a precise upper barrier for the values of the random walks $S_k$ at every $k\leq \log\log T$, as introduced in the seminal work of Bramson \cite{Bra78}.  
This approach cannot work directly for $\log |\zeta|$ as one needs to control large deviations for Dirichlet polynomials involving prime numbers close to $T$. 
This amounts to computing large moments of long Dirichlet sums, and current number theory techniques cannot access these with a small error.

To circumvent this problem, the proof of Theorem \ref{thm:main} is based on an iteration scheme that recursively constructs {\it upper} and {\it lower barrier} constraints for the values of the partial sums as the scales $k$ approaches $\log\log T$. 
Each step of the iteration relies on elaborate second and twisted fourth moments of the Riemann zeta function, which may be of independent interest.
The lower barrier reduces in effect the number of $h$'s to be considered for the maximum of $\log|\zeta|$.
One upshot is that smaller values for the Dirichlet sums are needed, and thus only moments with good errors are necessary.
Furthermore, the reduction of the number of $h$'s improves the approximation of $\log |\zeta|$ in terms of Dirichlet sums for the subsequent scales in the iteration.
Lower constraints have appeared before in \cite{ArgBovKis2013} to study correlations between extrema of the branching Brownian motion. 
There, they were proved {\it a posteriori} based on the work of Bramson on the maximum.

The paper is organized as follows.
The iterative scheme is described in details in Section \ref{sec:iteratedScheme}. Its initial condition, induction and final step are proved in
Sections \ref{sec:first}, \ref{sec:induction} and \ref{sec:last}. The number-theoretic input of the recursion using second and twisted fourth moments of the Riemann zeta function is the subject of Sections \ref{sec:2nd} and \ref{sec:4th}.

 In a subsequent paper we will complement the upper bound in Theorem \ref{thm:main} with matching lower bounds, for fixed $y  > 1$. This will also rely on the multiscale analysis and on twisted moments.\\

{\noindent \bf Notations.}
We use Vinogradov's notation and write $f(T)\ll g(T)$ to mean $f(T)=\OO(g(T))$ as $T\to\infty$. If the $\OO$-term depends on some parameter $A$, we write $\ll_A$ or $\OO_A$ to emphasize the dependence.
We write $f(T)\asymp g(T)$ when $f(T)\ll g(T)$ and $g(T)\ll f(T)$.\\

{\noindent \bf Acknowledgments.}
The authors are grateful to Frederic Ouimet for several discussions, and to Erez Lapid and Ofer Zeitouni for their careful reading, pointing at a mistake in the initial proof of Lemma \ref{le:moliapprox}.
The research of LPA was supported in part by NSF CAREER
DMS-1653602.
PB acknowledges the support of NSF grant DMS-1812114 and a Poincar\'e chair.
MR acknowledges the support of NSF grant DMS-1902063 and a Sloan Fellowship.

\section{Initial Reductions}

Throughout the paper we will adopt probabilistic notations and conventions. In particular $\tau$ will denote a random variable uniformly distributed in $[T, 2T]$  and $\PP,\E$ the associated probability and expectation. Furthermore we set throughout $$n = \log\log T.$$
This notation will become natural later when $S_k$, $k\leq n$, given in Equation \eqref{eqn: Sk} will be thought of as a random walk.
We will find it convenient to consider $\zeta(\tfrac 12 + \ii \tau + \ii h)$ as a random variable and write for short $\zeta_{\tau}(h) = \zeta(\tfrac 12 + \ii \tau + \ii h)$. 
In this notation, Theorem \ref{thm:main} can be restated as follows. 

\begin{theorem*}
\label{thm: upper0}
Let $\tau$ be a uniformly distributed random variable in $[T, 2T]$. Then uniformly in $T\geq  3$, $y \geq 1$,  one has
$$
\mathbb P\Big(\max_{|h|\leq 1}|\zeta_{\tau}(h)|>e^y\, \frac{e^n}{n^{3/4}}\Big)\ll y e^{-2y}\ .
$$
\end{theorem*}
Along the proof, we will refer to well-known results, or variations of well-known results. To emphasize the core ideas of the proof, we chose to gather these in the appendix. Appendix \ref{se:moments} deals with estimates on sums of primes and on moments of Dirichlet polynomials. Appendix \ref{se:ballot} presents a version of the ballot theorem for random walks. Finally, tools for discretizing the maximum of Dirichlet polynomial on a short interval are presented in Appendix \ref{se:discretization}.
With this in mind, we first observe that it is easy to establish Theorem \ref{thm:main} for $y > n$.
\begin{lemma}\label{le:largevalues}
  Uniformly in $y > n$ we have
  \begin{equation} \label{eq:prr}
  \mathbb{P} \Big (\max_{|h| \leq 1} |\zeta_{\tau}(h)| > e^y\,\frac{e^n}{n^{3/4}}\Big ) \ll y e^{-2y}.
  \end{equation}
\end{lemma}
\begin{proof}
  By Chebyshev's inequality, the probability in \eqref{eq:prr} is 
  $$
  \leq e^{-4y} e^{-4n} n^{3} \mathbb{E} \Big [ \max_{|h| \leq 1} |\zeta_{\tau}(h)|^4 \Big ].
  $$ 
  By Lemma \ref{le:zetada} in Appendix \ref{se:discretization}, the above is
  $$
  \ll e^{-4y} e^{-4n} n^{3} e^{5n} = n^3 e^{n} e^{-4y}.
  $$
  Since $y > n$, this is $\ll y e^{-2y}$ and the claim follows. 
\end{proof}

To handle the remaining values $1 \leq y \leq  n$ it will be convenient to discretize the maximum over $|h| \leq 1$ into a maximum over a set
$$
\mathcal{T}_n=e^{-n-100}\mathbb Z\cap[-2,2].
$$
 To accomplish this, we use the following simple lemma. 
\begin{lemma}
  There exists an absolute constant $C > 1$ such that for any $V > 1$ and $A > 100$, 
  $$
  \mathbb{P}\Big (\max_{|h| \leq 1} |\zeta_{\tau}(h)| > V \Big ) \leq \mathbb{P}\Big ( \max_{h \in \mathcal{T}_n} |\zeta_{\tau}(h)| > V / C \Big ) + \OO_A(e^{-A n}).  
  $$
\end{lemma}\begin{proof}
  This is Lemma \ref{le:zetad} in Appendix \ref{se:discretization}.
  \end{proof}

Combining the above lemma with Lemma \ref{le:largevalues}, it suffices to prove the following result  to establish Theorem \ref{thm:main}. 
Without loss of generality, we state the result for $T\geq \exp(e^{1000})$ and $y > 4000$, which is more convenient for further estimates. 

\begin{theorem} \label{thm: upper}
  Let $\tau$ be a random variable, uniformly distributed in $[T, 2T]$. Then, uniformly in $T\geq \exp(e^{1000})$, $4000 \leq y \leq n$, we have
  $$
  \mathbb{P} \Big ( \max_{h \in \mathcal{T}_n} |\zeta_{\tau}(h)| > e^y \, \frac{\log T}{(\log\log T)^{3/4}} \Big ) \ll y e^{-2 y}.
  $$
\end{theorem}

\section{Iteration Scheme}\label{sec:iteratedScheme}

\subsection{Notations} In this section, we explain the structure of the proof of Theorem \ref{thm: upper}. 
We start by defining the main objects of study. Consider first the time scales
$$
T_{-1}= \exp(e^{1000}), \qquad T_0 = \exp(\sqrt{\log T}), \qquad T_{\ell} = \exp \Big ( \frac{\log T}{(\log_{\ell + 1} T)^{10^6}} \Big ),
$$
where $\ell\geq 1$ and $\log_\ell$ stands for the logarithm iterated $\ell$ times. We adopt the convention that
$
\log_0 n = n \text{ and } \log_{-1} n = e^n.
$
It is convenient to write the above in the $\log\log$-scale, denoting (remember $n=\log\log T$)
$$
n_{-1} = 1000, \qquad n_0=\tfrac{n}{2},\qquad n_{\ell} = \log\log T_{\ell} = n - 10^6 \log_{\ell} n.
$$

Consider the Dirichlet polynomial
\begin{equation}
\label{eqn: Sk}
S_k(\tfrac 12 + \ii \tau + \ii h):=S_k(h)  = \sum_{e^{1000} \leq \log p \leq e^k} \re \Big( p^{-(\tfrac 12 + \ii \tau + \ii h)} + \tfrac 12 p^{- 2(\tfrac 12 + \ii \tau + \ii h)} \Big ),  \ \ k \leq n,
\end{equation}
with $S_{n_{-1}}(h) = 0$. The above summand consists in the first two terms in the expansion of $-\log|1-p^{-s}|$. The second order may  be essentially ignored on a first reading;
however this additional term is necessary to handle the maximum of $|\zeta|$ up to tightness, due to the contribution of the small primes to $|\zeta(s)|$. 
Moreover,  starting the sum in (\ref{eqn: Sk}) at $e^{1000}$
will be convenient for some estimates in Section \ref{sec:4th}. We also define,
\begin{equation}
\label{eqn: complex S_k}
\widetilde{S}_k(\tfrac 12 + \ii \tau + \ii h):= \widetilde{S}_k(h)  = \sum_{e^{1000} \leq \log p \leq e^k}  \Big( p^{-(\tfrac 12 + \ii \tau + \ii h)} + \tfrac 12 p^{- 2(\tfrac 12 + \ii \tau + \ii h)} \Big ),  \ \ k \leq n,
\end{equation}
so that $S_k(h) = \re \ \widetilde{S}_k(h)$ and $|S_k(h)| \leq |\widetilde{S}_k(h)|$.

We use the probabilistic notation of omitting the dependence on the random $\tau$, and think of $(S_k(h))_{h\in[-2,2]}$ as a stochastic process. The dependence in $h$ will sometimes be omitted when there is no ambiguity.

It will be necessary to control the difference $\log|\zeta| -S_k$ which represents the contribution of primes larger than $e^{e^k}$. To do so, given $\ell \geq 0$, we define the following random mollifiers,
\begin{align*}
    \mathcal{M}_{\ell}(h) & = \sum_{\substack{p | m \implies p \in (T_{\ell-1},T_\ell] \\ \Omega_\ell(m) \leq (n_{\ell} - n_{\ell - 1})^{10^5}}} \frac{\mu(m)}{m^{\tfrac 12 +\ii \tau + \ii h}},
\end{align*}
  where $\Omega_\ell(m)$ stands for the number of prime factors  of $m$ in the interval $(T_{\ell-1}, T_\ell]$, counted with multiplicity, and  $\mu$ denotes the M\"obius function\footnote{We could have also counted the prime factors of $m$ without multiplicity because $m$ has to be square-free, but $\Omega_\ell(m)$ will be more consistent with other constraints appearing along the proof.}. Furthermore we set $\mathcal{M}_{-1}(h)= 1$ for all $h \in \mathbb{R}$. 
Given $\ell \geq 0$ and $k \in [n_{\ell - 1}, n_{\ell}]$, we define the mollifier up to $k$ as
\begin{align*}
    \mathcal{M}^{(k)}_{\ell - 1}(h) & = \sum_{\substack{p | m \implies p \in (T_{\ell - 1},\exp(e^k)] \\\Omega_\ell(m) \leq (n_{\ell} - n_{\ell - 1})^{10^5}} }\frac{\mu(m)}{m^{\tfrac 12 +\ii \tau + \ii h}}.
\end{align*}
This way we have $\mathcal{M}_{\ell - 1}^{(n_{\ell - 1})} = 1$ and $\mathcal{M}_{\ell - 1}^{(n_{\ell})} = \mathcal{M}_{\ell}$. 
The product $\mathcal{M}_{-1} \ldots \mathcal{M}_{\ell - 1} \mathcal{M}_{\ell -1}^{(k)}$ will be 
a good proxy for $\exp(-S_k)$ for most $\tau$,  cf. Lemma \ref{le:moliapprox} in Appendix \ref{se:moments}.

Finally, 
the deterministic centering of the maximum is denoted
$$
m(k) = k \, \Big ( 1 - \frac{3}{4} \frac{\log n}{n} \Big ). 
$$
For a fixed $y\geq 1$, we set the following upper  and lower barriers for the values of $S_k$:
\begin{align}
  \label{eq:barup}U_y(k) &= y+  \begin{cases}
    \infty & \text{ for } 1 \leq k < \lceil y / 4 \rceil, \\
 10^3 \log k & \text{ for } \lceil y / 4 \rceil \leq k \leq n/2, \\ 
10^3  \log(n - k) & \text{ for } n/2<k< n,
  \end{cases}\\
\label{eq:bardown} L_y(k) &= y - \begin{cases}
 \infty & \text{ for } 1 \leq k < \lceil y / 4 \rceil, \\ 
 20 k & \text{ for } \lceil y / 4 \rceil \leq k \leq n/2, \\
  20(n - k) & \text{ for } n/2<k<n.
 \end{cases}
 \end{align}
Note that $U_y(k)-L_y(k)$ is independent of $y$ and that $L_y(y/4)=-4y$ is negative.

\subsection{Iterated good sets}
The proof of Theorem \ref{thm:main} progressively reduces the set of $h$'s for which $\zeta$ is large.
We define iteratively the following decreasing subsets for $\ell \geq 0$:
\begin{align*}
  A_{\ell} & = A_{\ell-1}\cap \{ h \in \mathcal{T}_n:  |\widetilde{S}_k(h)-\widetilde{S}_{n_{\ell-1}}(h)| \leq 10^3(n_{\ell} - n_{\ell - 1}) \text{ for all } k \in (n_{\ell-1},n_{\ell}] \}\\
  B_{\ell} & = B_{\ell-1}\cap\{ h\in  \mathcal{T}_n : S_{k}(h) \leq  m(k) + U_y(k) \text{ for all }k \in (n_{\ell-1},n_{\ell}] \} \\
  C_{\ell} & = C_{\ell-1}\cap\{ h\in \mathcal{T}_n :  S_{k}(h)> m(k)+L_y(k) \text{ for all }k \in (n_{\ell-1},n_{\ell}] \}\\
  D_{\ell} & = D_{\ell-1}\cap\{ h \in \mathcal{T}_n : |(\zeta_{\tau} e^{-S_k})(h)| \leq c_{\ell} |(\zeta_{\tau} \mathcal{M}_{-1} \ldots \mathcal{M}_{\ell - 1} \mathcal{M}_{\ell - 1}^{(k)})(h)| + e^{- 10^4 (n - n_{\ell - 1})} \\ & \qquad \qquad \qquad \qquad \qquad \text{ for all } k \in (n_{\ell-1},n_{\ell}]\}{,}
\end{align*}
where $c_{\ell} := \prod_{i = 0}^{\ell} (1 + e^{-n_{i - 1}})$,  and
where we set $A_{-1} = B_{-1} = C_{-1} = D_{-1} = [-2,2]$.
Define the ``good" sets 
$$
G_\ell=A_\ell\cap B_\ell\cap C_\ell \cap D_\ell, \quad \ell\geq -1,
$$
and the set of interest in Theorem \ref{thm: upper}
$$
H(y)= \Big \{ h\in \mathcal{T}_n : |\zeta_{\tau}(h)| > e^y\frac{ e^n}{n^{3/4}}  \Big \},
$$
where 
$
\zeta_{\tau}(h)  $
stands for $\zeta(\tfrac 12 +\ii \tau + \ii h)$ as before.
We will call the points $h \in \mathcal{T}_n$ belonging to $H(y)$ the ``high points''.
The subsets $A_{\ell}$ and $D_{\ell}$ will be needed as auxiliary steps towards the proof that high points are in $C_{\ell}$, and $C_{\ell}$ will be needed for the proof of $B_{\ell}$.

\subsection{Induction steps}
Theorem \ref{thm: upper} follows from three propositions.
The first one proves that most high points are in the good set $G_0$. 
This control for small primes up to $n_0$ is simple, because the barrier $U_y$ is quite high and the $p^{\ii \tau}$'s show strong decoupling (i.e ``quasi-random'' behavior) for primes small enough with respect to $T$.

\begin{prop} There exists $K>0$ such that for any $4000 \leq y \leq n$,  one has
\label{prop: step 0}
$$
\mathbb P(\exists h\in H(y) \cap G^\cc_0)\leq K e^{-2y}.
$$
\end{prop}

Second, the proposition below  gives a precise control of the large values of $(S_k(h))_{h\in[-2,2]}$ for all $k$ up to $n_\ell$. 
This proposition is the most involved part of the proof.

\begin{prop} \label{prop: step ell}
There exists $K>0$ such that for any $4000 \leq y \leq n$, 
and
$\ell \geq 0$ such that $
  \exp(10^6 (n - n_{\ell})^{10^5} e^{n_{\ell + 1}}) \leq \exp(\tfrac{1}{100} e^n)
  $,
one has
$$
\mathbb{P} \Big ( \exists h \in H(y)\cap G_{\ell}\Big)\leq  \frac{K ye^{-2y}}{\log_{\ell+1} n}+\mathbb{P} \Big ( \exists h \in H(y)\cap G_{\ell+1}\Big).
$$
\end{prop}

Finally, one has the following estimate for the remaining points of the set.
\begin{prop} \label{prop: step L}
There exists $K>0$ such that for any 
$4000 \leq y \leq n$, 
and
$\ell \geq 0$ such that $
  \exp(10^6 (n - n_{\ell})^{10^5} e^{n_{\ell + 1}}) \leq \exp(\tfrac{1}{100} e^n)
  $,
one has
$$
\mathbb{P} \Big ( \exists h \in H(y)\cap G_{\ell}\Big)\leq K  y e^{-2 y} e^{10^3 (n - n_{\ell})}.
$$
\end{prop}

Theorem \ref{thm: upper} can be proved assuming Propositions \ref{prop: step 0}, \ref{prop: step ell} and \ref{prop: step L} .
\begin{proof}[Proof of Theorem \ref{thm: upper}]
  Let $L$ be the largest index $\ell$ such that
  $$
  \exp(10^6 (n - n_{\ell})^{10^5} e^{n_{\ell + 1}}) \leq \exp \Big ( \frac{1}{100} e^{n} \Big ),
  $$
  so that in particular $n - n_{L} = \OO(1)$. We clearly have
  $$
  \mathbb{P}(\exists h \in H(y))  \leq \mathbb{P}(\exists h \in H(y) \cap G_0^\cc) + \mathbb{P} (\exists h \in H(y) \cap G_0).
  $$
  By Proposition \ref{prop: step 0} and iterating Proposition \ref{prop: step ell} up to $L$, the above is
  $$
  \leq K e^{-2y} + \sum_{1 \leq \ell \leq L} \frac{K y e^{-2y}}{\log_{\ell} n} + \mathbb{P}(\exists h \in H(y) \cap G_{L}) \ll y e^{-2y} + \mathbb{P}(\exists h \in H(y) \cap G_{L}),
  $$
  since the sum over $\ell$ is rapidly convergent. 
  Finally, Proposition \ref{prop: step L} implies
  $$
  \mathbb{P}(\exists h \in H(y) \cap G_L) \leq K e^{(n - n_{L})^{10^3}} y e^{-2y} \ll y e^{-2y},
  $$
  since $n - n_{L} = \OO(1)$. All the above steps together yield  $
  \mathbb{P}(\exists h \in H(y)) \ll y e^{-2y}$, as expected.
\end{proof}

We note that  to obtain
$\mathbb{P}(\max_{|h| \leq 1} |\zeta(\tfrac 12 + \ii\tau + \ii h)|>e^y(\log T)/(\log\log T)^{3/4})=\oo(1)$
for large $y$, the number of steps in the induction can be lower than $L$. For example if $y$ is of  order $\log_2n$ as in \eqref{eqn: harper bound}, iterating up to $\ell=3$ suffices. Further iterations improves the error by extra logarithms.

\section{Initial Step}\label{sec:first}

This section proves  Proposition \ref{prop: step 0}.
Notice that by a union bound
\begin{align*}
\mathbb{P}(\exists h \in H(y) \cap G_0^{\cc})  \leq \mathbb{P}(\exists h \in A_0^{\cc})  + \mathbb{P}(\exists h \in D_0^{\cc} \cap A_0)  + \mathbb{P}(\exists h \in C_0^{\cc})  + \mathbb{P}(\exists h \in B_0^{\cc}). 
\end{align*}
The first two probabilities on the right-hand side will be bounded by $\ll e^{-7n}$, and the last two by $\ll e^{-2y}$. This will imply the claim. 

For the first probability, a union bound on $h$ and $k\leq n_0$ together with the Gaussian tail \eqref{eqn: gaussian tail cplx} yield
$$
\mathbb{P}(\exists h \in A_0^{\cc}) \ll e^n\,  n_0 \exp(- 10^2 n)\ll e^{-7n}.
$$

We now show that $\mathbb P(\exists h\in B_0^\cc) \ll e^{-2y}$. 
  A union bound on $y / 4 < k\leq n_0$ implies that for any sequence of integers $q_k \geq 1$,
\begin{equation}
\label{eqn: B0}
\begin{aligned}
\mathbb P(\exists h\in B_0^\cc)
&\leq \sum_{y / 4 < k\leq n_0} \mathbb P\Big(\max_{|h|\leq 2} S_k(h)>U_y(k)+k-\frac{3}{4}\log k\Big)\\
&\leq \sum_{y / 4 < k \leq n_0} \mathbb{E} \Big [ \max_{|h| \leq 2} \ \frac{|S_k(h)|^{2q_{k}}}{(y+k+10\log k)^{2q_{k}}} \Big ] ,
\end{aligned}
\end{equation}
where we use the fact that $m(k)\geq k-\frac{3}{4}\log k$ for $k>e$.
We discretize the maximum over $q_k e^k$ points using Lemma \ref{le:discretization} in Appendix \ref{se:discretization} with $N=\exp(2 q_k e^k)$ and $A=1000$. 
We can also apply \eqref{eqn:mom3} on each of these terms, taking $q_k=\lceil (y+k+10\log k)^2/(k+C) \rceil$ with $C > 0$ an absolute constant.
It is easily checked that the condition $2q_k\leq e^{n-k}$ is fulfilled here to get a Gaussian tail, as $y\leq n$ and $k\leq n_0$.

Note  that the second sum on the right-hand side of \eqref{eqn: discretize} is negligible. To see this, all terms up to $\frac{2\pi j}{8e^k}=T/2$ yield the same moment, as the average over $\tau$ could be replaced by an average over $[T/2, 2T]$ which yields the same bounds. The prefactor $1/(1+j^{1000})$ then makes the contribution negligible. For larger $j$'s, that is $j>\frac{2}{\pi}\frac{T}{\sqrt{\log T}}$, we use the deterministic bound $|S_k|^{2q_k}\leq \exp(q_k\cdot (\log T)^{1/2})$, so that the corresponding sum is at most $\sum_{|j|>\frac{T}{\sqrt{\log T}}}|j|^{-1000}\exp(q_k\cdot (\log T)^{1/2})\ll T^{-10}e^{(y+n)^2\sqrt{\log T}}\ll e^{-2y}$ for $y<n$.

Putting this together yields
\begin{align*}
\mathbb P (\exists h\in B_0^\cc) & \ll \sum_{y / 4 < k\leq n_0} e^k\, \frac{(k+y)^3}{k^{3/2}}\,\exp\big(-(k+10\log k+y)^2/(k+C)\big)
\\ & \ll  e^{-2y}\, \sum_{y / 4 < k\leq n_0} (k^{3/2} + y^3 k^{-3/2}) \, k^{-20}  \ll e^{-2y}.
\end{align*}

To bound the probability $\mathbb{P}(\exists h \in C_0^{\cc})$ 
we note that if there exists $h$ in $C_{0}^{\cc}$ then  $S_{k}(h) \leq y-20 k$ for some $h \in \mathcal{T}_n$ and some $y / 4 < k \leq n_0$. Therefore we obtain the bound,
\begin{align} \nonumber
\mathbb{P}(\exists h \in C_0^{\cc}) & \leq \sum_{y / 4 < k \leq n_0} \mathbb{P} \Big ( \max_{|h| \leq 2} |S_k(h)| > 20 k - y \Big ) \\ \label{eq:aa} & \leq \sum_{y / 4 \leq k \leq n_0} \mathbb{E} \Big [ \max_{|h| \leq 2} \frac{|S_k(h)|^{2q_k}}{(20 k - y)^{2 q_k}} \Big ]
\end{align}
for any choice of $q_k \geq 1$. We choose $q_k = \lceil (20 k - y)^2 / k \rceil$.
The length of $S_k(h)^{q_k}$ is $\exp(2 q_k e^k)$. We discretize the maximum over $q_k e^k$ points using Lemma \ref{le:discretization} in Appendix \ref{se:discretization} with $N=\exp(2 q_k e^k)$ and $A=1000$. This shows that \eqref{eq:aa} is
$$
\ll \sum_{y / 4 \leq k \leq n_0} q_k \, e^k \ \mathbb{E} \Big [ \frac{|S_k(0)|^{2q_k}}{(20 k - y)^{2q_k}} \Big ].
$$
By Equation (\ref{eqn:mom3}) from Lemma \ref{lem: Gaussian moments} in Appendix \ref{se:moments}, and the bound $q_k \ll k$ valid in the range $y / 4 \leq k \leq n_0$, this is
$$
\ll \sum_{y / 4 \leq k \leq n_0} k^{3/2} e^k \, \exp \Big ( - \frac{(20 k - y)^2}{k+C} \Big ) \leq \sum_{y / 4 \leq k \leq n_0} k^{3/2} e^k \, e^{-400 k + 20 y} \leq e^{-2y},
$$
with $C$ an absolute constant.
This is the expected result.

Finally, we show that $\mathbb{P}(\exists h \in D_0^{\cc} \cap A_0) \ll e^{-100 n}$. Suppose that we are placed on a $\tau$ for which for all $h \in \mathcal{T}_n$ we have
\begin{equation}\label{eq:trivbound}
|\zeta_{\tau}(h)| \leq e^{100 n}.
\end{equation}
Then for all $h \in A_0$ we have by Lemma \ref{le:moliapprox} in Appendix \ref{se:moments} 
$$
|e^{-S_k(h)}| \leq (1 + e^{-n_{-1}}) |\mathcal{M}_{-1}^{(k)}(h)| + e^{-10^5 (n_0 - n_{-1})}.
$$
It follows that for such $\tau$'s we have for all $h \in A_0$, 
\begin{align*}
|(\zeta_{\tau} e^{-S_k})(h)| & \leq (1 + e^{-n_{-1}}) \, |(\zeta_{\tau} \mathcal{M}_{-1}^{(k)})(h)| + e^{100 n - 10^5 (n_0 - n_{-1})} \\ & \leq (1 + e^{-n_{-1}}) \, |(\zeta_{\tau} \mathcal{M}_{-1}^{(k)})(h)| + e^{ - 10^4 (n - n_{-1})},
\end{align*}
as claimed. Therefore, we are left with the elementary bound
$$
\mathbb{P}(\exists h \in D_0^{\cc} \cap A_0) \leq \mathbb{P}(\exists h : |\zeta_{\tau}(h)| \geq e^{100 n}) \leq \sum_{h \in \mathcal{T}_n} \mathbb{E} \Big [ \frac{|\zeta_{\tau}(h)|^2}{e^{200 n}}\Big ] \ll e^{-100 n},
  $$
  by the second moment bound for the zeta function (Lemma \ref{le:secondmoment}, Appendix \ref{se:moments}).

\section{Induction}\label{sec:induction}

We now prove Proposition \ref{prop: step ell}.
The subsets $A$, $B$, $C$ and $D$'s need to be refined to account for the intermediate increments in the interval $(n_{\ell}, n_{\ell+1}]$:
For $k\in [n_{\ell}, n_{\ell+1}]$, define
\begin{align*}
  A^{(k)}_{\ell} & = A_{\ell}\cap\{ h \in \mathcal{T}_n:  |\widetilde{S}_j(h)-\widetilde{S}_{n_\ell}(h)| \leq 10^3 (n - n_{\ell}) \text{ for all } n_{\ell}<j\leq k \},\\
  B^{(k)}_{\ell} & = B_{\ell}\cap\{ h\in \mathcal{T}_n : S_{j}(h)\leq m(j) + U_y(j) \text{ for all } n_{\ell}<j\leq k \}, \\
  C^{(k)}_{\ell} & = C_{\ell}\cap\{ h\in \mathcal{T}_n :  S_{j}(h)> m(j)+L_y(j) \text{ for all } n_{\ell}<j\leq k \},\\
  D^{(k)}_{\ell} & = D_{\ell}\cap\{ h \in \mathcal{T}_n : |(\zeta_{\tau} e^{-S_{k}})(h)| \leq c_{\ell + 1} |(\zeta_{\tau} \mathcal{M}_{-1} \ldots \mathcal{M}_{\ell} \mathcal{M}_{\ell}^{(k)})(h)| + e^{ - 10^4 (n - n_{\ell})} \\ & \qquad \qquad \qquad \qquad \qquad \text{ for all } n_{\ell}<j\leq k\},
\end{align*}
where $c_{\ell + 1} := \prod_{i = 0}^{\ell + 1} (1 +  e^{-n_{i - 1}})$. 
Note that with this notation $A^{(n_{\ell+1})}_\ell=A_{\ell+1}$.
We also take as a convention that $A_{\ell}^{(n_{\ell})} = A_{\ell}$.
The same holds for $B_{\ell}^{(k)}$, $C_{\ell}^{(k)}$ and $D_{\ell}^{(k)}$. 

The proof of Proposition \ref{prop: step ell} is based on the following two lemmas. We defer the proofs of these lemma to later sections. 

\begin{lemma} \label{master lemma}
  Let $\ell \geq 0$ be such that $\exp(10^6 (n - n_{\ell})^{10^5} e^{n_{\ell + 1}}) \leq \exp(\tfrac{1}{100} e^n)$. Let $k \in (n_{\ell}, n_{\ell + 1}]$. 
  Let $\mathcal{Q}$ be a Dirichlet polynomial of length $N \leq \exp(\tfrac{1}{100} e^n)$. Suppose that $\mathcal{Q}$ is supported on integers all of whose prime
  factors are $> \exp(e^{k})$.
  Then, for $4000 \leq y \leq n$ and $L_y(k)<w - m(k) < U_y(k)$, one has
  \begin{align*}
 \mathbb{E} & \Big [ \max_{|h|\leq 2}|\mathcal{Q}(\tfrac 12 + \ii \tau+\ii h)|^2 \cdot \mathbf{1} \Big ( h \in B_{\ell}^{(k)} \cap C_{\ell}^{(k)} \text{ and }  S_{k}(h) \in (w,w+1] \Big ) \Big ]\\  & \ll
    \mathbb{E} \Big [ |\mathcal{Q}(\tfrac 12 +\ii \tau)|^2 \Big ] \,  \Big ( e^{-k} \log N + (n-k)^{800} \Big ) \, y\, (U_y(k)-w+m(k)+2)\, e^{- 2 (w - m(k))},
  \end{align*}
where the  implicit constant in $\ll$ is absolute and in particular independent of $\ell$ and $k$.
 \end{lemma}

\begin{lemma} \label{master lemma twisted}
  Let $\ell \geq 0$ with $\exp(10^6 (n - n_{\ell})^{10^5} e^{n_{\ell + 1}})) \leq \exp(\tfrac{1}{100} e^n)$. Let $k \in [n_{\ell} , n_{\ell + 1}]$. 
  Let $\gamma(m)$ be a sequence of complex coefficients with $|\gamma(m)| \ll \exp(\tfrac{1}{1000} e^n)$ for all $m \geq 1$. 
  Let
  \begin{equation} \label{eq:qle}
\mathcal{Q}_{\ell}^{(k)}(h) := \sum_{\substack{p | m \implies p \in (T_{\ell}, \exp(e^k)] \\ \Omega_{\ell + 1}(m) \leq (n_{\ell + 1} - n_{\ell})^{10^4}}} \frac{\gamma(m)}{m^{\tfrac 12 +\ii \tau + \ii h}}.
  \end{equation}
 Then, for any $h \in [-2,2]$, $4000 \leq y \leq n$ and $L_y(n_{\ell}) < u - m(n_{\ell}) \leq U_{y}(n_{\ell})$,
  \begin{multline*}
  \mathbb{E} \Big [ |(\zeta_{\tau} \mathcal{M}_{-1} \ldots \mathcal{M}_{\ell} \mathcal{M}^{(k)}_{\ell})(h)|^4
  \cdot |\mathcal{Q}_{\ell}^{(k)}(h)|^2 \cdot \mathbf{1} \Big ( h \in B_{\ell} \cap C_{\ell} \text{ and } S_{n_{\ell}}(h) \in (u, u + 1] \Big ) \Big ]\\
  \ll e^{4(n - k)} \, \mathbb{E} \Big [ |\mathcal{Q}^{(k)}_{\ell}(h)|^2 \Big ]  \, e^{-n_{\ell}}\,  y \, (U_y(n_{\ell}) - u + m(n_{\ell}) + 2)\, e^{-2 (u - m(n_{\ell}))},
  \end{multline*}
where the  implicit constant in $\ll$ is absolute and in particular independent of $\ell$ and $k$.
\end{lemma}
Note that we allow $k = n_{\ell}$ in which case $\mathcal{Q}^{(k)}= 1$.
Some explanations on the heuristics of Lemmas \ref{master lemma} and  \ref{master lemma twisted} might be in order.
First, one expects the partial sums $S_k(h)$ to be approximately Gaussian. In fact, one can see $S_k(h)$ for a fixed $h$ as a Gaussian random walk of mean $0$ and variance $1/2$ for each of its increment. For such a random walk, the endpoint $S_k$ is independent of the ``bridge" $S_j-\frac{j}{k}S_k$ for all $j\leq k$. Since $S_k\approx m(k)$, the latter is approximately $S_j-m(j)$. With this in mind, the indicator function can be thought of as the restriction of the endpoint $S_k$ being in $w$ and that the walk $S_j-m(j)$ starting at $0$ and ending at $w-m(k)$ stays below the barrier $y+U_y(k)$. Using the ballot theorem, Proposition \ref{lem: ballot} from Appendix \ref{se:ballot}, the probability of this happening for a fixed $h$ is 
$$
\frac{y(U_y(k)-w+m(k))}{k^{3/2}}e^{-\frac{w^2}{k}}
\ll y (U_y(k)-w+m(k)) e^{-k}e^{-2(w-m(k))}.
$$
Since $S_k( h)$ has length $\exp(e^k)$ as a Dirichlet polynomial, one expects that there are approximately $e^k$ independent random walks as $h$ varies in $[-2,2]$. Moreover, the Dirichlet polynomial $\mathcal Q$ is supported on primes larger than $\exp(e^k)$, so its value should be independent of the $e^{k}$ walks, as they are ``supported" on different primes. Also, due to the greatest frequency $\log N$ in the summands of $\mathcal{Q}$,  there should be $\log N$ independent values when discretizing the maximum.  The factor $(n-k)^{800}$ comes from the process of approximating the indicator function by a Dirichlet polynomial. These factors together reproduce the result of Lemma \ref{master lemma}. The heuristics for Lemma \ref{master lemma twisted} is the same with the extra fourth moment. Again, one expects $\log \zeta_\tau(h)-S_{k}(h)$ to be independent of $\mathcal Q_\ell^{(k)}$ and $S_{n_\ell}$. Therefore, the expectation of the fourth moment could formally be factored out. The variable $\log \zeta_\tau(h)-S_{k}(h)$ should be approximately  Gaussian with variance $n-k$. Therefore, $\E[e^{4(\log \zeta_\tau(h)-S_{k}(h))}]\approx e^{4(n-k)}$. The mollifiers $\mathcal M$ are designed to approximate $e^{-S_{n_\ell}}$.

We are now ready to begin the proof of Proposition \ref{prop: step ell}. 
Notice that by a union bound,
$$
\mathbb{P}(\exists h \in H(y) \cap G_{\ell}) \leq \mathbb{P}(\exists h \in H(y) \cap G_{\ell} \cap G_{\ell + 1}^{\cc}) + \mathbb{P}(\exists h \in H(y) \cap G_{\ell + 1}).
$$
The first term can be further split by another union bound,
\begin{align*}
  \mathbb{P}(\exists h \in H(y) \cap G_{\ell} \cap G_{\ell + 1}^{\cc}) & \leq  \mathbb{P}( \exists h \in A_{\ell + 1}^{\cc} \cap H(y) \cap G_{\ell}) \\ &   + \mathbb{P}(\exists h \in D_{\ell + 1}^{\cc} \cap A_{\ell + 1} \cap H(y) \cap G_{\ell}) \\ & + \mathbb{P}(\exists h \in C_{\ell + 1}^{\cc} \cap D_{\ell + 1} \cap A_{\ell + 1} \cap H(y) \cap G_{\ell}) \\ & + \mathbb{P}(\exists h \in B_{\ell + 1}^{\cc} \cap C_{\ell + 1} \cap A_{\ell + 1} \cap H(y) \cap G_{\ell}). 
\end{align*}

It will be shown that each of the above probabilities is bounded by
$$
\ll \frac{y e^{-2y}}{(\log_{\ell + 1} n)^{100}}.
$$
This will conclude the proof. The proof of each bound is broken down into a separate subsection. 
The estimate with $B_{\ell+1}^c$ is the tightest.
We will sometimes drop some events that are not needed to achieve the bound.

\subsection{Bound on increments} We first consider $A_{\ell + 1}^{\cc}$.
This is the simplest bound.
We show by a Markov-type inequality that
$$
\mathbb{P}(\exists h \in A_{\ell + 1}^{\cc} \cap G_{\ell} ) \ll y e^{-2y} (\log_{\ell - 1} n)^{-1}.
$$
(Recall our convention that $\log_{-1} n = e^n$ and $\log_{0} n = n$.) 
If there is a $k\in (n_{\ell}, n_{\ell+1}]$ and an $h$ such that $|\widetilde{S}_k(h) - \widetilde{S}_{n_\ell}(h)| > 10^3 (n - n_{\ell})$, then one has that
  $$
 \sum_{k\in (n_{\ell},n_{\ell+1}]} \max_{|h|\leq 2}\frac{|\widetilde{S}_k(h) - \widetilde{S}_{n_\ell}(h)|^{2q}}{(10^3 (n - n_{\ell}))^{2q}} \geq 1,  \text{    for all $q\geq 1$. }
    $$
  Therefore,  for any choice of $q \geq 1$, the following bound holds
 \begin{align} \label{eqn:random1}
   \mathbb{P} & (\exists h \in A_{\ell+1}^{\cc} \cap G_{\ell}) \leq \sum_{k \in (n_{\ell}, n_{\ell + 1}]} \mathbb{E} \Big [ \max_{|h| \leq 2} \frac{|(\widetilde{S}_{k} - \widetilde{S}_{n_\ell})(h)|^{2q}}{(10^3 (n - n_{\ell}))^{2q}} \cdot \mathbf{1} \Big (h \in G_{\ell} \Big ) \Big ]  .
   \end{align}
We pick $q=\lfloor10^6(n-n_\ell)^2/(k-n_\ell)\rfloor$. The Dirichlet polynomial $(\widetilde{S}_{k} - \widetilde{S}_{n_\ell})^q$ is then of length at most $\exp(2 q e^k)\ll \exp (2 \cdot 10^{6}(n - n_{\ell})^2 e^{n_{\ell+1}})$, which is much smaller than $\exp(e^n/100)$ by the definition of $n_\ell$. 
Lemma \ref{master lemma} thus bounds the right-hand side of (\ref{eqn:random1}) with
$$
 y e^{-2y}\, \sum_{k \in (n_{\ell}, n_{\ell + 1}]} (q + (n - n_{\ell})^{800}) \, e^{100 (n - n_{\ell})} \, \mathbb{E} \Big [ \frac{|\widetilde{S}_k - \widetilde{S}_{n_\ell}|^{2q}}{(10^3 (n - n_{\ell}))^{2q}} \Big ].
$$
For our choice of $q$, we have $2q\ll (n-n_\ell)^2 \leq e^{n-k}$, so that  the estimate in Lemma \ref{lem: Gaussian moments cplx} applies. Together with
Stirling's approximation as in \eqref{eqn: gaussian tail cplx}
we conclude that the above is
     $$
     \ll y e^{-2y} e^{- (n - n_{\ell})} \ll y e^{-2y} (\log_{\ell - 1} n)^{-1}.
     $$

\subsection{Bound with mollifiers} We now estimate $D_{\ell+1}^\cc$.
In this section, we obtain 
\begin{equation}\label{eqn:aim1}
\mathbb{P}(\exists h \in D_{\ell + 1}^{\rm c} \cap A_{\ell + 1} \cap G_\ell) \ll y e^{-2y} (\log_{\ell - 1} n)^{-1}.
\end{equation}
For $h$ in $A_{\ell+1}\cap D_\ell$, we have
\begin{align}
|(\widetilde{S}_k-\widetilde{S}_{n_\ell})(h)|&<10^3(n_{\ell + 1} -n_{\ell})\label{chain1},\\
|(\zeta_{\tau} e^{-S_{n_\ell}})(h)|&< c_{\ell} |(\zeta_{\tau} \mathcal{M}_{-1}\dots\mathcal{M}_\ell)(h)|+e^{-10^4(n-n_{\ell - 1})}\label{chain2},
\end{align}
where $c_{\ell} = \prod_{i = 0}^{\ell} (1 + { e^{-n_{i - 1}}})${.}
If we additionally assume that, for all $h \in A_{\ell + 1} \cap D_{\ell}$, both
\begin{align}
|(\zeta_{\tau} \mathcal{M}_{-1}\dots\mathcal{M}_\ell)(h)|&<e^{ 10^3(n-n_\ell)}\label{chain3}\\
|(e^{-(S_k-S_{n_\ell})})(h)| & \leq (1 + e^{-n_{\ell}}) \, |\mathcal{M}_{\ell}^{(k)}(h)| + e^{-10^5(n_{\ell + 1}-n_{\ell})}\label{chain4},
\end{align}
hold for all $k \in (n_{\ell}, n_{\ell + 1}]$, then
we obtain (where each of the expression below is evaluated at $h$),
\begin{align*}
|\zeta_{\tau} e^{-S_k}|&=|\zeta_{\tau} e^{-S_{n_\ell}}|\, e^{-(S_k-S_{n_\ell})}\\
&{<}\Big(c_{\ell} |\zeta_{\tau} \mathcal{M}_{-1}\ldots\mathcal{M}_\ell|+e^{-10^4(n-n_{\ell - 1})}\Big)e^{-(S_k-S_{n_\ell})}\\
&\leq
c_{\ell} |\zeta_{\tau} \mathcal{M}_{-1}\dots\mathcal{M}_\ell|e^{-(S_k-S_{n_\ell})}+e^{-10^3(n-n_{\ell - 1})}\\
&\leq c_{\ell + 1} |\zeta_{\tau} \mathcal{M}_{-1}\dots\mathcal{M}_\ell \mathcal{M}_{\ell}^{(k)}| + c_{\ell} |\zeta_{\tau} \mathcal{M}_{-1} \ldots \mathcal{M}_{\ell} | \, e^{-10^5(n_{\ell + 1}-n_{\ell})}+e^{-10^3(n-n_{\ell - 1})}\\
&\leq c_{\ell + 1} |\zeta_{\tau} \mathcal{M}_{-1}\dots\mathcal{M}_\ell\mathcal{M}_{\ell}^{(k)}|+e^{-10^4(n-n_{ \ell})}.
\end{align*}
Here, we have successively used the estimates (\ref{chain2}), (\ref{chain1}), (\ref{chain4}), (\ref{chain3}), and the fact that the sequence $c_\ell$, $\ell>-1$, is rapidly convergent.
It remains to verify that the bounds \eqref{chain3} and \eqref{chain4} hold with high probability for $h \in A_{\ell + 1} \cap D_{\ell}$. The bound \eqref{chain4} holds pointwise for all $h \in A_{\ell + 1}$ by Lemma \ref{le:moliapprox} in Appendix \ref{se:moments}.
As for Equation \eqref{chain3}, the probability of the complement of the event is
\begin{align}
\label{eqn: D estimate}
\sum_{h\in\mathcal{T}_n} & \mathbb{P}\Big(|(\zeta_{\tau} \mathcal{M}_{-1}\dots\mathcal{M}_\ell)(h)| \geq e^{10^3(n-n_\ell)},h\in G_\ell\Big)\\ &\ll e^{-4 \cdot 10^3 (n - n_{\ell})} \, e^n \, \mathbb{E} \Big [ |(\zeta_{\tau} \mathcal{M}_{-1} \ldots \mathcal{M}_{\ell})(0)|^4 \cdot \mathbf{1}(0 \in G_{\ell}) \Big ].
\end{align}
Lemma \ref{master lemma twisted}  applied for $\mathcal Q\equiv 1$ and $k=n_\ell$ then implies the expected bound,
$$
\ll y e^{-2 y - 4 \cdot 10^3 (n - n_{\ell})} \,e^{100 (n - n_{\ell})} \ll y e^{-2y} (\log_{\ell -1} n)^{-1}.
$$
Note that  \eqref{eqn: D estimate} can be made small, because the union bound on the random variables $\log |(\zeta_{\tau} \mathcal{M}_{-1}\dots\mathcal{M}_\ell)(h)|$ (which are approximately Gaussian of variance $n-n_\ell$) is effectively on the $h$'s in $G_\ell(0)$. The number of such $h$'s is small enough, of order $e^{n-n_\ell}$.

\subsection{Extension of the lower barrier}
We now want to prove the following bound on $C_{\ell+1}^\cc$:
\begin{equation}\label{eqn:needed}
\mathbb{P}(\exists h\in H(y) \cap C_{\ell + 1}^{\rm c}\cap D_{\ell + 1} \cap A_{\ell + 1} \cap G_\ell) \ll y e^{-2y} (\log_{\ell} n)^{-1}.
\end{equation}
Here, we explicitly make use of the fact that $\zeta_\tau$ is large. 
Let $h \in C_{\ell + 1}^{\cc} \cap D_{\ell + 1} \cap G_{\ell} \cap H(y)$.
By definition of $C_{\ell + 1}^{\rm c}$, there must be a $k$ such that $S_k(h)\leq  m(k) - 20 (n - k) + y$. 
We split $S_k(h)$ according to the value of $S_{n_{\ell}}(h) \in [u, u + 1]$ and $(S_{k} - S_{n_{\ell}})(h) \in [v, v + 1]$,
where $u, v \in \mathbb{Z}$, $|v| \leq 10^3 (n - n_{\ell})$ and $u + v \leq  m(k) - 20 (n - k) + y$.
We notice that since $h \in H(y)$, 
\begin{equation}
\label{eqn: Markov D}
|(\zeta_\tau \, e^{-S_{k}})(h)| > V e^{-u - v},
\end{equation}
where $V = e^y e^n n^{-3/4}$. Since $h \in D_{\ell + 1}$ also, we either have
$$
|(\zeta_\tau \mathcal{M}_{-1} \dots \mathcal{M}_{\ell} \mathcal{M}_{\ell}^{(k)} )(h)| \gg V e^{-u - v}
$$
or $\tfrac 12 V e^{- u - v} \leq e^{-10^4 (n - n_{\ell})}$. However, the second possibility cannot occur since it implies that $e^{u + v} > e^y e^n e^{10^4 (n - n_{\ell})} n^{-3/4}$ and hence $e^u > e^y e^n e^{10^3 (n - n_{\ell})} n^{-3/4}$. This means that $S_{n_{\ell}}(h)$ is above the barrier, and this is impossible because $h \in G_{\ell}$. 

Therefore, with a union bound and \eqref{eqn: Markov D}, the left-hand side of \eqref{eqn:needed} is bounded for any $q \geq 1$ by
\begin{align} \label{eq:asdas}
\sum_{\substack{k \in (n_{\ell}, n_{\ell + 1}]\\ h\in\mathcal{T}_n}}\sum_{\substack{u + v \leq m(k) - 20 (n - k) + y\\ |v| \leq 10^3 (n_{\ell + 1} - n_{\ell}) \\  L_y(n_{\ell}) \leq u - m(n_{\ell}) \leq U_y(n_{\ell}) }} \frac{e^{4 u + 4 v}}{V^4} \cdot \mathbb{E} \Big [  |(\zeta_{\tau} \mathcal{M}_{-1} \ldots \mathcal{M}_{\ell} \mathcal{M}_{\ell}^{(k)})(h) |^4 \cdot \frac{|(S_{k} - S_{n_{\ell}})(h)|^{2q}}{(1 + v^{2q})}\\
 \nonumber \times \mathbf{1} \Big ( S_{n_{\ell}}(h) \in [u, u + 1] \text{ and } h \in A_{\ell} \cap B_{\ell} \cap C_{\ell} \Big ) \Big ].
\end{align}
Pick $q = \lfloor v^2 / (k - n_{\ell}) \rfloor$. Since $q \leq 10^7 (n - n_{\ell})^2$, the Dirichlet polynomial $(S_k - S_{n_{\ell}})^{q}$ can be written in the form \eqref{eq:qle}. In particular, Lemma \ref{master lemma twisted} with $\mathcal Q=(S_k-S_{n_\ell})^q$ is applicable. Lemma \ref{lem: Gaussian moments} and Stirling's approximation also imply
$$
\mathbb{E} \Big [ \frac{|(S_{k} - S_{n_{\ell}})(h)|^{2q}}{(1 + v^{2q})} \Big ] \ll e^{-q} \ll \exp \Big ( - \frac{v^2}{k - n_{\ell}} \Big ).
$$
Therefore, Lemma \ref{master lemma twisted} and the above computation show that \eqref{eq:asdas} is 
$$
\ll e^n \sum_{\substack{ k \in (n_{\ell}, n_{\ell + 1}] \\ u + v \leq m(k) - 20 (n - k) + y \\ |v| \leq 10^3 (n - n_{\ell}) \\ L_y(n_{\ell}) \leq u - m(n_{\ell}) \leq U_y(n_{\ell})}} \frac{e^{4 u +  4 v}}{V^4}  e^{4(n-k)}e^{- \frac{v^2}{k - n_{\ell}}}\   \frac{y  (U_y(n_{\ell}) - u + m(n_{\ell}) + 2))}{e^{n_{\ell}}} \ e^{- 2(u - m(n_{\ell}))}.
$$
  We use the restriction $u-m(n_{\ell}) \in [L_y(n_{\ell}) ,  U_y(n_{\ell})]$ to bound
  $$
  0 \leq U_{y}(n_{\ell}) - u + m(n_{\ell}) \leq U_y(n_{\ell}) - L_y(n_{\ell}) \ll (n - n_{\ell}) \ll \log_{\ell} n \text{ for all $y$.}
  $$
Subsequently we remove this restriction on $u$. After replacing $V$ by $e^y e^n n^{-3/4}$, the above sum is thus bounded by
\begin{align*}
y e^{-4y} \sum_{k\in (n_{\ell},n_{\ell+1}]}e^{n - 4k- n_{\ell}} \, n^3 \sum_{\substack{u + v \leq m(k) - 20 (n - k) + y \\ |v| \leq 10^3 (n - n_{\ell})  \\ u, v \in \mathbb{Z}}} & e^{2 u + 2 m(n_{\ell}) + 4v} \, (\log_{\ell} n)  \, \exp \Big ( - \frac{v^2}{k- n_{\ell}} \Big ).
\end{align*}
Performing the summation over $u$, we get
$$
\ll y e^{-4y}\sum_{k\in (n_{\ell},n_{\ell+1}]}e^{n - 4k- n_{\ell}} \, n^3  \sum_{\substack{|v| \leq 10^3 (n - n_{\ell}) \\ v \in \mathbb{Z}}} e^{2 m (k) + 2 m(n_{\ell}) - 40 (n -k) + 2 v + 2y} \, (\log_{\ell} n) \, \exp \Big ( - \frac{v^2}{k - n_{\ell}} \Big ) .
$$
The sum over $v$ can then be performed and yields the bound
\begin{equation}
\label{eqn: bound C}
\begin{aligned}
& y e^{-2y}\sum_{k\in (n_{\ell},n_{\ell+1}]}e^{n - 4k-n_{\ell}} \, n^3  \, e^{2 m(k) + 2 m(n_{\ell}) + (k - n_{\ell})}\, (\log_{\ell} n)  \, e^{- 40 (n - k)} (k-n_\ell)^{1/2} \\ & \ll y e^{-2y} \sum_{k\in (n_{\ell},n_{\ell+1}]}(\log_{\ell} n)^{3/2} e^{-  9 ( n - k)} \ll y e^{-2y} (\log_{\ell} n)^{-1} ,
\end{aligned}
\end{equation}
since $n-k\geq n - n_{\ell + 1}= 10^6 \log_{\ell + 1} n$.
Notice that in the case $\ell = 0$ we use the fact that we save a large power of $n$ in $e^{- (n - k)}$ to offset the term $n^3$, whereas in the case $\ell \geq 1$ we use the fact that $e^{4m(k)} n^{3} \asymp e^{4k}$ for $k \in (n_{\ell}, n_{\ell + 1}]$.

\subsection{Extension of the upper barrier} We need the following bound on $B_{\ell+1}^\cc$:
$$
\mathbb{P}(\exists h\in H(y) \cap B_{\ell + 1}^{\rm c} \cap A_{\ell + 1}\cap C_{\ell+1} \cap G_\ell) \ll \frac{ye^{-2y}}{(\log_{\ell+1}n)^{100}}.
$$
In fact, we show the stronger estimate 
\begin{equation}
\label{eqn: proof B}
\mathbb{P}(\exists h\in (B_\ell\setminus B_{\ell + 1}) \cap C_{\ell+1})\ll  \frac{ye^{-2y}}{(\log_{\ell+1}n)^{100}}.
\end{equation}
We write $\overline S_j=S_j -m(j)$ for simplicity.

By considering a union bound on $k\in [n_\ell, n_{\ell+1})$ and by partitioning the values of $S_k(h)$ according to $S_k(h) \in [w, w + 1]$ with $w \in \mathbb{Z}$, the above reduces to 
$$
\begin{aligned}
&\ll \sum_{k \in [n_{\ell}, n_{\ell + 1})} \mathbb{P}(\exists h\in (B^{(k)}_\ell\setminus B^{(k+1)}_{\ell}) \cap C^{(k)}_{\ell})\\
&\ll \sum_{\substack{k \in [n_{\ell}, n_{\ell + 1}) \\ w\in [L_y(k),U_y(k))}} \mathbb{P}(\exists h:  \overline S_j(h) < U_y(j) \ \forall j\leq k,  \overline S_{k+1}(h)>U_y(k+1),    \overline S_k(h)\in (w,w+1]).
\end{aligned}
$$
Note that the condition $\overline S_{k+1}>U_y(k+1)$ under the restriction $\overline S_k(h)\in (w,w+1]$ can be rewritten as
$$
 S_{k+1}- S_{k}>U_y(k+1)+m(k+1)-m(k)-\overline S_k>U_y(k+1)-w+\oo(\log n/n).
$$
Write $V_{w,k}=U_y(k+1)- w$. By Markov's inequality, the above sum is bounded by
$$
\ll \sum_{\substack{k \in [n_{\ell}, n_{\ell + 1}) \\ w\in [L_y(k),U_y(k))}}\mathbb E\Big[\max_{|h|\leq 2} \frac{|(S_{k+1}-S_k + 1)(h)|^{2q}}{(V_{w,k}+1)^{2q}} \mathbf 1 \Big (\overline S_j(h)<U_y(j) \ \forall j\leq k,  \overline S_k(h)\in (w,w + 1] \Big ) \Big].
$$
We pick $q=(V_{w,k} + 1)^2 / 10 = (U_y(k+1)- w + 1)^2 / 10 \leq 400 (n-k)^2$, by the bounds on $w$. For this choice, note that the Dirichlet $(S_{k+1}-S_k + 1)^{q}$ has length at most $\exp(2q e^{k+1})\leq \exp(1000 (n-k)^2 e^{k+1})$. In particular, Lemma \ref{master lemma} can be applied (note that the Dirichlet polynomial $S_{k + 1} - S_k + 1$ is supported on integers all of whose prime factors are $> \exp(e^k)$ since $1$ is not a prime!). 
This yields the bound
$$
\ll \sum_{k \in [n_{\ell}, n_{\ell + 1})} (n-k)^{800} \sum_{w\in [L_y(k),U_y(k))} \frac{\mathbb E[|(S_{k+1}-S_k + 1)(0)|^{2q}]}{(V_{w,k}+1)^{2q}} \, y \, e^{-2 w} \, (U_y(k)-w+1).
$$
The expectation is $\ll (2q)! / q! + 4^q \ll 100^q (q / e)^{q}$ by Equation \eqref{eqn: moment+1} of Lemma \ref{lem: Gaussian moments} in Appendix \ref{se:moments}. We then find using Stirling's formula (similarly as in \eqref{eqn: gaussian tail}, but the optimal exponent is not needed here), that
$$
 \frac{\mathbb E[|(S_{k+1}-S_k + 1)(0)|^{2q}]}{(V_{w,k}+1)^{2q}}\ll e^{-(V_{w,k} + 1)^2 / 10}.
$$
Putting this back in the estimate gives the bound
$$
\begin{aligned}
\ll  y \sum_{k \in [n_{\ell}, n_{\ell + 1})} (n-k)^{800} e^{-2U_y(k)} \sum_{w\in [L_y(k),U_y(k))}  (U_y(k)-w+1)  e^{- \frac{1}{10}(U_y(k+1)- w + 1)^2   + 2(U_y(k + 1) - w + 1)} ,
\end{aligned}
$$
where we added $U_y(k + 1)$ and subtracted $U_y(k)$ which is allowed since $U_y(k + 1) - U_y(k) = \OO(1)$. 
Finally $-(U_y(k + 1) - w + 1)^2 / 10 + 2 (U_y(k + 1) - w + 1) = - (1/10)  (U_y(k + 1) - w - 9)^2 + 10$ so the last sum over $w$ is finite.  It remains to recall that $U_y(k)=y +10^3\log (n-k)$ to conclude that
$$
\mathbb{P}(\exists h\in (B_\ell\setminus B_{\ell + 1}) \cap C_{\ell+1})\ll ye^{-2y} \sum_{k\in [n_\ell, n_{\ell+1})} (n-k)^{800} e^{-10^3\log (n-k)}\ll \frac{ye^{-2y}}{(\log_{\ell+1}n)^{100}}.
$$

\section{Final Step}\label{sec:last}

This short section proves Proposition \ref{prop: step L}. We notice that if $h \in H(y) \cap G_{\ell}$, then $S_{n_{\ell}}(h) \in [v, v + 1]$ with $|v - y - m(n_{\ell})| \leq 20 (n - n_{\ell})$, and 
$
|(\zeta_{\tau} e^{-S_{n_{\ell}}})(h)| \geq V e^{-v}
$
where $V = e^{y} e^n / n^{3/4}$. We wish to apply Markov's inequality and Lemma \ref{master lemma twisted}. We first need to compare the expression to the one with mollifiers.
To this end, note that since $h \in D_{\ell}$ and $V e^{-v} > 2e^{-10^4 (n - n_{\ell-1})}$, we have
$
V e^{-v} \ll |(\zeta_{\tau} \mathcal{M}_{-1} \ldots \mathcal{M}_{\ell})(h)|.
$
Therefore, Markov's inequality implies that
\begin{align*}
\mathbb{P} & (\exists h \in H(y) \cap G_{\ell}) \\ & \ll \sum_{\substack{h \in \mathcal{T}_n \\ |v - y - m(n_{\ell})| \leq 20 (n - n_{\ell})}} \frac{e^{4v}}{V^4} \, \mathbb{E} \Big [ |(\zeta_{\tau} \mathcal{M}_{-1} \ldots \mathcal{M}_{\ell})(h)|^4 \cdot \mathbf{1} \Big ( S_{n_{\ell}}(h) \in [v, v + 1] \text{ and } h \in B_{\ell} \cap C_{\ell} \Big ) \Big ] .
\end{align*}
By Lemma \ref{master lemma twisted}, this is
\begin{align*}
\ll e^{-4y} e^{-4n} n^3 \, e^n \sum_{|v - y - m(n_{\ell})| \leq 20 (n - n_{\ell})} { e^{4(n-n_\ell)}}e^{2v} e^{2 m(n_{\ell})} \, y \, (n - n_{\ell}) e^{-n_{\ell}} \ll y e^{-2y} e^{100 (n - n_{\ell})}.
\end{align*}
The last inequality is obtained similarly as in Equation \eqref{eqn: bound C}: when $\ell = 0$ the term $n^3$ is included in $e^{100 (n - n_0)}$, while for $\ell \geq 1$ we have $e^{4m(n_{\ell})} n^{3} \ll e^{4n_{\ell}}$. This concludes the proof of \eqref{eqn: proof B}.

\section{Decoupling and Second Moment}\label{sec:2nd}

\subsection{Lemmas from harmonic analysis}
We will need the following lemmas from harmonic analysis.

\begin{lemma} \label{le:ingham}
  There exists a smooth function $F_0$ such that 
  \begin{enumerate}
  \item For all $x \in \mathbb{R}$, we have $0 \leq F_0(x) \leq 1$ and $\widehat{F}_0(x) \geq 0$. 
  \item $\widehat{F}_0$ is compactly supported on $[-1,1]$.
  \item Uniformly in $x \in \mathbb{R}$, we have
    $$
    F_0(x) \ll e^{-|x| / \log^2 (|x| + 10)}.
    $$
  \end{enumerate}
\end{lemma}
\begin{proof}
  This follows from the sufficient part of the main theorem of \cite{InghamFourier}. Note that this theorem does not state the positivity conditions on $F_0$ and $\hat F_0$ but these can be obtained from the explicit construction in \cite{InghamFourier}.
\end{proof}

The above lemma allows us to construct a convenient approximation to the indicator function of a small
interval $[0, \Delta^{-1}]$. 

\begin{lemma} \label{le:harmo}
  There exists an absolute constant $C > 0$ such that for any $\Delta, A \geq 3 $ there exists an entire function $G_{\Delta, A}(x) \in L^2(\mathbb{R})$ such that
  \begin{enumerate}
  \item The Fourier transform $\widehat{G}_{\Delta, A}(x)$ is supported on $[-\Delta^{2A}, \Delta^{2A}]$.
  \item We have,
    $
    0 \leq G_{\Delta, A}(x) \leq 1
    $
    for all $x \in \mathbb{R}$.
  \item We have
    $
    \mathbf{1}(x \in [0, \Delta^{-1}]) \leq G_{\Delta, A}(x) \cdot (1 + C e^{-\Delta^{A - 1}}).
    $
  \item We have,
    $
    G_{\Delta, A}(x) \leq \mathbf{1}(x \in [-\Delta^{-A/2} , \Delta^{-1} + \Delta^{-A/2}]) + C e^{-\Delta^{A - 1}}.
    $
\item We have, 
  $
  \int_{\mathbb{R}} |\widehat{G}_{\Delta, A}(x)| \dd x \leq 2\Delta^{2A} .
  $
  \end{enumerate}\begin{proof}
    Let $F = F_0 / \| F_0 \|_1$ so that $\int_{\mathbb R} F(x)\dd x=1$,  where $F_0$ is the function of Lemma \ref{le:ingham}.
  Consider
    \begin{equation} \label{eq:harmo}
    G_{\Delta, A}(x) = \int_{-\Delta^{-A}}^{\Delta^{-1} + \Delta^{-A}} \Delta^{2A} F(\Delta^{2A} (x - t)) \dd t.
    \end{equation}
    Notice that the Fourier transform of $F(\Delta^{2 A} (x - t))$ is compactly supported on $[-\Delta^{2A}, \Delta^{2A}]$, and therefore so is the Fourier transform of $G_{\Delta, A}$. 
 Clearly, $G_{\Delta, A}$  is non-negative. By completing the integral to infinity and a change of variables,  $G_{\Delta, A}$ is bounded by $1$. This proves the first two assertions.
        
For a given $x \in [0, \Delta^{-1}]$, the right-hand side of \eqref{eq:harmo} is at least
    $$
    C_{\Delta, A} = \int_{-\Delta^{-A}}^{\Delta^{-A}} \Delta^{2A} F(\Delta^{2A} x) \dd x = \int_{-\Delta^{A}}^{\Delta^{A}} F(x) \dd x = 1 + \OO(e^{-\Delta^{A - 1}}).
    $$
    Hence, for $x \in [0, \Delta^{-1}]$, we have
    $
    1 \leq G_{\Delta, A}(x) / C_{\Delta, A} = G_{\Delta, A}(x) \, (1 + \OO(e^{-\Delta^{A - 1}})),
    $
    thus proving the third assertion.

    For $x \in [-\Delta^{-A/2}, \Delta^{-1} + \Delta^{-A/2}]$, the upper bound $G_{\Delta, A}(x) \leq 1$ is immediate from completing the integral in \eqref{eq:harmo} to all $t\in \mathbb{R}$. Thus we can assume that $x \not \in [-\Delta^{-A/2}, \Delta^{-1} + \Delta^{-A/2}]$. We want to show that for such $x$ we have $G_{\Delta, A}(x) \ll e^{-\Delta^{A - 1}}$. Assuming first that $x < -\Delta^{-A/ 2}$ we get 
    $$
    G_{\Delta, A}(x) = \int_{-\Delta^{-A}}^{\Delta^{-1} + \Delta^{-A}} \Delta^{2A} F(\Delta^{2A}(x - t)) \dd t \ll e^{-\Delta^{A - 1}},
    $$
    using the decay bound $F(x) \ll e^{-|x| / \log^2 (10 + |x|)}$. The bound for $x > \Delta^{-1} + \Delta^{-A / 2}$ is obtained in the same way.

    Finally, to prove the last claim, we first notice that, since $\widehat{G}_{\Delta, A}(x)$ is supported on $[-\Delta^{2A},\Delta^{2A}]$, the Cauchy-Schwarz inequality and the Plancherel theorem imply that
    \begin{equation}
    \label{eqn: G fourier bound}
    \int_{\mathbb{R}} |\widehat{G}_{\Delta, A}(x)| \dd x \leq \sqrt{2}\Delta^{A} \Big ( \int_{\mathbb{R}} |G_{\Delta, A}(x)|^2 \dd x \Big )^{1/2}.
    \end{equation}
Second, the Cauchy-Schwarz inequality also implies, taking $u=\Delta^{2A} t$ in \eqref{eq:harmo},
    $$
    |G_{\Delta, A}(x)|^2 \leq \Delta^{2A} \int^{\Delta^{2A-1}+\Delta^A}_{- \Delta^{A}} F^2(\Delta^{2A} x - u) \dd u \leq \Delta^{2A}  \int^{\Delta^{2A-1}+\Delta^A}_{- \Delta^{A}} F(\Delta^{2A} x - u) \dd u,
    $$
    since $0\leq F\leq 1$.
    Thus, we have by integrating
    $$
    \int_{\mathbb{R}} |G_{\Delta, A}(x)|^2 \dd x \leq 2\Delta^{2A},
    $$
    giving the desired bound in Equation \eqref{eqn: G fourier bound}.
  \end{proof}

\end{lemma}

\subsection{Approximation of indicators by Dirichlet polynomials}
We will work throughout with the increments
$$
Y_j := S_j-S_{j-1} \ , \ j \geq 1,
$$
 with $S_j$ as  in Equation \eqref{eqn: Sk}.
For $\ell \geq -1$ and $k \in (n_{\ell}, n_{\ell + 1}]$,  consider the discretization parameter
$$
\Delta_j=(\min(j,n-j))^{4}, \qquad j\in (n_{\ell},k].
$$
We approximate indicator functions of $Y_j$ on intervals of width $\Delta_j^{-1}$.
This choice for $\Delta_j$ is guided by two constraints. First, some summability is used, in particular in (\ref{eqn:summ1}). From the proof it will be clear that we could choose any exponent strictly greater than 1 instead of $4$. Second, the Gaussian approximation of the Dirichlet sums gets worse for very small primes, imposing a decrease down to $\Delta_j\asymp 1$ for $j\asymp 1$, see Equation (\ref{eqn:GaussExp}) below.

Set $r=r(y) = \lceil y / 4 \rceil$.  Since $y > 4000$ we have $r > n_{-1} = 1000$. 
For $L_y(r) \leq v - m(r) \leq U_y(r) $ and $L_y(k) \leq w - m(k) \leq U_y(k)$,
define the set $\mathcal{I}_{r,k}(v,w)\subset \mathbb R^{k-r}$ of $(k-r)$-tuples $(u_{r + 1},\dots, u_k)$ with $u_j\in \Delta_j^{-1}\mathbb Z$, $r<j\leq k$ such that  
 \begin{equation}
 \label{eqn: I def}
 \begin{aligned}
\text{for all $j \in (r, k]$: }  L_y(j) - 1&\leq v +\sum_{i=r+1}^j u_i-m(j)\leq U_y( j) + 1,\\
     \Big | \sum_{i=r+1}^k u_i + v-w \Big | & \leq 1
     \end{aligned}
 \end{equation}

Note that since $U_y(j)-L_y(j)\leq 40 \min(j, n-j)$ the first restriction on the $u_j$'s imply that $|u_j|\leq 100 \Delta_j^{1/4}$ for every  $j\in (r,k]$.

Given $\Delta, A > 1$, we define the following truncated polynomial,
   \begin{equation}
     \label{eqn: Q}
     \mathcal D_{\Delta, A}(x)= \sum_{\ell \leq \Delta^{10A}} \frac{(2\pi \ii x)^{\ell}}{\ell!} \int_{\mathbb{R}} \xi^{\ell} \widehat{G}_{\Delta, A}(\xi) {\rm d}\xi .
   \end{equation}
We will be approximating the indicator function $ \mathbf{1}(Y_j(h)  \in [u_j, u_j + \Delta_j^{-1}])$ by the Dirichlet polynomial  $\mathcal D_{\Delta_j, A}(Y_j-u_j)$.
The following properties of $\mathcal D_{\Delta_j, A}(Y_j-u_j)$ are straightforward from the  definition of $\mathcal D_{\Delta, A}$ and $Y_j$:
\begin{enumerate}
\item It is is supported on integers $n$ whose prime factors lie in $(\exp e^{j-1},\exp(e^j)]$ and such that $\Omega(n) \leq \Delta_j^{10A}$.
\item The length of the Dirichlet polynomial   $\mathcal D_{\Delta_j, A}(Y_j-u_j)$ is at most $\exp(2\Delta_j^{10A}e^j)$ (the factor $2$ in the exponential is due to the second order term $p^{-1-2\ii h}$ in the summands of $S_k$).
\item We have
\begin{equation}
\label{eqn: bound G}
\int_{\mathbb{R}} |\xi|^{\ell} |\widehat{G}_{\Delta,A}(\xi)| {\rm d} \xi \leq \Delta^{2A \ell}\int_{\mathbb{R}} |\widehat{G}_{\Delta,A}(\xi)| {\rm d} \xi \leq 2 \Delta^{2A \ell} \Delta^{2A},
\end{equation}
by properties {\it (1)} and {\it (5)} of Lemma \ref{le:harmo}. In particular, the coefficients of $\mathcal D_{\Delta_j, A}(Y_j-u_j)$ are bounded by $\ll \Delta_j^{2A(\ell+1)}$.
\end{enumerate}
The first lemma successively approximates the indicator functions $\mathbf{1}(Y_j(h)  \in [u_j, u_j + \Delta_j^{-1}])$ by the polynomials $\mathcal D_{\Delta_j, A}(Y_j(h)-u_j)$.

\begin{lemma}
  \label{lem: smooth indicator}
  Let $A > 10$. Let $y > 4000$, $\ell\geq -1$ and $k > r$. 
  Let $w$ be such that $L_y(k)\leq w - m(k) \leq U_y(k)$.
  Then, for any fixed $\tau$, one has
  $$
  \begin{aligned}
 & \mathbf{1} \Big ( h \in  B^{(k)}_{\ell} \cap C^{(k)}_{\ell}: S_{k}(h) \in [w, w + 1] \Big ) \\
 &\hspace{1cm} \leq C \sum_{\substack{v \in  \Delta_{r}^{-1}\mathbb{Z} \\ L_y(r) \leq v - m(r) \leq U_y(r) \\ \mathbf{u} \in \mathcal{I}_{r,k}(v, w)}} |\mathcal{D}_{\Delta_{r}, A}(S_{r}(h) - v)|^2 
  \prod_{j =r+1}^{k} |\mathcal D_{\Delta_j, A}(Y_j(h)-u_j) |^2 .
\end{aligned}
  $$
  with $C > 0$ an absolute constant. 
\end{lemma}

The proof of the above lemma is split in two parts. We will first rely on the following claim: For every $j \in (n_{\ell}, k]$ and any $|u_j| \leq 100 \min(j, n - j)$, we have
      \begin{equation} \label{eq:mainineq0}
        \mathbf{1}(Y_j(h)  \in [u_j, u_j + \Delta_j^{-1}]) \leq |\mathcal{D}_{\Delta_j, A}(Y_j (h) - u_j)|^2 \, (1 + C e^{-\Delta_{j}^{A - 1}}),
  \end{equation}
  and for $|v| \leq 100 \min(j, n - j)$

    \begin{equation} \label{eq:Srapprox}
        \mathbf{1}(S_r(h)  \in [v, v + \Delta_r^{-1}]) \leq C|\mathcal{D}_{\Delta_r, A}(S_r (h) - v)|^2,
  \end{equation}
      with $C > 0$ an absolute constant. 
      
\subsubsection{Proof of Equations \eqref{eq:mainineq0} and \eqref{eq:Srapprox}}
We prove Equation \eqref{eq:mainineq0}. Equation \eqref{eq:Srapprox} is done the same way.
Lemma \ref{le:harmo} implies
\begin{align*}
\label{eqn:indicatorDP2}
\mathbf 1(Y_j(h)\in [u_j,u_j+\Delta_j^{-1}]) & \leq |G_{\Delta_j,A}(Y_j-u_j)|^2 \, (1 + C e^{-\Delta_j^{A - 1}}) \\ & = \Big |\int_{\mathbb R} e^{2\pi \ii \xi (Y_j(h)-u_j)} \widehat{G}_{\Delta_j,A}(\xi){\rm d} \xi\Big|^2 \, (1 + C e^{-\Delta_j^{A - 1}}) \ ,
\end{align*}
with $C > 0$ an absolute constant.
Expanding the exponential up to $\nu = \Delta_{j}^{10A}$, the integral in the absolute value is equal to
\begin{equation}
\label{eqn:indicatorDP1}
  \sum_{\ell \leq \nu} \frac{(2\pi \ii)^{\ell}}{\ell!}  (Y_j(h)-u_j)^{\ell}  \int_{\mathbb{R}} \xi^{\ell} \widehat{G}_{\Delta, A}(\xi) {\rm d}\xi 
  + \OO^{\star} \Big ( \frac{(2\pi)^{\nu}}{\nu!} |Y_j(h)-u_j|^{\nu}\int_{\mathbb{R}} |\xi|^{\nu} |\widehat{G}_{\Delta_j, A}(\xi)| {\rm d}\xi \Big)
\end{equation}
where $\OO^{\star}$ means that the implicit constant in the $\OO$ is $ \leq 1$. 

To bound the error term, observe that, since $h \in B_{\ell}^{(k)} \cap C_{\ell}^{(k)}$, the restriction on $u_j$ and on $Y_j(h)$ imposed by the upper and lower barriers imply  $|Y_j(h)-u_j|\leq 10^4 \Delta_j^{1/4}$. 
Together with \eqref{eqn: bound G}, this implies the bound
  \begin{equation}
  \label{eqn: error lemma4}
  \frac{(2\pi)^{\nu}}{\nu!} |Y_j(h)-u_j|^{\nu} \int_{\mathbb{R}} |\xi^{\nu}| |\widehat{G}_{\Delta_j, A}(\xi)| {\rm d}\xi \leq \frac{(10^6)^{\nu}}{\nu !} \Delta_j^{\nu/4}\,\Delta_j^{2A(\nu + 1)} \leq \frac{(10^6)^{\nu}}{\nu !} \Delta_j^{3A\nu},
  \end{equation}
  provided that $A > 5$. 

  The choice $\nu=\Delta_j^{10A}$ ensures that altogether the error is of order $\leq e^{-\Delta_j^{4A}}$. Thus we have shown that,
  $$
  \mathbf{1}(Y_j(h)\in [u_j,u_j+\Delta_j^{-1}]) \leq |\mathcal{D}_{\Delta_j, A}(Y_j(h) - u_j) + \OO^{\star} (e^{-\Delta_j^{4A}})|^2 \, (1 + C e^{-\Delta_j^{A - 1}}).
  $$
  Notice that if the left-hand side is equal to one, then $\mathcal{D}_{\Delta_j, A}(Y_j (h) - u_j)$ is at least $1/2$ in absolute value, therefore we can re-write the above as (\ref{eq:mainineq0})
  for some absolute constant $C > 0$, establishing the claim.

\subsubsection{Conclusion of the proof of Lemma   \ref{lem: smooth indicator}}
We partition the event $L_y(r) \leq S_{r}(h) -m(r) \leq U_y(r)$ into the union of events $S_{r}(h) \in [v,v+\Delta_{r}^{-1}]$ with 
$$
v-m(r)  \in  [L_y(r) , U_y(r)]\cap \Delta_{r}^{-1}\mathbb Z.
$$ 
Moreover, for $h \in B_{\ell}^{(k)} \cap C_{\ell}^{(k)}$, if we assume that for all $j \in (r, k)$ $Y_j(h) \in [u_j, u_j + \Delta_j^{-1}]$, $S_k(h) \in [w, w + 1]$ and $S_{r}(h) \in [v,v+\Delta_{r}^{-1}]$, then one must have
\begin{equation}
\label{eqn: I}
\begin{aligned}
v+\sum_{r + 1 \leq i\leq k}u_i&\leq S_{r}(h)+\sum_{i=r+1}^kY_i(h)\leq w+ 1,\\
 v+\sum_{r + 1 \leq i\leq k}u_i&\geq S_{r}(h)-\Delta_{r}^{-1}+\sum_{i=r+1}^k(Y_i(h)-\Delta^{-1}_i)\geq w-2(\Delta_{k}^{-3/4} + \Delta_{r}^{-3/4}),
\end{aligned}
\end{equation}
and under the same assumption for $j \in (r, k)$,
\begin{equation}
\label{eqn: I2}
\begin{aligned}
v+\sum_{r + 1 \leq i\leq j}u_i&\leq S_{r}(h)+\sum_{i=r+1}^jY_i(h)\leq m(j) + U_y(j),\\
 v+\sum_{r + 1 \leq i\leq j}u_i&\geq S_{r}(h)-\Delta_{r}^{-1}+\sum_{i=r+1}^j(Y_i(h)-\Delta_i^{ -1})\geq m(j)+ L_y(j)-1.
\end{aligned}
\end{equation}
These are the defining properties of the set $\mathcal I_{r,k}(v,w)$ in \eqref{eqn: I def}.
These observations and the inequality \eqref{eq:mainineq0} applied successively to every $Y_j(h)$ and to $S_{r}(h)$ yield
  $$
  \begin{aligned}
 & \mathbf{1} \Big ( h \in B^{(k)}_{\ell} \cap C^{(k)}_{\ell}: S_{k}(h) \in [w, w + 1] \Big ) \\
 &\leq  C \sum_{\substack{v \in \Delta_{r}^{-1} \mathbb{Z} \\ - L_y(r) \leq v - m(r) \leq U_y(r) \\ \mathbf{u} \in \mathcal{I}_{r,k}(v, w)}}  |\mathcal{D}_{\Delta_{r}, A}(S_{r}(h) - v)|^2 
    \prod_{j =r+1}^{k} \Big ( |\mathcal D_{\Delta_j, A}(Y_j(h)-u_j) |^2 \, (1 +  C e^{-\Delta_j^{A - 1}} ) \Big ).
\end{aligned}
$$
  Finally, we have
  $
  \prod_{j = r + 1}^{k} (1 + C e^{-\Delta_j^{A - 1}}) \leq C_0
  $
  for some absolute constant $C_0 > 0$. This proves the lemma.

\subsection{Comparison with a random model}
Define the random variables
\begin{equation}
\label{eqn: gaussian walk}
\mathcal S_{k}(h) = \sum_{e^{1000} \leq \log p \leq e^k} \re \Big ( Z_p \, p^{-(\tfrac 12 +\ii h)} + \tfrac 12 \, Z_p^2 \, p^{-(1 + 2 \ii h)} \Big ) , \qquad \mathcal Y_k(h)=\mathcal S_{k}(h)-\mathcal S_{k-1}(h),
\end{equation}
where $(Z_p, p \text{ prime})$ are independent and identically distributed copies of a random variable uniformly distributed on the unit circle $|z| = 1$. 
Notice that, since the increments $\mathcal Y_k(h)$ are sums of independent variables, one expects that they are approximately Gaussian with mean zero and variance $\tfrac 12$. 
Moreover, denote
\begin{equation}
\label{eqn:trulygaussian}
\mathcal G_{k} = \sum_{1000 \leq \ell \leq k}  \mathcal{N}_{\ell}, 
\end{equation}
where the $\mathcal{N}_{\ell}$'s are centered, independent real Gaussian random variables, with variance $\frac{1}{2}$. Note that $\mathcal G$ does not depend on $h$.

The following lemma shows that one can replace the Dirichlet polynomial $Y_j$ in expectation by the random variables $\mathcal Y_j,\mathcal N_j$ in the approximate indicators with a small error. This uses Lemma \ref{lem: Transition} and Lemma \ref{le:splitting} in Appendix \ref{se:moments}.

\begin{lemma} \label{lem: Gaussian comparison}
  Let $y > 4000$.
Let $A > 10$ and $\ell \geq -1$ with $\exp(10^6 (n - n_{\ell})^{10 A} e^{n_{\ell + 1}}) \leq \exp(\tfrac{1}{100} e^n)$ be given. Let $k \in (n_\ell, n_{\ell + 1}]$. Let $L_y(r) \leq v - m(r) \leq U_y(r)$. 
One has for $h\in[-2,2]$,
\begin{align*}
\mathbb E  \Big [|\mathcal{D}_{\Delta_{r}, A} & (S_{r}(h) - v)|^2 \prod_{j=r + 1}^k|\mathcal D_{\Delta_j, A}(Y_j(h)-u_j)|^2\Big ]\\ & \leq ( 1+ C e^{-c e^n}) \mathbb{E} [ |\mathcal{D}_{\Delta_{r}, A}(\mathcal{S}_{r}(h) - v)|^2 ] \ \prod_{j=r + 1}^k\mathbb E[|\mathcal D_{\Delta_j, A}(\mathcal Y_j(h)-u_j)|^2],
\end{align*}
with $C,c > 0$ absolute constants. 
Furthermore, for $w - m(k) \in [L_y(k), U_y(k)]$, we have 
\begin{equation}\label{eqn:fromDtoP}
 \begin{aligned}
&\sum_{\substack{v \in \Delta_{r}^{-1} \mathbb{Z} \\ v - m(r) \in [L_y(r), U_y(r)] \\ \mathbf u\in \mathcal{I}_{r,k}(v,w)}} \mathbb{E} \Big [ |\mathcal{D}_{\Delta_{r}, A} (\mathcal S_{r}(h) - v)|^2 \Big ] \prod_{j =r+1}^{k} \mathbb E\Big[ |\mathcal D_{\Delta_j, A}(\mathcal Y_j(h)-u_j) |^2\Big] \\
&\hspace{0.1cm}\leq C \sum_{\substack{v \in \Delta_{r}^{-1}  \mathbb{Z} \\ v - m(r) \in [L_y(r), U_y(r)] \\ \mathbf u\in \mathcal{I}_{r,k}(v,w)}}\   \mathbb P\big( \mathcal{G}_{r} \in [v, v + \Delta_{r}^{-1}] \text{ and } \mathcal N_j\in [u_j, u_j+\Delta_j^{-1}] \ \forall r<j\leq k\big) ,
 \end{aligned}
\end{equation}
 with $C > 0$ an absolute constant and   $\mathcal{I}_{k,\ell}(v,w)$ defined in \eqref{eqn: I}. 
\end{lemma}

\begin{proof}

  Note that
  $\mathcal{D}_{\Delta_{r}, A}(S_{r}(h) - v) \prod_{j =r+1}^{k} \mathcal D_{\Delta_j, A}({Y_j}(h)-u_j)$
  is a Dirichlet polynomial of length at most
  $$
  \exp \Big (2 \sum_{j = r}^{k} e^j \Delta_j^{10A} \Big ) \leq \exp \Big (10 e^{n_{\ell + 1}} \Delta_{n_{\ell}}^{10 A} \Big ).
  $$
  The first claim then follows from Lemma \ref{lem: Transition} and Lemma \ref{le:splitting}, both in Appendix \ref{se:moments}. Note that the multiplicative error term from these lemmas is $1+N/T$ with $N$ the above degree of the Dirichlet polynomial; this error is bounded thanks to the assumption $\exp(10^6 (n - n_{\ell})^{10A} e^{n_{\ell + 1}}) \leq \exp(\tfrac{1}{100} e^n)$.

  To prove the second assertion, it will suffice to show that for every $j \in (r, k]$ we have,
    \begin{equation} \label{eq:toprove}
\E \Big [|\mathcal D_{\Delta_j, A}(\mathcal Y_j (h) - u_j)|^2 \Big ] \leq \PP (\mathcal N_j \in [u_j, u_j + \Delta_j^{-1}]) \cdot (1 + \OO(\Delta_j^{-A/4})),
\end{equation}
    with an absolute implicit constant in $\OO(\cdot)$, and moreover that,
    \begin{equation} \label{eq:toprove2}
\E \Big [|\mathcal D_{\Delta_{r}, A}(\mathcal S_{r} (h) - v)|^2 \Big ] \leq C \PP (\mathcal G_{r} \in [v, v + \Delta_{r}^{-1}]).    
    \end{equation}
    with $C > 0$ an absolute constant. 
Then taking the product of the above inequalities over all $j \in (r, k]$, we conclude that
  \begin{align*}
\E & \Big [|\mathcal{D}_{\Delta_{r}, A}(\mathcal{S}_{r}(h) - v)|^2 \prod_{j\in (r,k]}|\mathcal D_{\Delta_j, A}(\mathcal Y_j(h) -u_j)|^2 \Big ] \\ & \leq C \mathbb{P}(\mathcal{S}_{r} \in [v, v + \Delta_{r}^{-1}]) \prod_{j\in (r,k]} \PP(\mathcal N_j\in [u_j,u_j+\Delta_j^{-1}]) \cdot ( 1 + \OO(\Delta_j^{-A / 4} ))  .
  \end{align*}
  This gives the claim since $\prod_{j = r}^{k} (1 + \OO(\Delta_j^{-A / 4})) \leq C$ with $C > 0$ an absolute constant.

  It remains to prove \eqref{eq:toprove} and \eqref{eq:toprove2}.
  The first step is to replace $\mathcal D_{\Delta_j, A}$ by $G_{\Delta_j,A}$ with a good error using Equations  \eqref{eqn: Q} and \eqref{eqn:indicatorDP1} (with $\mathcal Y_j$ instead of $Y_j$ and $\mathcal{S}_{r}$ instead of $S_{r}$).
Note that on the event $|\mathcal Y_j(h)-u_j|\leq \Delta_j^{6A}$, the estimate   \eqref{eqn: error lemma4}
still holds. Indeed we have, with $\nu = \Delta_j^{10A}$, 
$$
  \frac{(2\pi)^{\nu}}{\nu!} |\mathcal Y_j(h)-u_j|^{\nu} \int_{\mathbb{R}} |\xi^{\nu}| |\widehat{G}_{\Delta_j, A}(\xi)| {\rm d}\xi \leq \frac{(10^6)^{\nu}}{\nu !} \Delta_j^{6 A \nu} \cdot\Delta_j^{2A (\nu + 1)} \leq \frac{(10^6)^{\nu}}{\nu !} \Delta_j^{9 A\nu} \cdot \Delta_j^{2\nu},
  $$
since $A > 10$. Moreover, since $\nu = \Delta_j^{10A}$, the above is $\leq e^{-\Delta_j^{4A}}$. 
This implies
\begin{equation}
\label{eqn: D<=}
\begin{aligned}
&\E[|\mathcal D_{\Delta_j, A}(\mathcal Y_j(h)-u_j)|^2\cdot\mathbf 1(|\mathcal Y_j(h)-u_j|\leq \Delta_j^{6A})]\\
&=\E[|G_{\Delta_j,A}(\mathcal Y_j(h)-u_j)+\OO(e^{-\Delta_j^{4A}})|^2\cdot\mathbf 1(|\mathcal Y_j(h)-u_j|\leq \Delta_j^{6A})]\\
&\leq \E[|G_{\Delta_j,A}(\mathcal Y_j(h)-u_j)|^2]+\OO(e^{-\Delta_j^{4A}}),
\end{aligned}
\end{equation}
since by Lemma \ref{le:harmo} we have $G_{\Delta_j, A}(\mathcal Y_j(h) - u_j) \in [0,1]$. 
A quick computation shows that $\mathbb{E}[e^{K \mathcal{Y}_j(h)}] \ll_{K} 1$ for any given $K > 1$ and all $j \geq 1$ and $h \in [-2,2]$, see Lemma \ref{lem:basicFour} in Appendix \ref{se:moments}.
Therefore the contribution of the event $|\mathcal Y_j(h)-u_j|> \Delta_j^{6A}$ can be bounded by Chernoff's inequality:
$$
\begin{aligned}
\E[|\mathcal D_{\Delta_j, A}& (\mathcal Y_j(h)-u_j)|^2\cdot\mathbf 1(|\mathcal Y_j(h)-u_j|> \Delta_j^{6A})] \\
&\leq \E[|\mathcal D_{\Delta_j, A}(\mathcal Y_j(h)-u_j)|^4]^{1/2}\ \PP(|\mathcal Y_j(h)-u_j|> \Delta_j^{6A})^{1/2}\\
&\ll \E[|\mathcal D_{\Delta_j, A}(\mathcal Y_j(h)-u_j)|^4]^{1/2}\, e^{- \tfrac 14 \Delta_j^{6A}},
\end{aligned}
$$
where we used $|u_j| \leq  100\Delta_j^{1/4}$ in the Chernoff's inequality.
The fourth moment is easily bounded using an estimate similar to \eqref{eqn: bound G}:
\begin{multline*}
 \E[|\mathcal D_{\Delta_j, A}(\mathcal Y_j(h)-u_j)|^4]
  \leq \E\Big[ \Big ( \sum_{\ell \leq \Delta_j^{10 A}} \frac{(2\pi )^{\ell}}{\ell!}  2\Delta_j^{2A(\ell+1)}  (|\mathcal Y_j(h)|+10^4 \Delta_j^2)^{\ell} \Big )^4\Big ]  \\
  \ll  \Delta_j^{2A} \, \E[\exp( 9\pi   \Delta_j^{2A}  (| \mathcal Y_j(h)|+10^4 \Delta_j^2))] \ll e^{\Delta_j^{5A}},
\end{multline*}
 where we used Lemma \ref{lem:basicFour} together with $e^{c|\mathcal Y|}\leq e^{c\mathcal Y}+e^{-c\mathcal Y}$. 
Putting this together we get
\begin{equation}
\label{eqn: Q to G}
\E[|\mathcal D_{\Delta_j, A}(\mathcal Y_j(h)-u_j)|^2] \leq \E[|G_{\Delta_j,A}(\mathcal Y_j(h)-u_j)|^2]+\OO(e^{-\tfrac 18 \Delta_j^{6A}}).
\end{equation}
Furthermore, by Lemma \ref{le:harmo}, we have
$$
\E[|G_{\Delta_j, A}(\mathcal Y_j(h) - u_j)|^2] \leq \PP(\mathcal Y_j(h) \in [u_j - \Delta_j^{-A / 2}, u_j + \Delta_j^{-1} + \Delta_j^{-A / 2}]) + \OO ( e^{-\Delta_j^{A - 1}}). 
$$
Since $|u_j| \leq 100 \min(j, n - j)$ and $j > y / 4$, we obtain from Lemma \ref{lem:compProb} in Appendix \ref{se:ballot} that for all $h$
\begin{equation}\label{eqn:GaussExp}
\PP(\mathcal Y_j(h) \in [u_j - \Delta_j^{-A / 2}, u_j + \Delta_j^{-1} + \Delta_j^{-A / 2}]) = \PP(\mathcal N_j \in [u_j , u_j + \Delta_j^{-1}]) \cdot (1 + \OO(\Delta_j^{-A/4})). 
\end{equation}
Note that the Gaussian distribution and the restriction on $u_j$ and $j$ are heavily used here to get the error term.
This concludes the proof of \eqref{eq:toprove}.

The proof of \eqref{eq:toprove2} is similar, with the main difference being that we use Lemma \ref{le:saddlepoint} in order to show that
$$
\mathbb{P}(\mathcal{S}_r \in [v - \Delta_{r}^{-A / 2}, v + \Delta^{-1}_{r} + \Delta_{r}^{-A / 2}]) + e^{-\Delta_{r}^{A - 1}} \leq C \mathbb{P}(\mathcal G_{r} \in [v , v + \Delta_{r}^{-1}]).
  $$
\end{proof}

\subsection{Proof of Lemma \ref{master lemma}}
Let $A = 20$.
 By Lemma \ref{lem: smooth indicator}, we have
  \begin{align} \label{eq:mainnnineq} \nonumber
  \mathbf{1} \Big ( h \in B_{\ell}^{(k)} & \cap C_{\ell}^{(k)} \text{ and } S_k(h) \in (w, w + 1] \Big ) 
 \\ & \leq C 
 \sum_{\substack{v \in \Delta_{r}^{-1} \mathbb{Z} \\ L_y(r) \leq v- m(r) \leq U_y(r) \\  \mathbf{u} \in \mathcal{I}_{r,k}(v, w)}} |\mathcal{D}_{\Delta_{r}, A}(S_{r}(h) - v)|^2  
  \prod_{j\in (r,k]} |\mathcal D_{\Delta_j, A}(Y_j(h)-u_j) |^2  ,
  \end{align}
  $C > 0$ an absolute constant.
By the properties of $\mathcal D_{\Delta_j, A}(Y_j(h)-u_j)$, we can write the right-hand side of \eqref{eq:mainnnineq} as 
\begin{equation}
\label{eq:bull}
\begin{aligned}
      \sum_{i \in \mathcal{I}} |D_{i}(\tfrac 12 + \ii \tau + \ii h)|^2
   \end{aligned}
\end{equation}
a linear combination of squares of Dirichlet polynomials $D_i$, each of length
   $$
   \leq \exp \Big ( 2 \sum_{0 \leq j \leq k} e^j \Delta_j^{200} \Big ) \leq \exp(100 e^k (n - k)^{800}).
   $$
   Therefore multiplying \eqref{eq:mainnnineq} by an arbitrary Dirichlet polynomial $\mathcal{Q}$ of length $N \leq \exp(\tfrac {1}{100} n)$ and applying the  discretization in Lemma \ref{le:discretization}, we conclude that
   \begin{equation}
   \begin{aligned} \label{eq:bounddd}
   \mathbb{E} & \Big [\max_{|h| \leq 2} |\mathcal{Q}(\tfrac 12 +\ii \tau + \ii h)|^2 \cdot \mathbf{1}(h \in B_{\ell}^{(k)} \cap C_{\ell}^{(k)} \text{ and } S_k(h) \in (w, w + 1]) \Big ] \\ & \ll \Big ( \log N + e^k (n - k)^{800} \Big ) \sum_{i \in \mathcal{I}} \mathbb{E} \Big [ |\mathcal{Q}(\tfrac 12 + \ii \tau )|^2 \, |D_{i} (\tfrac 12 + \ii \tau)|^2 \Big ].
   \end{aligned}
   \end{equation}
   Here we use the fact that the expectations have the same values (up to negligible factors) for the $\OO( \log N + e^k (n - k)^{800} )$ relevant $h$'s in Lemma \ref{le:discretization},
   and the contribution of the remaining $h$'s associated with very large $j$ in  Lemma \ref{le:discretization} are bounded similarly to the paragraph after (\ref{eqn: B0}).
   All the following expressions are evaluated at $h=0$. 
   The Dirichlet polynomials $D_i$ are all of length $\leq \exp(\tfrac {1}{100} n)$ and supported on integers $n$ all of whose prime factors are in $\leq \exp(e^k)$, while $\mathcal{Q}$ is supported on integers $n$ all of whose prime factors are $> \exp(e^k)$. Therefore, Lemma \ref{le:splitting} can be applied and yields
   $$
   \mathbb{E}[ | \mathcal{Q}(\tfrac 12 +\ii \tau)|^2 \, |D_{i}(\tfrac 12 +\ii \tau)|^2 ] \leq 2 \mathbb{E}[|\mathcal{Q}(\tfrac 12 +\ii \tau)|^2] \, \mathbb{E}[|D_i(\tfrac 12 +\ii \tau)|^2]. 
   $$
   Finally, by the definition of $D_i$ and $\mathcal{I}$ in   \eqref{eq:bull} and Lemma \ref{lem: Gaussian comparison},
 we have
   \begin{equation}
\begin{aligned}
    \label{eq:rest}
 & \sum_{i \in \mathcal{I}} \mathbb{E} \Big [ |D_i(\tfrac 12 +\ii \tau)|^2 \Big ]\leq \\ &  C \sum_{\substack{v \in \Delta_{r}^{-1} \mathbb{Z} \\ L_y(r) \leq v - m(r) \leq U_y(r) \\ \mathbf{u} \in \mathcal{I}_{r,k}(v, w)}} \mathbb{P}(\mathcal G_r \in [v, v + \Delta_{r}^{-1}] \text{ and }\mathcal N_j\in [u_j, u_j + \Delta_{j}^{-1}] \ \forall r < j \leq k),
   \end{aligned}
   \end{equation}
with $C > 0$ an absolute constant. 

 If for every $r < j \leq k$, we have $\mathcal N_{j} \in [u_j, u_j + \Delta_j^{-1}]$  and moreover $\mathcal G_{r} \in [v, v + \Delta_{r}^{-1}]$ and \eqref{eqn: I def} holds, then we have
\begin{equation}\label{eqn:summ1}
  \begin{aligned}
\forall j \in (r, k] :   \mathcal G_{j} & \leq m(j) + U_y(j) + 1 + \sum_{r < i \leq j} \Delta_i^{-1}, \\
  |\mathcal  G_k - w | & \leq 1 + \sum_{r \leq j \leq k} \Delta_j^{-1}, \\
 \mathcal G_{r} &\in [v, v + \Delta_{r}^{-1}].
  \end{aligned}
\end{equation}
As a result after summing over $v \in \Delta_{r}^{-1} \mathbb{Z}$ we can bound \eqref{eq:rest} by 
$$
\leq C \, \mathbb{P} \Big ( \mathcal G_j \leq m(j) + U_y(j) + 2 \text{ for all } r \leq j \leq k \text{ and } {\mathcal G}_{k} \in [w - 2, w + 2] \Big ).
$$
  Consequently, plugging this into \eqref{eq:bounddd}, we obtain for $h\in[-2,2]$,
  \begin{align*}
   \mathbb{E} & \Big [\max_{|h| \leq 2} |\mathcal{Q}(\tfrac 12 +\ii \tau + \ii h)|^2 \cdot \mathbf{1}(h \in B_{\ell}^{(k)} \cap C_{\ell}^{(k)} \text{ and } S_k(h) \in (w, w + 1]) \Big ] \\ & \ll \Big ( \log N + e^k (n - k)^{800} \Big ) \mathbb{E} \Big [ |\mathcal{Q}(\tfrac 12 +\ii \tau )|^2 \Big ] \\ & \times \mathbb{P} \Big (\mathcal G_j \leq m(j) + U_y(j) + 2 \text{ for all } r \leq j \leq k \text{ and } \mathcal G_k(0) \in [w - 2, w + 2] \Big ) .
  \end{align*}
  It remains to apply  the version of the ballot theorem from Proposition \ref{lem: ballot} (with $y$ replaced by $y+2$ and adding the bounds with $w$ replaced by $w+i$, $i \in \{-2,-1,0,1\}$)
  to conclude that
$$
\begin{aligned}
&\mathbb E\Big[\max_{|h|\leq 2}|\mathcal{Q}(\tfrac 12 + \ii \tau+\ii h)|^2 \cdot \mathbf{1} \Big ( h \in B_{\ell}^{(k)} \cap C_{\ell}^{(k)}:  S_{k}(h) \in (w,w+1] \Big ) \Big ]\\ 
&\ll  \mathbb E[|\mathcal{Q}(\tfrac 12 + \ii \tau)|^2] \, \Big ( e^{-k} \, \log N + (n-k)^{800} \Big ) \, y\, (U_y(k)-w+m(k)+2)\, e^{-2(w-m(k))}.
\end{aligned}
$$
This concludes the proof of Lemma \ref{master lemma}.

\section{Decoupling and Twisted Fourth Moment}\label{sec:4th}

We now prove Lemma \ref{master lemma twisted}. We will need the following class of ``well-factorable'' Dirichlet polynomials. 

\begin{definition}
\label{df: well-factorable}
  Given $\ell \geq 0$ and $k \in [n_{\ell}, n_{\ell + 1}]$, we will say that a Dirichlet polynomial
  $\mathcal{Q}$ is degree-$k$ well-factorable if it can be written as
  $$
\Big ( \prod_{0 \leq \lambda \leq \ell} \mathcal{Q}_{\lambda}(s) \Big ) \mathcal{Q}_{\ell}^{(k)}(s),
  $$
where
\begin{align*}
 \mathcal{Q}_\lambda(s) := \sum_{\substack{p | m \implies p \in (T_{\lambda - 1}, T_\lambda] \\ \Omega_{\lambda}(m) \leq 10 (n_{\lambda} - n_{\lambda - 1})^{10^4}}} \frac{\gamma(m)}{m^s} \ \text{ and } \ 
   \mathcal{Q}_{\ell}^{(k)}(s):=\sum_{\substack{p | m \implies p \in (T_{\ell}, \exp(e^k)] \\ \Omega_{\ell}(m) \leq 10 (n_{\ell + 1}-n_{\ell})^{10^4}}} \frac{\gamma(m)}{m^s},
\end{align*}
and $\gamma$ are arbitrary coefficients such that $|\gamma(m)| \ll \exp(\tfrac{1}{500} e^n)$ for every $m \geq 1$.
\end{definition}

The proof of Lemma \ref{master lemma twisted} will rely on the following result on the twisted fourth moment. We postpone the proof of this technical lemma to the next subsection. 

\begin{lemma} \label{le:barrierapproxtwistedbis}
    Let $\ell \geq 0$ be such that $\exp(10^6 (n - n_{\ell})^{10^5} e^{n_{\ell + 1}}) \leq \exp(\tfrac {1}{100} e^n)$. Let $k\in[n_{\ell},n_{\ell + 1}]$. 
  Let $\mathcal{Q}$ be a degree-$k$ well-factorable Dirichlet polynomial as in Definition \ref{df: well-factorable}. Then, we have
  $$
  \mathbb{E} \Big [ |(\zeta_{\tau}  \mathcal{M}_{-1} \ldots  \mathcal{M}_{\ell}\mathcal{M}_{\ell}^{(k)})(\tfrac 12 +\ii \tau )|^4
  \cdot |\mathcal{Q}(\tfrac 12 +\ii \tau)|^2  \Big ] 
  \ll  e^{4(n-k)} \, \mathbb{E} \Big [ |\mathcal{Q}(\tfrac 12 +\ii \tau)|^2\Big ].
  $$
\end{lemma}

We are now ready to prove Lemma \ref{master lemma twisted}.

\begin{proof}[Proof of Lemma \ref{master lemma twisted}]
As in the proof of Lemma \ref{master lemma}, by Lemma \ref{lem: smooth indicator}, we have for $A=20$
  \begin{align} \label{eq:mainnnineq2} \nonumber
  \mathbf{1} \Big ( h \in B_{\ell} & \cap C_{\ell} \text{ and } S_{n_{\ell}}(h) \in (u, u + 1] \Big ) 
 \\ & \leq C 
 \sum_{\substack{v \in \Delta_{r}^{-1} \mathbb{Z} \\ L_y(r) \leq v- m(r) \leq U_y(r) \\  \mathbf{u} \in \mathcal{I}_{r, n_\ell}(v, w)}} |\mathcal{D}_{\Delta_{r}, A}(S_{r}(h) - v)|^2  
  \prod_{j\in (r,n_{\ell}]} |\mathcal D_{\Delta_j, A}(Y_j(h)-u_j) |^2  ,
  \end{align}
with  $C > 0$ an absolute constant.
By the properties of $\mathcal D_{\Delta_j, A}(Y_j(h)-u_j)$, we can write \eqref{eq:mainnnineq2} as 
\begin{equation}
\label{eq:bull2}
\begin{aligned}
      \sum_{i \in \mathcal{I}} |D_{i}(\tfrac 12 + \ii \tau + \ii h)|^2,
   \end{aligned}
\end{equation}
a linear combination of squares of Dirichlet polynomials of length
   $$
   \leq \exp \Big ( 2 \sum_{0 \leq j \leq k} e^j \Delta_j^{200} \Big ) \leq \exp(100 e^k (n - k)^{800}).
   $$
  We claim that, for every $i \in \mathcal{I}$, the Dirichlet polynomial
   $$
   \mathcal{Q}_{\ell}^{(k)}(\tfrac 12 +\ii \tau + \ii h) D_i(\tfrac 12 +\ii \tau + \ii h)
   $$
   is degree-$k$ well-factorable.
   This follows from the properties of $\mathcal D_{\Delta,A}$ listed after Equation \eqref{eqn: Q}. More precisely, 
   each $D_i$ has length  $$\leq \exp \Big ( 2 \sum_{0 \leq j \leq n_\ell} e^j \Delta_j^{200} \Big )\leq \exp(e^n/100).$$ Moreover, each $D_i$ is supported on the set of integers $m$ such that $p | m \implies p \leq e^{n_{\ell}}$, and for every $j \leq n_\ell$, $\Omega_j(m) \leq \Delta_{j}^{200}$. Furthermore, its coefficients are bounded by $\exp(e^n/500)$. 
   
   It then follows from Lemma \ref{le:barrierapproxtwistedbis} that
   \begin{align*}
  \sum_{i \in \mathcal{I}} \  & \mathbb{E}[|(\zeta_{\tau} \mathcal{M}_{-1} \ldots \mathcal{M}_{\ell} \mathcal{M}_{\ell}^{(k)})(h)|^2 \cdot |\mathcal{Q}_{\ell}^{(k)}(\tfrac 12 + \ii \tau + \ii h)|^2 \cdot |D_i(\tfrac 12 +\ii \tau + \ii h)|^2 ] \\ & \ll e^{4(n - k)} \sum_{i \in \mathcal{I}} \mathbb{E} [|\mathcal{Q}_\ell^{(k)}(\tfrac 12 + \ii \tau + \ii h)|^2 \cdot |D_i(\tfrac 12 + \ii \tau + \ii h)|^2 ].
   \end{align*}
   Moreover, since $\mathcal{Q}_\ell^{(k)}$ is supported on integers $n$ having only prime factors in $(\exp(e^{n_{\ell}}), \exp(e^k)]$, while $D_i$ is supported on integers $n$ all of whose prime factors are $\leq \exp(e^{n_{\ell}})$, and both Dirichlet polynomials have length $\leq \exp(\tfrac{1}{100} n)$, we conclude from Lemma \ref{le:splitting} that
   $$
   \mathbb{E}[|\mathcal{Q}_\ell^{(k)}(\tfrac 12 +\ii \tau + \ii h)|^2 \, |D_i(\tfrac 12 +\ii \tau + \ii h)|^2 ] \ll \mathbb{E} [|\mathcal{Q}_\ell^{(k)}(\tfrac 12 +\ii \tau + \ii h)|^2 ] \, \mathbb{E}[|D_i(\tfrac 12 + \ii \tau + \ii h)|^2].
   $$
   Therefore, we obtain that
   \begin{align*}
     \mathbb{E} & \Big [ |(\zeta_{\tau} \mathcal{M}_{-1} \ldots \mathcal{M}_{\ell} \mathcal{M}^{(k)}_{\ell})(h)|^4
       \cdot |\mathcal{Q}_{\ell}^{(k)}(h)|^2 \cdot \mathbf{1} \Big ( h \in B_{\ell} \cap C_{\ell} \text{ and } S_{n_{\ell}}(h) \in [u, u + 1] \Big ) \Big ]  \\
     & \ll e^{4(n - k)}\, \mathbb{E}[|\mathcal{Q}_{\ell}^{(k)}(\tfrac 12 + \ii \tau)|^2] \, \sum_{i \in \mathcal{I}} \mathbb{E}[|D_i(\tfrac 12 + \ii \tau)|^2].
   \end{align*}
   Now, proceeding exactly as in the proof of Lemma \ref{master lemma} starting from Equation \eqref{eq:rest} one gets\
   $$
   \sum_{i \in \mathcal{I}} \mathbb{E}[|D_i(\tfrac 12 + \ii\tau)|^2] \ll y \,(U_y(n_{\ell}) - u + m(n_\ell) + 2) e^{-2(u - m(n_\ell))} e^{-n_{\ell}}.
   $$
   This concludes the proof. \end{proof}

\subsection{Proof of Lemma \ref{le:barrierapproxtwistedbis}}\label{sec:twisted}
We first need to introduce some notations.
Define for $0 \leq i \leq \ell + 1$, 
  \begin{equation}\label{eqn:betai}
  \beta_{i}(m) := \sum_{\substack{m = abc \\ \Omega_{i}(a), \Omega_{i}(b) \leq (n_i - n_{ i - 1})^{10^5} \\ \Omega_i(c) \leq 10 (n_i - n_{i - 1})^{10^4}}} \mu(a) \mu(b) \gamma(c),
  \end{equation}
where $\Omega_i(m)$ denotes as before the number of prime factors of $m$ in the range $(T_{i-1},T_i]$.
  Given $m$, write $m = m_0 \dots m_{\ell} m_{\ell}^{(k)}$ where $m_j$ with $0 \leq j \leq \ell$ has prime factors in $(T_{j - 1}, T_{j}]$, and $m_{\ell}^{(k)}$ has prime factors in the interval $(T_{\ell}, \exp(e^k)]$.
Let $\beta(m)$ be defined by
  $$
  \sum_{m \geq 1} \frac{\beta(m)}{m^s} = \Big ( \prod_{0 \leq i \leq \ell} \mathcal{M}_{i}^2(s) \mathcal{Q}_{i}(s) \Big ) (\mathcal{M}_{\ell}^{(k)}(s))^2 \mathcal{Q}_{\ell}^{(k)}(s).
  $$
Note that, 
  \begin{equation}
  \label{eqn: beta}
  \beta(m) = \prod_{0 \leq i \leq \ell} \beta_{i}(m_i) \, \beta_{\ell + 1}(m_{\ell}^{(k)}).
  \end{equation}
  It will be convenient to redefine $T_{\ell + 1} := \exp(e^k)$ so that the above can be written as
  $$
  \prod_{0 \leq i \leq \ell + 1} \beta_{i}(m_i),
  $$
  with $m_{\ell + 1}$ defined as the largest divisor of $m$ all of whose prime factors belong to $(T_{\ell}, T_{\ell + 1}]$ and where $T_{\ell + 1} := \exp(e^k)$. 
  Given complex numbers $z_1,z_2,z_3,z_4$ and $n\in\mathbb{N}$, set $\mathbf{z} := (z_1, z_2, z_3, z_4)$ and consider
$$
B_{\mathbf{z}}(n) := B_{(z_1,z_2,z_3,z_4)}(n)=\prod_{p\mid n}
\left(\sum_{j\geq 0}\frac{\sigma_{z_1,z_2}(p^{v_p(n)+j})\sigma_{z_3,z_4}(p^j)}{p^j}\right)
\left(\sum_{j\geq 0}\frac{\sigma_{z_1,z_2}(p^{j})\sigma_{z_3,z_4}(p^j)}{p^j}\right)^{-1},
$$
with $\sigma_{z_1,z_2}(n)=\sum_{n_1n_2=n}n_1^{-z_1}n_2^{-z_2}$, and $v_p(n)$, the greatest integer $k$ such that $p^k\mid n$. We are now ready to start the proof. 

\begin{proof}[Proof of Lemma \ref{le:barrierapproxtwistedbis}]  
As proved in \cite[Section 6]{HeaRadSou2019}, the twisted fourth moment can be bounded by
\begin{equation}
\label{eqn: twisted 1}
\mathbb{E} [ |(\zeta_{\tau} \mathcal{M}_{-1} \ldots \mathcal{M}_{\ell} \mathcal{M}_{\ell}^{(k)})(0)|^4 \cdot |\mathcal{Q}(\tfrac 12 + \ii \tau)|^2 ] \ll e^{4n} \max_{\substack{j = 1,2,3,4 \\ |z_j| = 3^j / e^n}} |G(z_1, z_2, z_3, z_4)|,
  \end{equation}
  where
  \begin{equation}
\label{eqn: G}
  G(z_1, z_2, z_3, z_4) := \sum_{m_1,m_2} \frac{\beta(m_1) \overline{\beta(m_2)}}{[m_1,m_2]} B_{\mathbf{z}} \Big ( \frac{m_1}{(m_1,m_2)} \Big ) B_{\pi \mathbf{z}} \Big ( \frac{m_2}{(m_1,m_2)} \Big ),
  \end{equation}
  and $\mathbf{z} = (z_1, z_2, z_3, z_4)$, $\pi \mathbf{z} = (z_3, z_4, z_1, z_2)$. Equation (\ref{eqn: G}) relies on the reasoning from  \cite[Section 6]{HeaRadSou2019}. It requires a slightly changed version of Proposition 4 in \cite[Section 5]{HeaRadSou2019}, requiring a shorter Dirichlet polynomial with $\theta \leq \tfrac{1}{100}$ but allowing the coefficients to be as large as $T^{1/100}$. 
This change in the assumptions is possible by appealing to  \cite{HughesYoung} instead of \cite{Bettin} in the argument, see the third remark after Theorem 1 in \cite{HughesYoung}.
   
  Allowing  coefficients to be as large as $T^{1/100}$ is necessary because of our assumptions on the coefficients $\gamma$. We notice that by the definition of $\mathcal{M}_i$ and $\mathcal{M}_{i}^{(k)}$ the Dirichlet polynomial $\prod_{0 \leq i \leq \ell} (\mathcal{M}_i \, \mathcal{M}_{\ell}^{(k)})^2$ is of length at most $\exp(2 (n_{\ell + 1} - n_{\ell})^{10^5} e^{n_{\ell + 1}})$. The assumptions of the lemma imply $\exp(2 (n_{\ell + 1} - n_{\ell})^{10^5} e^{n_{\ell + 1}}) \leq \exp( 10^{-4} n)$. Furthermore by the definition of a degree-$k$ well factorable Dirichlet polynomial, $\mathcal{Q}$ is of a length $\leq \exp(\tfrac{1}{500} n)$. Therefore, the total length of the Dirichlet polynomial $\prod_{0 \leq i \leq \ell} (\mathcal{M}_{i} \, \mathcal{M}_{\ell}^{(k)})^2 \, \mathcal{Q}$ is $\leq \exp(\tfrac{1}{100} n)$ as needed.

From Equation \eqref{eqn: beta}, the function $G$ can be written as the product
  \begin{equation} \label{eq:unrealsum}
  \prod_{i \leq \ell  + 1} \Big ( \sum_{p | m_1,m_2 \implies T_{i - 1} < p \leq T_{i}} \frac{\beta_i(m_1) \overline{\beta_i(m_2)}}{[m_1,m_2]} B_{\mathbf{z}} \Big ( \frac{m_1}{(m_1,m_2)} \Big ) B_{\pi \mathbf{z}} \Big ( \frac{m_2}{(m_1,m_2)} \Big ) \Big ). 
  \end{equation}
 
Applying the definition (\ref{eqn:betai}) with the decompositions $m_1 = a_1 b_1 c_1$ and $m_2 = a_2 b_2 c_2$,
the inner sum in (\ref{eq:unrealsum}) at a given $i$ can also be written as
  \begin{equation}
  \label{eq:realsum}
  \begin{aligned}
  &\sum_{\substack{p | c_1, c_2 \implies p \in (T_{i - 1}, T_{i}]\\ \Omega_i(c_1),\Omega_i(c_2) \leq 10 (n_i - n_{i - 1})^{10^4}}} \gamma(c_1) \overline{\gamma(c_2)}\times \\ 
  &\sum_{\substack{p | a_1, a_2 \implies p \in (T_{i - 1}, T_{i}] \\ p | b_1,b_2 \implies p \in (T_{i - 1}, T_{i}]\\  \Omega_{i}(a_1) , \Omega_i(a_2) \leq (n_i - n_{i - 1})^{10^5} \\ \Omega_i(b_1) , \Omega_i(b_2) \leq (n_i - n_{i - 1})^{10^5}}}  \frac{\mu(a_1) \mu(a_2) \mu(b_1) \mu(b_2)}{[a_1 b_1 c_1, a_2 b_2 c_2]}   \, B_{\mathbf{z}} \Big ( \frac{a_1 b_1 c_1}{(a_1 b_1 c_1, a_2 b_2 c_2)} \Big ) B_{\pi \mathbf{z}} \Big ( \frac{a_2 b_2 c_2}{(a_1 b_1 c_1, a_2 b_2 c_2)} \Big ). 
  \end{aligned}
  \end{equation}

  Given an interval $I$, and integers $c_1, c_2 \geq 1$, we define the quantity
  \begin{equation}
  \label{eqn: mathfrak S}
 \mathfrak S_{I}(c_1, c_2) :=      \sum_{\substack{p |u, v \implies p \in I}}  \frac{f(u) f(v)}{[u c_1, v c_2]} B_{\mathbf{z}} \Big ( \frac{u c_1}{(u c_1, v c_2)} \Big ) B_{\pi \mathbf{z}} \Big ( \frac{v c_2}{(u c_1, v c_2)} \Big ),
\end{equation}
where $f$ is the multiplicative function such that $f(p) = -2$, $f(p^2) = 1$ and $f(p^{\alpha}) = 0$ for $\alpha \geq 3$. 
The rest of the argument relies on Lemma \ref{le:boundinglarge}  and Lemma \ref{le:twooftwo}.
Lemma \ref{le:boundinglarge} shows that the restriction on the number of factors for the $a$ and $b$'s can be dropped with a small error. Lemma \ref{le:twooftwo} evaluates the sum of \eqref{eqn: mathfrak S} without these restrictions.

  \begin{lemma}\label{le:boundinglarge}
    For $0 \leq i \leq \ell + 1$ 
    the equation \eqref{eq:realsum} is equal to
    $$
    \sum_{\substack{p | c_1, c_2 \implies p \in (T_{i - 1}, T_{i}] \\ \Omega_i(c_1), \Omega_i(c_2) \leq 10 (n_i - n_{i - 1})^{10^4}}} \gamma(c_1) \overline{\gamma(c_2)} \mathfrak S_{(T_{i - 1}, T_{i}]}(c_1, c_2) + \OO \Big ( e^{-100 (n_i - n_{i - 1})} \sum_{p | c \implies p \in (T_{i - 1}, T_{i}]} \frac{|\gamma(c)|^2}{c} \Big ),
        $$
        with an absolute implicit constant in $\OO(\cdot)$. 
  \end{lemma}
  
  We now define
    $$
  \mathcal{S}_{I}= \sum_{\substack{p | c_1, c_2 \implies p \in I \\ \Omega_i(c_1), \Omega_i(c_2) \leq 10 (n_i - n_{i - 1})^{10^4}}} |\gamma(c_1)| \cdot |\gamma(c_2)| \cdot |\mathfrak S_{I}(c_1, c_2)|.
  $$
  
    \begin{lemma} \label{le:twooftwo}
    We have, for $0 \leq i \leq \ell + 1$ and every interval $I \subset [T_{i - 1}, T_{i}]$,
    $$
  \mathcal{S}_{I}\leq \exp \Big ( e^{6000} (n_i - n_{i -1})^{4 \cdot 10^4} e^{n_i - n} \Big ) \exp \Big ( - \sum_{p \in I} \frac{4}{p} \Big ) \sum_{\substack{p | c \implies p \in I}} \frac{|\gamma(c)|^2}{c}. 
      $$
    \end{lemma}
The proof of these lemmas is deferred to the next subsections. We first conclude the proof of Lemma \ref{le:barrierapproxtwistedbis}.

   It follows from \eqref{eqn: twisted 1},  \eqref{eq:unrealsum} and Lemma \ref{le:boundinglarge} that
   \begin{equation} \label{eq:closetogood2}
  \begin{aligned}
    \mathbb{E} [ & |(\zeta_{\tau} \mathcal{M}_{-1} \ldots \mathcal{M}_{\ell} \mathcal{M}_{\ell}^{(k)})(0)|^4 \cdot |\mathcal{Q}(\tfrac 12 + \ii \tau)|^2 ] \\ & \ll e^{4n} 
        \prod_{i = 0}^{\ell + 1} \Big ( \mathcal{S}_{(T_{i - 1}, T_{i}]} +  C e^{-100 (n_i - n_{i - 1})} \sum_{p | c \implies p \in (T_{i}, T_{i + 1}]} \frac{|\gamma(c)|^2}{c} \Big ),
  \end{aligned}
   \end{equation}
   with $C > 0$ an absolute constant.
  Combining \eqref{eq:closetogood2} and Lemma \ref{le:twooftwo}, we conclude that (with $C > 0$ an absolute constant),
     \begin{equation} \label{eq:closetogood}
  \begin{aligned}
    \mathbb{E} [  & |(\zeta_{\tau} \mathcal{M}_{-1} \ldots \mathcal{M}_{\ell} \mathcal{M}_{\ell}^{(k)})(0)|^4 \cdot |\mathcal{Q}(\tfrac 12 + \ii \tau)|^2 ] \\ & \ll e^{4n} \prod_{i = 0}^{\ell + 1} \Big ( \exp(C (n_i - n_{i - 1})^{10^5} e^{n_i - n} \, (1 + C e^{- (n_i - n_{i - 1})}) \Big ) \\ & \qquad \qquad \times \prod_{i = 0}^{\ell + 1} \Big ( \exp \Big ( - \sum_{p \in (T_{i - 1}, T_{i}]} \frac{4}{p} \Big )  \sum_{p | c \implies p \in (T_{i - 1}, T_{i}]} \frac{|\gamma(c)|^2}{c} \Big )  \\ & \ll e^{4 (n - k)} \sum_{c \geq 1} \frac{|\gamma(c)|^2}{c} \ll e^{4(n - k)} \, \mathbb{E} [|\mathcal{Q}(\tfrac 12 + \ii \tau)|^2].
  \end{aligned}
     \end{equation}
     (Recall that $T_{\ell+1}=\exp(e^k)$.)
   In the last line, we used that
           $$
\prod_{i = 0}^{\ell + 1} \Big ( \sum_{p | c \implies p \in (T_{-1}, T_{i}]} \frac{|\gamma(c)|^2}{c} \Big ) = \sum_{c \geq 1} \frac{|\gamma(c)|^2}{c},
       $$
    which is a consequence of the assumption that the Dirichlet polynomial $\mathcal{Q}$ is degree-$k$ well-factorable.
    We also used Lemma \ref{lem: Transition}.
\end{proof}

\subsection{Proof of Lemma \ref{le:boundinglarge}}
We bound the contribution from  $a_j$'s or $b_j$'s, $j=1,2$, such that $\Omega_i(a_j) > (n_i - n_{i - 1})^{10^5}$ or $\Omega_i(b_j) > (n_i - n_{i - 1})^{10^5}$ for $j=1$ or $2$, using Chernoff's bound (also known in this setting as Rankin's trick).
We write down the argument only for $a_1$ as the other cases are dealt with in an identical fashion. 
Note that since $a_1$ is square-free, we have $\Omega_i(a_1) = \omega_i(a_1)$, where $\omega_i$ denotes the number of distinct prime factors in $(T_{i - 1}, T_{i}]$ counted without multiplicity. 
For any $\rho \in (0,2000)$, the contribution of such $a_1$'s is bounded by,
  \begin{equation}
  \begin{aligned} \label{eq:mainqqq}
&e^{-\rho (n_i - n_{i - 1})^{10^5}} \sum_{\substack{p | c_1, c_2 \implies p \in (T_{i - 1}, T_{i}] \\ \Omega_i(c_1), \Omega_i(c_2) \leq 10 (n_i - n_{i - 1})^{10^4}}} |\gamma(c_1) \gamma(c_2)|   \sum_{\substack{p | a_1, b_1 \implies p \in (T_{i- 1}, T_{i}] \\ p | a_2, b_2 \implies p \in (T_{i - 1}, T_{i}] \\ a_1, b_1, a_2, b_2 \leq T^{1/100}}} e^{\rho\omega_i(a_1)} \\
&  \hspace{3cm} \times\frac{\mu^2(a_1) \mu^2(a_2) \mu^2(b_1) \mu^2(b_2)}{[a_1 b_1 c_1, a_2 b_2 c_2]}  \Big | B_{\mathbf{z}} \Big ( \frac{a_1 b_1 c_1}{(a_1 b_1 c_1, a_2 b_2 c_2)} \Big ) B_{\pi \mathbf{z}} \Big ( \frac{a_2 b_2 c_2}{(a_1 b_1 c_2, a_2 b_2 c_2)} \Big ) \Big | . 
  \end{aligned}
  \end{equation}
  We now claim that $|B_{\mathbf{z}}(m)| \ll d_3(m)$ provided that $\mathbf{z} = (z_1, z_2, z_3, z_4)$ are such that $|z_j| =3^j / \log T$ for all $1 \leq j \leq 4$ and $m \leq T$. 
Here $d_k(m)$ denotes the $k$th divisor function: $d_k(n)=\sum_{n=m_1\dots m_k}1$.
  To prove this, from Lemma \ref{le:bzn}, for every $p^\alpha \leq T$ and integer $\alpha \geq 1$, we have
  $$
  |B_{\mathbf{z}}(p^{\alpha})| \leq  d_2(p^{\alpha}) \, \Big ( 1 + \OO \Big ( \frac{\alpha \log p}{\log T} + \frac{1}{p} \Big ) \Big ). 
  $$

  Therefore, by taking the product over all $p | m$, we obtain
  $$
  |B_{\mathbf{z}}(m)| \ll d_2(m) \prod_{p | m} \Big ( 1 + \OO \Big ( \frac{\alpha \log p}{\log T} \Big ) \Big ) \prod_{p | m} \Big ( 1 + \OO \Big ( \frac{1}{p} \Big ) \Big ) \ll d_2(m) \, \Big (\frac{3}{2} \Big )^{\omega(m)} \ll d_3(m),
  $$
  since $\prod_{p | m} (1 + \OO (\alpha \log p / \log T)) \ll \exp(\log m / \log T) \ll 1$ for $m \leq T$ and $\prod_{p | m} (1 + \OO (1/p)) \ll (3/2)^{\omega(m)}$.

  Furthermore, note that the factors $\mu^2$ in \eqref{eq:mainqqq} ensure that only the square-free $a$'s and $b$'s are counted. In particular, we have $v_p(a_jb_j)\leq 2$ for every $j$. 
Grouping $a_1 b_1$ (resp. $a_2 b_2$) as a single variable with $k_1:=v_p(a_1b_1)$ (resp. $k_2:=v_p(a_2b_2)$), we find that the sum over $a_1, a_2, b_1, b_2$ (for fixed $c_1$ and $c_2$) in (\ref{eq:mainqqq}) is bounded by the Euler product
  \begin{align} \label{eq:eulerrr}
 C \prod_{\substack{ p \in (T_{i - 1}, T_{i}]}} & \Big ( \sum_{\substack{0 \leq k_1,k_2 \leq 2}} \frac{e^{\rho \omega_{i}(p^{k_1})} d_2(p^{k_1}) d_3(p^{k_1 + v_p(c_1)}) d_2(p^{k_2}) d_3(p^{k_2 + v_p(c_2)})}{p^{\max(k_1 + v_p(c_1), k_2 + v_p(c_2))}} \Big ),
 \end{align}
  where $d_2(p^{k_1})$ accounts for the number of choices of $(a_1,b_1)$ giving the single variable $a_1b_1$, and the same for $d_2(p^{k_2})$. Note that we have not used the gcd factors and simply bounded $d_3(m / v)\leq d_3(m)$ for $v\mid m$. 
  
  The sum over the powers $k_1$ and $k_2$  in \eqref{eq:eulerrr} can be bounded further by using the inequalities $d_3(p^{k_2 + \alpha_2}) \leq d_3(p^{k_2}) d_3(p^{\alpha_2})$ and $d_2(p^{k_1}) = k_1 + 1$. This shows that the factor in \eqref{eq:eulerrr} is
  \begin{align} \nonumber
    \leq \frac{d_3(p^{v_p(c_1)}) d_3(p^{v_p(c_2)})}{p^{\max(v_p(c_1), v_p(c_2))}} & \sum_{\substack{0 \leq k_1,k_2 \leq 2}}  \frac{e^{\rho \omega_{i}(p^{k_1})} d_2(p^{k_1}) d_3(p^{k_1}) d_2(p^{k_2}) d_3(p^{k_2})}{p^{\max(k_1 + v_p(c_1), k_2 + v_p(c_2))  - \max(v_p(c_1), v_p(c_2))}}
    \\ \label{eq:estimateaa}
    \leq \frac{d_3(p^{v_p(c_1)}) d_3(p^{v_p(c_2)})}{p^{\max(v_p(c_1), v_p(c_2))}} & \cdot
    \begin{cases}
      1 + 100 e^{\rho} / p & \text{ if } v_p(c_1) = v_p(c_2), \\
      100 e^{\rho} & \text{ if } v_p(c_1) \neq v_p(c_2).
    \end{cases}
  \end{align}
  We notice that the contribution to \eqref{eq:eulerrr} of primes $p \in (T_{i - 1}, T_{i}]$ for which $v_p(c_1)=v_p(c_2)= 0$ is bounded by
    $$
    \ll \Big ( \frac{\log T_{i}}{\log T_{i - 1}} \Big )^{100 e^{\rho}} = e^{100 e^{\rho} (n_{i} - n_{i - 1})}.
    $$
    As a result of the last two equations, the Euler product in \eqref{eq:eulerrr} is bounded by
    \begin{equation}
    \label{eqn: bound c1c2 fixed}
    \ll d_3(c_1) d_3(c_2) \, \frac{f(c_1, c_2)}{[c_1, c_2]} \, \Big ( \frac{\log T_{i}}{\log T_{i - 1}} \Big )^{100 e^{\rho}},
    \end{equation}
    where $f(c_1, c_2)$ is a multiplicative function of two variables such that $f(p^{\alpha}, p^{\alpha}) = 1 + 100 e^{\rho} / p$ for all  $\alpha \geq 1$ and $f(p^{\alpha}, p^{\beta}) = 100 e^{\rho}$ for $\alpha\geq 0$ and $\beta\geq 0$ with $\alpha \neq \beta$. 
    Going back to Equation \eqref{eq:mainqqq}, it remains to estimate the sum over $c_1$ and $c_2$ using \eqref{eqn: bound c1c2 fixed}:
    \begin{align*}
      \sum_{\substack{p | c_1, c_2 \implies p \in (T_{i -1}, T_{i}] \\ \Omega_i(c_1) , \Omega_i(c_2) \leq 10 (n_i - n_{i - 1})^{10^4}}} & \frac{|\gamma(c_1) \gamma(c_2)| f(c_1, c_2)}{[c_1, c_2]} \, d_3(c_1) d_3(c_2) \\ 
      & \leq e^{1000 (n_{i} - n_{i - 1})^{10^4}} \sum_{\substack{p | c_1, c_2 \implies p \in (T_{i - 1}, T_{i}] \\ \Omega_i(c_1), \Omega_i(c_2) \leq 10 (n_i - n_{i - 1})^{10^4}}} \frac{|\gamma(c_1) \gamma(c_2)| f(c_1, c_2)}{[c_1, c_2]},
    \end{align*}
    where the restriction on the number of prime factors of $c_1$ and $c_2$ is used to trivially bound $d_3$. Furthermore, using the inequality $|\gamma(c_1) \gamma(c_2)| \leq \frac{1}{2} |\gamma(c_1)|^2 + \frac{1}{2}|\gamma(c_2)|^2$  and dropping the restriction on $\Omega_i(c_1)$, the sum over $c_1, c_2$ above is less than
    \begin{equation}
    \label{eqn: sum gamma c1c2}
    \sum_{\substack{p | c_2 \implies p \in (T_{i - 1}, T_{i}] \\ \Omega_i(c_2) \leq 10 (n_i - n_{i - 1})^{10^4}}} |\gamma(c_2)|^2 \sum_{\substack{p | c_1 \implies p \in (T_{i - 1}, T_{i}]}} \frac{f(c_1, c_2)}{[c_1, c_2]}.
        \end{equation}
        We note now that by the definition of $f$ the sum over $c_1$ with $c_2$ fixed can be bounded by an Euler product
        \begin{align*}
          & \sum_{p | c_1 \implies p \in (T_{i - 1}, T_{i}]} \frac{f(c_1, c_2)}{[c_1, c_2]} = \prod_{p \in (T_{i - 1}, T_{i}]} \Big ( \sum_{k \geq 0} \frac{f(p^k, p^{v_p(c_2)})}{p^{\max(k, v_p(c_2))}} \Big )
              \\ & = \prod_{p \in (T_{i - 1}, T_{i}]} \Big ( \sum_{0 \leq k < v_{p}(c_2)} \frac{f(p^k, p^{v_p(c_2)})}{p^{v_{p}(c_2)}} + \frac{f(p^{v_p(c_2)}, p^{v_p(c_2)})}{p^{v_p(c_2)}} + \sum_{k > v_p(c_2)} \frac{f(p^k, p^{v_p(c_2)})}{p^k} \Big ),
        \end{align*}             
and using the definition of $f$ we conclude that, 
\begin{align*}
  \sum_{p | c_1 \implies p \in (T_{i - 1}, T_{i}]} \frac{f(c_1, c_2)}{[c_1, c_2]} \leq \prod_{\substack{p \in (T_{i - 1}, T_{i}] \\ v_{p}(c_2) = 0}} \Big ( 1 + \frac{200 e^{\rho}}{p} \Big ) \prod_{\substack{p \in (T_{i - 1}, T_{i}] \\ v_{p}(c_2) > 0}} \Big ( \frac{200 e^{\rho} \, v_{p}(c_2)}{p^{v_{p}(c_2)}} \Big ). 
  \end{align*}
          Since $c_2$ has at most $10 (n_i - n_{i - 1})^{10^4}$ prime factors counted with multiplicity, this is 
          \begin{equation}
          \label{eqn: f c1c2}
 \leq         \frac{1}{c_2}  \exp \Big (200 e^{\rho} (n_i - n_{i - 1}) + 10^4 \rho (n_i - n_{i - 1})^{10^4 + 1} \Big ).
          \end{equation}
          As a result, putting together equations \eqref{eqn: f c1c2}, \eqref{eqn: sum gamma c1c2} and \eqref{eqn: bound c1c2 fixed},
          we see that \eqref{eq:mainqqq} is bounded by
            $$
\ll e^{- \rho (n_i - n_{i - 1})^{10^5} +300 e^{\rho} (n_i - n_{i - 1}) +10^4\rho (n_i - n_{i - 1})^{10^4+1}} \sum_{\substack{p | c_2 \implies p \in (T_{i - 1}, T_{i}]}} \frac{|\gamma(c_2)|^2}{c_2} .
  $$
   Here, the condition on $\Omega_i(c_2)$ was dropped.
  Choosing $\rho = 1000$ we see that this is
  $$
  \leq e^{-100 (n_i - n_{i - 1})^{10^5}} \sum_{\substack{p | c \implies p \in (T_{i - 1}, T_{i}]}} \frac{|\gamma(c)|^2}{c},
    $$
    as needed.

  \subsection{Proof of Lemma \ref{le:twooftwo}}
  
       A crucial step in the proof of Lemma \ref{le:twooftwo} will be the following estimate for $\mathfrak S_I(c_1, c_2)$ defined in Equation \eqref{eqn: mathfrak S}.
       \begin{lemma}\label{le:estimater}
         Let $\ell \geq 0$ be such that
         $\exp( 10^6 (n_{\ell + 1} - n_{\ell})^{10^5} e^{n_{\ell + 1}} ) \leq \exp ( \tfrac{1}{100} e^n )$. 
      Let $I \subset [\exp(e^{1000}), \exp(e^{n_{\ell + 1}})]$ be an interval. Let $z_1, \ldots, z_4$ be complex number with $|z_j| = 3^j / e^n$ for $j = 1,2,3,4$. Let $\mathbf{z} = (z_1, z_2, z_3, z_4)$ and $\pi \mathbf{z} = (z_3, z_4, z_1, z_2)$. Given integers $c_1, c_2 \geq 1$ with at most $10 (n_{\ell + 1} - n_{\ell})^{10^4}$ prime factors, consider $\mathfrak S_{I}(c_1, c_2)$ as in Equation (\ref{eqn: mathfrak S}). 

    Write $c_1 = r c_1'$ and $c_2 = r c_2'$ with $r := (c_1, c_2)$. Then, we have
      $$
      |\mathfrak S_I(c_1, c_2)| \leq \prod_{p \in I} \Big ( 1 - \frac{4}{p} + e^{4000}\,  \frac{\log p}{p e^n} + \frac{e^{4000}}{p^2} \Big )  \frac{h(c_1') h(c_2')}{r c_1' c_2'}, 
      $$
      where $h$ is a multiplicative function such that, for all prime $p \geq 2$ and integer $\alpha \geq 1$, 
      $$
      h(p^{\alpha}) = \frac{e^{5000} \alpha^2 \log p}{e^n}.
      $$
    \end{lemma}
The lemma is proved in the next subsection. Assuming it, we prove Lemma \ref{le:twooftwo}. 
We start by writing $c_1 = r c_1'$, $c_2 = r c_2'$
            and then we use the inequality $|\gamma(r c_1') \gamma(r c_2')| \leq \tfrac 12(|\gamma(r c_1')|^2 +|\gamma(r c_2')|^2)$. 
 Lemma \ref{le:estimater} then reduces the proof to evaluating
          $$
           \sum_{\substack{p | r, c'_1 \implies p \in I  \\ \Omega_i(r c_1') \leq 10 (n_{i} - n_{i -1})^{10^4}}} \frac{|\gamma(r c_1')|^2 h(c_1')}{r c'_1} \sum_{\substack{p | c_2' \implies p \in (T_{i - 1}, T_{i}]}} \frac{h(c_2')}{c_2'}.
             $$
        The definition of $h(p^\alpha)$ implies
              \begin{align*}
\sum_{p | c_{2}' \implies p \in I} \frac{h(c_2')}{c_2'} & = \prod_{p \in I} \Big ( 1 + \sum_{\alpha \geq 1} \frac{e^{5000} \alpha^2 \log p}{p^{\alpha} e^n}  \Big ) \leq \exp ( e^{6000} e^{n_i - n} ),
              \end{align*}
              using the fact that $I \subset [T_{i - 1}, T_{i}]$ and the bound
              $
              \sum_{p \leq T_i} \frac{\log p}{p} \leq 2 e^{n_i}.
              $

              Finally, it remains to bound
              \begin{equation}
              \label{eqn: gamma ho}
              \sum_{\substack{p | r \implies p \in I \\ \Omega_i(r c_1') \leq 10 (n_i - n_{i - 1})^{10^4}}} \frac{|\gamma(r c_1')|^2 h(c_1')}{r c_1'} = \sum_{\substack{p | m \implies p \in I \\ \Omega_i(m) \leq 10 (n_i - n_{i - 1})^{10^4}}} \frac{|\gamma(m)|^2}{m} \, g(m),
              \end{equation}
              where $g$ is a multiplicative function such that, 
              $$
              g(p^{\alpha}) := \sum_{r m = p^{\alpha}} h(m) \leq 1 + \frac{e^{5000} \alpha^3 \log p}{e^n},
              $$
              for every prime $p$ and integer $\alpha \geq 1$. Since $m$ has at most $10 (n_i - n_{i - 1})^{10^4}$ prime factors, all of which are less than $\exp(e^{n_i})$, we have
\begin{multline*}
              g(m)  \leq \exp \Big ( e^{5000} \sum_{p | m} \frac{v_p(m)^3 \log p}{e^n} \Big )  \leq
              \exp \Big ( e^{5000} \, 10000 (n_i - n_{i - 1})^{4 \cdot 10^4} \, e^{n_i - n} \Big )\\
               \leq \exp \Big ( e^{6000} (n_i - n_{i - 1})^{4 \cdot 10^4} e^{n_i - n} \Big ). 
\end{multline*}
            Therefore we obtain a final bound for \eqref{eqn: gamma ho}
              $$
              \leq \exp \Big ( e^{6000} (n_i - n_{i - 1})^{4 \cdot 10^4} e^{n_i - n} \Big ) \sum_{\substack{p | c \implies p \in I}} \frac{|\gamma(c)|^2}{c},
              $$
              thereby concluding the proof of the lemma.

\subsection{Proof of Lemma \ref{le:estimater}}

   Using multiplicativity, we can write $\mathfrak S_I(c_1,c_2)$ given in Equation \eqref{eqn: mathfrak S} as a product
   \begin{equation*}
   \begin{aligned}
   \mathfrak S_I(c_1,c_2)= \prod_{p \in I} \mathcal{P}_{\mathbf{z}, \pi \mathbf{z}}(c_1, c_2, p),
   \end{aligned}
   \end{equation*}
   where $\mathcal{P}_{\mathbf{z}, \pi \mathbf{z}}(c_1, c_2, p)$ is defined as
   \begin{align} \label{eq:defofp}
 \sum_{\substack{0 \leq k_1, k_2 \leq 2}} \frac{f(p^{k_1}) f(p^{k_2})}{p^{\max(k_1 + v_p(c_1), k_2 + v_p(c_2))}} B_{\mathbf{z}} \Big ( \frac{p^{k_1 + v_p(c_1)}}{p^{\min(k_1 + v_p(c_1), k_2 + v_p(c_2))}} \Big ) B_{\pi \mathbf{z}} \Big ( \frac{p^{k_2 + v_p(c_2)}}{p^{\min(k_1 + v_p(c_1), k_2 + v_p(c_2))}} \Big ).
   \end{align}
    Writing $c_1 = r c_1'$ and $c_2 = r c_2'$ with $r = (c_1, c_2)$, we begin by noticing that
   $$
   \mathcal{P}_{\mathbf{z}, \pi \mathbf{z}}(c_1, c_2, p) = \frac{1}{p^{v_p(r)}} \, \mathcal{P}_{\mathbf{z}, \pi \mathbf{z}} (c_1', c_2', p). 
   $$
    It therefore remains to understand $\mathcal{P}_{\mathbf{z}, \pi \mathbf{z}}(c_1', c_2', p)$. Since $(c_1', c_2') = 1$ there are only two possibilities to consider: either $(p, c_1' c'_2) = 1$ or $p$ divides only one of $c_1'$, $c_2'$.

 On one hand, if $(p, c_1' c_2') = 1$, Lemma \ref{le:bzn} in Appendix \ref{se:moments} yields
  \begin{equation}
  \label{eqn: miracle 1}
  \begin{aligned}
     \mathcal{P}_{\mathbf{z}, \pi \mathbf{z}}(c_1', c_2', p) & = \sum_{\substack{0 \leq k_1, k_2 \leq 2}} \frac{f(p^{k_1}) f(p^{k_2})}{p^{\max(k_1, k_2)}} B_{\mathbf{z}} \Big ( \frac{p^{k_1}}{p^{\min(k_1, k_2)}} \Big ) B_{\pi \mathbf{z}} \Big ( \frac{p^{k_2}}{p^{\min(k_1, k_2)}} \Big )
     \\ & = \sum_{\substack{0 \leq k_1, k_2 \leq 2}} \frac{f(p^{k_1}) f(p^{k_2})}{p^{\max(k_1, k_2)}} B_{\mathbf{0}} \Big ( \frac{p^{k_1}}{p^{\min(k_1, k_2)}} \Big ) B_{\mathbf{0}} \Big ( \frac{p^{k_2}}{p^{\min(k_1, k_2)}} \Big ) + \OO^{\star} \Big ( e^{4000} \, \frac{\log p}{p e^n} \Big )
       \\ & = 1 - \frac{4}{p} + \OO^{\star} \Big ( e^{4000} \, \frac{\log p}{p e^n} + \frac{e^{4000}}{p^2} \Big ),
 \end{aligned}
  \end{equation}
   where $\OO^{\star}(\cdot)$ means that the implicit constant is $\leq 1$. Note that we have used the simple bound $|f|\leq2$ and that either $k_1\geq 1$ or $k_2\geq 1$ if the above summands differ.

   On the other hand, if $p | c_1' c_2'$ we can assume that $p| c_1'$ and $p \nmid c_2'$ as the case of $p | c_2'$ and $p \nmid c_1'$ is identical. Then we have $v_p(c_2') = 0$ and hence, 
\begin{equation} \label{eq:zxxxx}
   \begin{aligned}
     \mathcal{P}_{\mathbf{z}, \pi \mathbf{z}}(c_1', c_2', p) & = \sum_{1 \leq k_1 \leq 2} \frac{f(p^{k_1})}{p^{k_1 + v_p(c_1')}} \sum_{0 \leq k_2 \leq 2} f(p^{k_2}) B_{\mathbf{z}}(p^{k_1 + v_p(c_1') - k_2}) + \frac{1}{p^{v_p(c_1')}} \sum_{0 \leq k_2 \leq 1} f(p^{k_2}) B_{\mathbf{z}}(p^{v_p(c_1') - k_2})
     \\ & + \begin{cases}
       \frac{1}{p^2} B_{\pi \mathbf{z}}(p) & \text{ if } v_{p}(c_1') = 1, \\
       \frac{1}{p^{v_p(c_1')}} B_{\mathbf{z}}(p^{v_p(c_1') - 2}) & \text{ if } v_{p}(c_1') \geq 2  .
     \end{cases}
   \end{aligned}
\end{equation}
By Lemma \ref{le:bzn}, we have $|B_{\mathbf{z}}(p^j) - B_{\mathbf{0}}(p^j)| \leq 100 \, e^{3000} \, j^2 \,(\log p) e^{-n}$ for all $0 \leq j \leq v_p(c_1') + 2$. Note that this uses that $p^{v_p(c'_1) + 2} \leq \exp(100 (n_{\ell + 1} - n_{\ell})^{10^4} e^{n_{\ell + 1}}) \leq \exp( \tfrac{1}{100} e^n)$. Therefore, using this inequality and \eqref{eq:zxxxx}, we get
\begin{equation} \label{eq:zxxxx3}
\mathcal{P}_{\mathbf{z}, \pi \mathbf{z}}(c_1', c_2', p) = \mathcal{P}_{\mathbf{0}, \mathbf{0}}(c_1', c_2', p) + \OO^{\star} \Big ( e^{4000} \, \frac{v_p(c'_1)^2 \cdot \log p}{p^{v_p(c_1')} e^n}  \Big ),
\end{equation}
where $\OO^{\star}(\cdot)$ is a $\OO(\cdot)$ with implicit constant $\leq 1$. 
We claim that $\mathcal{P}_{\mathbf{0}, \mathbf{0}} (c_1', c_2', p) = 0$. 
For $v_p(c'_1) \geq 2$, this follows from \eqref{eq:zxxxx} and the identity
\begin{equation} \label{eq:zxxxx2}
\sum_{0 \leq k \leq 2} f(p^k) B_{\mathbf{0}}(p^{\ell - k}) = 0 \quad \text{for $\ell \geq 2$.}
\end{equation}
The above identity follows from Lemma \ref{le:bzn} and the identities $\ell - 2 (\ell - 1) + (\ell - 2) = 0$ and $1 - 2 + 1 = 0$. For $v_p(c'_1) = 1$ by \eqref{eq:zxxxx} and \eqref{eq:zxxxx2} it suffices to check that,
$$
\frac{1}{p} \, ( B_{\mathbf{0}}(p) - 2 B_{\mathbf{0}}(1) ) + \frac{1}{p^2} \,B_{\mathbf{0}}(p) = 0.
$$
In particular upon factoring it is enough to check that $B_{\mathbf{0}}(p) \, (1 + 1/p) - 2 = 0$. 
This follows from $B_{\mathbf{0}}(p) = (1 - p^{-2})^{-1} \, (2 - 2/p)$.

We conclude from \eqref{eq:zxxxx3} and $\mathcal{P}_{\mathbf{0}, \mathbf{0}}(c_1', c_2', p) = 0$ that,
\begin{equation}
\label{eqn: miracle 2}
|\mathcal{P}_{\mathbf{z}, \pi \mathbf{z}}(c_1', c_2', p)| \leq e^{4000} \, \frac{1}{p^{v_p(c_1')}} \, \frac{v_p(c_1')^2 \log p}{e^n}.
\end{equation}

Equations \eqref{eqn: miracle 1} and \eqref{eqn: miracle 2} can then be used to get the bound
   \begin{align*}
    |\mathfrak S_{I}(c_1, c_2)| \leq \frac{1}{r c_1' c_2'} \prod_{\substack{p \in I \\ (p, c_1' c_2') = 1}} \Big ( 1 - \frac{4}{p} + e^{4000} \, \frac{\log p}{p e^n} + \frac{e^{4000}}{p^2}  \Big ) \prod_{i = 1}^{2} \prod_{\substack{p \in I \\ p | c'_i}} \Big ( e^{4000} \, \frac{v_p(c_i')^2 \log p}{e^n} \Big )
   \end{align*}
   The restriction $(p, c_1' c_2') = 1$ can be removed by multiplying and dividing by $(1 - 4 / p)$ for primes $p$ with $p | c_1' c_2'$. As a result, the above is bounded by
   \begin{align*}
 |\mathfrak S_I(c_1, c_2)| \leq \frac{1}{r c_1' c_2'} \prod_{\substack{p \in I}} \Big ( 1 - \frac{4}{p} + e^{4000} \, \frac{\log p}{p e^n} + \frac{e^{4000}}{p^2} \Big ) \times \prod_{i = 1}^{2} \prod_{\substack{p \in I \\ p | c'_i}} \Big ( 2 \cdot e^{4000} \, \frac{v_p(c_1')^2 \log p}{e^n} \Big ).
   \end{align*}
This is the claimed bound.

\begin{appendix}

\section{Estimates on Sums over Primes} \label{se:moments}

{\noindent A.1 \ \bf Moments of Dirichlet Polynomials.}
Let $(Z_p, p \text{ prime})$ a sequence of independent and identically distributed random variables, uniformly distributed on the unit circle $|z| = 1$. For an integer $n$ with prime factorization $n = p_1^{\alpha_1} \ldots p_k^{\alpha_k}$ with $p_1, \ldots, p_k$ all distinct, consider 
$$
Z_n := \prod_{i = 1}^{k} Z_{p_i}^{\alpha_i}. 
$$
Then we have $\mathbb{E}[Z_n \overline{Z}_m] = \mathbf{1}_{n = m}$, and therefore, for an arbitrary sequence $a(n)$ of complex numbers, the following holds
$$
\sum_{n \leq N} |a(n)|^2 = \mathbb{E} \Big [ \Big | \sum_{n \leq N} a(n) Z_n \Big |^2 \Big ]. 
$$
The next lemma shows that the mean value of Dirichlet polynomial is close to the one of the above random model. It follows directly from \cite[Corollary 3]{MontgomeryVaughan}.

\begin{lemma}[Mean-value theorem for Dirichlet polynomials] \label{lem: Transition}
  We have,
$$
  \mathbb{E} \Big [ \Big | \sum_{n \leq N} a(n) n^{\ii \tau} \Big |^2 \Big ]  = \Big (1 + \OO \Big ( \frac{N}{T} \Big ) \Big ) \sum_{n \leq N} |a(n)|^2  = \Big ( 1 + \OO \Big ( \frac{N}{T} \Big ) \Big ) \mathbb{E} \Big [ \Big | \sum_{n \leq N} a(n) Z_n \Big |^2 \Big ].
$$
  \end{lemma}

The above implies that Dirichlet polynomials that are supported on integers with prime factors in different ranges behave independently to some extent.
\begin{lemma}[Splitting Lemma] \label{le:splitting}
  Let
  $$
  A(s) := \sum_{\substack{n \leq N \\ p | n \implies p \leq w}} \frac{a(n)}{n^s} \text{ and } B(s) := \sum_{\substack{n \leq N \\ p | n \implies p > w}} \frac{b(n)}{n^s}
  $$
  be two Dirichlet polynomials with $N \leq T^{1/4}$. Then, we have
  $$
  \mathbb{E}[|A(\tfrac 12 + \ii\tau)|^2 \, |B(\tfrac 12 + \ii\tau)|^2 ] =(1+\OO(T^{-1/2}))\mathbb{E}[|A(\tfrac 12 + \ii\tau)|^2 ] \, \mathbb{E}[|B(\tfrac 12 + \ii\tau)|^2].
  $$
\end{lemma}
\begin{proof}
Note that $AB$ is a Dirichlet polynomial with length at most $T^{1/2}$, so  Lemma \ref{lem: Transition} gives
  $$
  \mathbb{E}[|A(\tfrac 12 + \ii\tau)|^2 \, |B(\tfrac 12 + \ii\tau)|^2 ] =(1+\OO(T^{-1/2}))\sum_{n} \frac{1}{n}\,  \Big | \sum_{\substack{n = m_1 m_2 \\ m_1,m_2\leq T^{1/4}\\p | m_1 \implies p \leq w \\ p | m_2 \implies p > w}} a(m_1) b(m_2) \Big |^2.
  $$
  Expanding the square we find that the sum over $n$ is equal to
  $$
  \sum_{\substack{m_1 m_2 = m'_1 m'_2 \\ m_1, m_2, m_1', m_2' \leq T^{1/4}  \\ p | m_1, m_1' \implies p \leq w \\ p | m_2, m_2' \implies p > w}} \frac{a(m_1) \overline{a(m'_1)} b(m_2) \overline{b(m'_2)}}{\sqrt{m_1 m'_1 m_2 m'_2}}
=
  \sum_{\substack{m_1 \leq T^{1/4} \\ p | m_1 \implies p \leq w}} \frac{|a(m_1)|^2}{m_1} \, \sum_{\substack{m_2 \leq T^{1/4} \\ p | m_2 \implies p > w}} \frac{|b(m_2)|^2}{m_2},
  $$
  where we have used the condition on the prime factors of $m_1, m'_1, m_2, m'_2$, which implies that $m_1 = m'_1$ and $m_2 = m'_2$.
   By Lemma \ref{lem: Transition}, the above right-hand side is equal to
      $
      (1 + \OO(T^{-3/4})) \mathbb{E} [|A(\tfrac 12 + \ii\tau)|^2] \, \mathbb{E}[|B(\tfrac 12 + \ii\tau)|^2]
      $, which concludes the proof.
\end{proof}

Moments of the Dirichlet polynomials $S_k$ as defined in Equation \eqref{eqn: Sk} are very close to Gaussian ones provided that the moments are not too large compared to their length. 
This is the content of Lemma \ref{lem: Gaussian moments} below.
For the proof of this, it is useful to consider the random variables
\begin{equation}
\label{eqn: Xp}
X_p(h)=\re \Big ( Z_p \, p^{-\tfrac 12-\ii h} + \tfrac 12 \, Z_p^2 \, p^{-1-2\ii h} \Big ), \text{ $p$ prime, $h\in[-2,2]$,}
\end{equation}
where we remind that the variables $Z_p$ are uniform on the unit circle.
We also use a precise form of Mertens' theorem: 
\begin{equation}
\label{eqn: mertens}
\sum_{a<p\leq b}\frac{1}{p}=\log_2b-\log_2a +\OO(e^{-\kappa\sqrt{\log a}}),
\end{equation}
for some $\kappa>0$.
Such estimates are given in \cite[Corollary 2]{Explicit}, for self-containedness we give a short proof below based on the following quantitative prime number theorem \cite{MonVau2003}: There exists $c>0$ such that uniformly in $x\geq 2$,
\begin{equation}
\label{eqn: PNT}
\pi(x)=|\{p\leq x\}|=\int_{2}^x\frac{\dd t}{\log t}+\OO\left(x e^{-c\sqrt{\log x}}\right).
\end{equation}
This implies by integration by parts
\begin{multline*}
\sum_{a<p\leq b}\frac{1}{p}
=\int_{a}^b \frac{\dd \pi(x)}{x}
=
\frac{\pi(b)}{b}-\frac{\pi(a)}{a}
+\int_{a}^b \frac{\pi(x)}{x^2}\dd x\\
=
\int_2^b\frac{\dd t}{b\log t}-\int_2^a\frac{\dd t}{a\log t}
+\int_{a}^b \frac{\dd x}{x^2}\int_2^x\frac{\dd t}{\log t}
+\OO(e^{-\kappa\sqrt{\log a}})
=\log_2b-\log_2 a+\OO(e^{-\kappa\sqrt{\log a}}),
\end{multline*}
where we chose $\kappa=c/2$.
For the proof of Lemma \ref{lem: Gaussian moments} below we will first need some control on the Laplace transform of the $X_p$'s.

\begin{lemma}\label{lem:basicFour} 
There exists an absolute $C>0$ such that for any $\lambda\in\mathbb{R}$ and $0\leq j\leq k$ we have 
\begin{equation}\label{eqn:Laplace2}
\E\Big[\exp\big(\lambda\sum_{e^j<\log p<e^k}X_p\big)\Big]\leq \exp((k-j+C)\lambda^2/4).
\end{equation}
\end{lemma}

\begin{proof}
As a preliminary, we consider the generating function of the increments $X_p$'s. For any $B>0$ there exists $A>0$ such that $|e^w-(w+w^2/2)|\leq A |w|^3$ uniformly in $|w|<B$, so that uniformly in $p\geq 2$ and $|z| \leq C \sqrt{p}$ we have
\begin{multline*}
\E\big[e^{z X_p}\big]=\int e^{z\left(
\frac{e^{\ii\theta}+e^{-\ii\theta}}{2\sqrt{p}}
+
\frac{e^{2\ii\theta}+e^{-2\ii\theta}}{8p}
\right)}\frac{\dd\theta}{2\pi}\\
=1+z^2\int\left(
\frac{e^{\ii\theta}+e^{-\ii\theta}}{2\sqrt{p}}
+
\frac{e^{2\ii\theta}+e^{-2\ii\theta}}{8p}
\right)^2\frac{\dd\theta}{2\pi}+\OO\left(z^3p^{-\frac{3}{2}}\right)=1+\frac{z^2}{4p}+\OO\left(z^3p^{-\frac{3}{2}}\right).
\end{multline*}
As a consequence, for any $C>0$ there exists $C'>0$ such that for any $p\geq 2$ and $|z| \leq C \sqrt{p}$, we have 
\begin{equation}\label{eqn:Laplace1}
\left|\E\big[e^{z X_p}\big]-e^{\frac{z^2}{4p}}\right|\leq C'\frac{|z^3|}{p^{3/2}}.
\end{equation}
To prove (\ref{eqn:Laplace2}), 
first note that for $1\leq p<\lambda^2/1000$, since the $Z_p$'s are bounded, we trivially have
$$
\E\big[e^{\lambda X_p}\big]\leq e^{\frac{|\lambda|}{\sqrt{p}}+\frac{|\lambda|}{2p}}\leq e^{\frac{\lambda^2}{4p}}.
$$
Moreover, for $p>\lambda^2/1000$, from (\ref{eqn:Laplace1}) there is an absolute $A>0$ such that
$$
\E\big[e^{\lambda X_p}\big]\leq e^{\frac{\lambda^2}{4p}+A\frac{|\lambda|^3}{p^{3/2}}}.
$$
We conclude that 
$$
\E\exp\Big(\lambda\sum_{e^j<\log p<e^k}X_p\Big)\leq \E\exp\Big(\lambda^2\sum_{e^j<\log p<e^k}\frac{1}{4p}+A|\lambda|^3\sum_{n>\frac{\lambda^2}{1000}}n^{-3/2}\Big) \leq \exp((k-j+C)\lambda^2/4),
$$
where we have used (\ref{eqn: mertens}).
\end{proof}

\begin{lemma}[Gaussian moments of Dirichlet polynomials]
\label{lem: Gaussian moments}
For any $h \in [-2,2]$ and integers $k,j,q$ satisfying $n_0 \leq j \leq k$,  $2q \leq e^{n-k} $,
and any constant $A > 0$ we have
\begin{align}\label{eqn:mom2}
\mathbb{E} [ |S_k(h) - S_{j}(h)|^{2q} ] &\ll \frac{(2q)!}{2^q q!} \, \Big(\frac{k - j}{2}\Big)^{q},\\
\label{eqn: moment+1}
\mathbb{E} [ | S_k(h) - S_{j}(h) + A|^{2q} ] &\ll \frac{(2q)!}{q!} \, (k - j)^{q} + (2 A)^{2q}.
\end{align}
Moreover, there exists $C>0$ such that for any $0 \leq j \leq k$,  $2q \leq e^{n-k} $, we have
\begin{equation}\label{eqn:mom3}
\mathbb{E} [ |S_k(h) - S_{j}(h)|^{2q} ] \ll q^{1/2} \frac{(2q)!}{2^q q!} \, \Big(\frac{k - j+C}{2}\Big)^{q}. 
\end{equation}
\end{lemma}

\begin{proof}

  Let $\Phi$ be a smooth function such that $\Phi \geq 0$ for all $x \in \mathbb{R}$, $\Phi(x) \gg 1$ for $x \in [-1,1]$ and $\widehat{\Phi}$ is compactly supported in $[-1,1]$, e.g,
  $$
  \Phi(x) := 
  \Big ( \frac{\sin \pi (t - 1)}{\pi (t - 1)} \Big )^2+
  \Big ( \frac{\sin \pi t}{\pi t} \Big )^2 + 
  \Big ( \frac{\sin \pi (t + 1)}{\pi (t + 1)} \Big )^2.
  $$
  For any two sets of primes $p_1, \ldots, p_k$ and $q_1, \ldots, q_{\ell}$ (with possible multiplicity) such that the products $p_1 \ldots p_k$ and $q_1 \ldots q_{\ell} $ are smaller than $T$, we have
\begin{multline*}
  \int_{\mathbb{R}} \Big ( \frac{p_1 \ldots p_{k}}{q_1 \ldots q_{\ell}} \Big )^{\ii t} \Phi \Big ( \frac{t}{2T} \Big )  = 2T \widehat{\Phi} \Big (2 T \log \frac{p_1 \ldots p_k}{q_1 \ldots q_{\ell}} \Big )  = 2T \widehat{\Phi}(0) \mathbf{1}_{p_1 \ldots p_{k} = q_1 \ldots q_{\ell}}  \\= 2T \widehat{\Phi}(0) \mathbb{E} \Big [ Z_{p_1 \ldots p_{k}} \overline{Z}_{q_1 \ldots q_{\ell}} \Big ]. 
\end{multline*}
Therefore, for any $h \in [-2,2]$ and
for primes $p_1, \ldots, p_{k}$ such that $p_1\dots p_{k} \leq T^{1/2}$, we have by developing the product
  $$
  \int_{\mathbb{R}} \prod_{\ell = 1}^{k} \Big ( \re \Big ( \frac{1}{p_{\ell}^{1/2 + \ii t + \ii h}} + \tfrac 12  \frac{1}{p_{\ell}^{1 + 2 \ii t + 2\ii h}} \Big ) \Big ) \Phi \Big ( \frac{t}{2T} \Big ) \dd t = 2T \widehat{\Phi}(0) \mathbb{E} \Big [ \prod_{i = 1}^{k} X_{p_i}(h) \Big ].
  $$
We therefore find that, for any $h \in [-2,2]$
  \begin{align*}
    \mathbb{E} [ |S_k(h) - S_j(h)|^{2q} ] & \ll \frac{1}{2T} \int_{\mathbb{R}} \Big ( \re \sum_{e^j < \log p \leq e^k}  \Big ( \frac{1}{p^{1/2 + \ii t + \ii h}} + \tfrac 12 \frac{1}{p^{1 + 2 \ii t + 2\ii h}} \Big ) \Big )^{2q} \Phi \Big ( \frac{t}{2T} \Big ) \dd t
    \\ & = \widehat{\Phi}(0) \mathbb{E} \Big [ \Big ( \sum_{e^j < \log p \leq e^k} X_p(h) \Big )^{2q} \Big ],
  \end{align*}
  where we have used $2q \leq e^{n-k}$ to ensure that all prime products in the expansion satisfy $p_1\dots p_{k} \leq T^{1/2}$.

 We now evaluate the above moment. Let $Z$ be a uniform random variable on the unit circle, and $Y$ a centered Gaussian random variable with variance $1/2$. For any integer $m$, we have
$\E [(\re Z)^{2m+1}]=\E [Y^{2m+1}]=0$ and 
\begin{equation}\label{eqn:mom}
\E[(\re Z)^{2m}]=\frac{(2m)!}{2^{2m}(m!)^2}\leq  \frac{(2m)!}{2^{2m} m!}=\E[Y^{2m}].
\end{equation}
Consider also $S_{jk}^{(1)}=\sum_{e^j < \log p \leq e^k} \re\,Z_p\, p^{-\tfrac 12}$, $S_{jk}^{(2)}={\tfrac 12}\sum_{e^j < \log p \leq e^k} \re\, Z_p^2\, p^{-1}$, and
 $G_{jk}^{(1)}=\sum_{e^j < \log p \leq e^k} Y_p\, p^{-\tfrac 12}$, $G_{jk}^{(2)}={\tfrac 12}\sum_{e^j < \log p \leq e^k} Y_p\, p^{-1}$, where $(Y_p)_p$ denote independent centered Gaussian random variables with variance $1/2$.
With expansion through the binomial formula, (\ref{eqn:mom}) implies
$$
\E[(S_{jk}^{(1)})^{2q}]\leq \E[(G_{jk}^{(1)})^{2q}],\ \ \E[(S_{jk}^{(2)})^{2q}]\leq \E[(G_{jk}^{(2)})^{2q}].
$$
Let $\sigma_1^2=\sum_{e^j < \log p \leq e^k}(2p)^{-1}$ and $\sigma_2^2=\sum_{e^j < \log p \leq e^k}(8p)^{-2}$. 
The above equation implies
\begin{multline*}
\E[(S_{jk}^{(1)}+S_{jk}^{(2)})^{2q}]\leq\left(\E[(S_{jk}^{(1)})^{2q}]^{\frac{1}{2q}}+\E[(S_{jk}^{(2)})^{2q}]^{\frac{1}{2q}}\right)^{2q} \\
\leq
\left(\E[(G_{jk}^{(1)})^{2q}]^{\frac{1}{2q}}+\E[(G_{jk}^{(2)})^{2q}]^{\frac{1}{2q}}\right)^{2q}
=(\sigma_1+\sigma_2)^{2q}\, \E(G^{2q})
\end{multline*}
where $G$ is a standard Gaussian random variable.

In the case $j \geq n_0$, the quantitative prime number theorem \eqref{eqn: PNT} implies $\sigma_1^2 = \frac{1}{2}(k - j) + \OO(e^{-\kappa e^{n_0}})$, for some absolute $\kappa>0$. Moreover, we trivially have $\sigma_2^2\leq C e^{-e^j}$, so that $q\sigma_2/\sigma_1\ll 1$ and (\ref{eqn:mom2}) follows.

For (\ref{eqn: moment+1}), we have
$$
  \mathbb{E} [ |S_k(h) - S_j(h) + A|^{2q} ]^{1/2q} \leq (\mathbb{E} [ |S_k(h) - S_j(h)|^{2q} ])^{1/(2q)} + A,
$$
so that
$$
\mathbb{E}[ | S_k(h) - S_j(h) + A]^{2q} ] \leq 2^{2q} \cdot \mathbb{E}[|S_k(h) - S_j(h)|^{2q}] + (2 A)^{2q}
$$
and the claim follows from  (\ref{eqn:mom2}). 

Finally, for (\ref{eqn:mom3}), we rely on (\ref{eqn:Laplace2}) and obtain, for any $\lambda>0$,
$$
 \mathbb{E}[(S_k(h) - S_j(h))^{2q}]\leq \frac{(2q)!}{\lambda^{2q}}\E[\cosh(\lambda(S_k(h) - S_j(h))]\leq
 \frac{(2q)!}{\lambda^{2q}}e^{(k-j+C)\lambda^2/4}.
$$
The choice $\lambda^2=4q/(k-j+C)$ and Stirling's formula give the expected result.
\end{proof}

The Gaussian moments yield Gaussian tails for the probability. Indeed, if a random variable $X$ is such that
$$
\E[X^{2q}]\ll \frac{(2q)!}{2^qq!}\sigma^{2q},
$$
then the Markov inequality together with Stirling's formula and optimization over $q$ yield
$$
\PP(X>V)\ll \exp(-V^2/(2\sigma^2)) \quad \text{for $2q=\lceil V^2/\sigma^2\rceil$.}
$$
This observation applied to $S_k-S_j$ with $j\geq n_0$ yields for any $h\in[-2,2]$ and $V>0$
\begin{equation}
\label{eqn: gaussian tail}
\PP(S_k(h)-S_j(h)>V)\ll \exp(-V^2/(k-j+1))
\end{equation}
as long as $V^2\leq e^{n-k}\cdot (k-j+1)/2$. Note in particular that such large deviation estimates are harder to get as $k$ gets closer to $n$.

We also recall the following analogous bound for the complex partial sums $\widetilde{S}(h)$.
\begin{lemma}\label{lem: Gaussian moments cplx}
  For any $h \in [-2,2]$ and integers $n_0 \leq j \leq k$ and $2 q \leq e^{n - k}$ we have, 
  We have,
  $$
  \mathbb{E}[|\widetilde{S}_k(h) - \widetilde{S}_j(h)|^{2q}] \ll q! (k - j + 1)^{q} 
  $$
\end{lemma}
\begin{proof}
  This follows from \cite[Lemma 3]{Soundararajan}.
\end{proof}
A simple consequence of Lemma \ref{lem: Gaussian moments cplx} is that, 
\begin{equation}\label{eqn: gaussian tail cplx}
\mathbb{P} \Big ( |\widetilde{S}_k(h) - \widetilde{S}_j(h)| > V \Big ) \ll \frac{V + 1}{(k - j + 1)^{1/2}} \exp \Big ( - \frac{V^2}{k - j + 1} \Big ). 
\end{equation}

 {\noindent A.2\ \bf Gaussian approximation.}
Recall the definition of the partial sums in \eqref{eqn: gaussian walk} for $h=0$:
\begin{equation}
\mathcal S_{k}= \sum_{e^{1000}<\log p \leq e^k}X_p.
\end{equation}
We have the following simple estimate for the probability density function of $\mathcal{S}_k$. 
\begin{lemma} \label{le:saddlepoint}
  Let $|v| \leq 100 r$.
  Then, for $r > 1000$ and for all $\Delta \geq 1$, we have
  $$
  \mathbb{P}(\mathcal{S}_r \in [v, v + \Delta^{-1}]) \asymp \frac{1}{\Delta} \cdot \frac{1}{\sqrt{r}} \exp \Big ( - \frac{v^2}{r} \Big ) .
  $$
\end{lemma}
\begin{proof}
  We will merely sketch the proof of this standard result (see e.g \cite[Theorem 1]{LargeDeviations1} or \cite[Theorem 2.1]{LargeDeviations2} for more detailed accounts).
  The probability density function of $\mathcal{S}_r$ can be written (by inverse Fourier transform and contour deformation)
  $$
  f_r(x) = \frac{1}{2\pi} \int_{- \infty}^{\infty} \mathbb{E}[ \exp \Big ( (\sigma + \ii t ) \mathcal{S}_r \Big ) \Big ] \exp \Big ( - ( \sigma + \ii t ) x \Big ) {\rm d}t\ , \ \sigma = \frac{2x}{r}.
  $$
  It remains therefore to analyze the above integral using the saddle point method.
  First we notice that in the region $0 \leq \re z \leq 200$ we have,
  $$
  \mathbb{E}[e^{z \mathcal{S}_r}] = \exp \Big ( \frac{z^2 r}{4} \Big ) H(z),
  $$
  with $H=H_r$ a function analytic in the strip $0 \leq \re z \leq 200$
  such that $\tfrac 12 \leq |H(z)| \leq 10^3$ and $|H'(z)| \leq 10^{-6}$ uniformly in the strip $0 \leq \re z \leq 200$, and uniformly in $r$. 
(This uses that the $X_p$'s appearing in $\mathcal{S}_r$ have $p > \exp(e^{1000})$).
  The rest of the proof now proceeds by a standard application of the saddle point method. The region $|t| > 100 \log r / \sqrt{r}$ gives a negligible contribution, while the region $|t| \leq 100 \log r  / \sqrt{r}$ contributes,
  $$
  \frac{1}{2\pi} \int_{|t| \leq 100\log r / \sqrt{r}} H \Big ( \frac{2x}{r} + \ii t \Big ) \exp \Big ( \frac{r}{4} \Big (\frac{2x}{r} + \ii t\Big )^2 - \frac{2x^2}{r} - \ii t x \Big ) {\rm d}t.
  $$
  By a Taylor expansion, the above is equal to
  $$
 \exp \Big ( - \frac{x^2}{r} \Big ) \cdot \frac{1}{2\pi} \int_{|t| \leq 100\log r / \sqrt{r}} \Big ( H \Big ( \frac{2x}{r} \Big ) + \OO^{\star} \Big ( 10^{-4} \frac{\log r}{\sqrt{r}} \Big ) \Big ) \exp \Big ( - \frac{t^2 r}{4} \Big ) {\rm d}t\asymp \frac{1}{\sqrt{r}} \exp \Big ( - \frac{x^2}{r} \Big ). 
 $$
 with $\OO^{\star}$ denoting a $\OO$ with implicit constant $\leq 1$. 
 Thus, uniformly in $|x| \leq 100 r$, we have $f_r(x) \asymp r^{-1/2} \exp ( - x^2 / r)$. The result follows upon integrating $x \in [v, v + \Delta^{-1}]$. 
\end{proof}

We now remind the following version of the Berry-Esseen theorem, see for example Corollary 17.2 in \cite{BhaRan1976}. The probability measure $\mathbb{P}$ below is arbitrary, and $\eta_{\mu,\sigma}$ denotes the Gaussian measure with mean $\mu$ and variance $\sigma$.

\begin{lemma}\label{lem:compOne}
Let $W_j$ be a sequence of independent random variables on $(\mathbb{R},\mathcal B(\mathbb{R}),\mathbb{P})$, with associated expectation denoted $\mathbb{E}$, and let $\mathbb{Q}_m$ be the distribution of $W_1+\dots+W_m$. Let
$$
\mu_m=\sum_{j=1}^m\E[W_j], \ \ \ \sigma_m=\sum_{j=1}^m \E[(W_j-\E(W_j))^2],
$$
 and  $\mathcal A$ be the set of intervals in $\mathbb{R}$. There exists an absolute constant $c$ such that 
$$
\sup_{A\in\mathcal A}|\mathbb{Q}_m(A)-\eta_{\mu_m,\sigma_m}(A)|\leq \frac{c}{\sigma_m^{3/2}}\sum_{j=1}^m\E[|W_j-\E(W_j)|^3].
$$
\end{lemma}

The following consequence of Lemma \ref{lem:compOne} compares the probabilistic model
$(\mathcal S_i)_{i\geq 1}$ defined previously with a natural Gaussian analogue.
To state this comparison, remember the definitions (\ref{eqn: gaussian walk}) and (\ref{eqn:trulygaussian}). In the statement below we omit the argument $h$ to mean $h=0$.

\begin{lemma}\label{lem:compProb}
There exists a constant $c>0$ such that, for any interval $A$ and $k\geq 1$,
$$
{\mathbb P}\Big({\mathcal Y}_k \in A\Big)
=\mathbb P\Big(\mathcal N_k\in A\Big)+\OO(e^{-c e^{k/2}}).
$$
\end{lemma}

\begin{proof}
Let 
$$
\mathcal N'_{k} = \sum_{e^{k-1} < \log p \leq e^k}  X_p'
$$
where the $(X'_p, p \text{ prime})$ are centered, independent real Gaussian random variables, with variance $\frac{1}{2p}+\frac{1}{8p^2}$ matching exactly the variance of the summands $X_p$ of $\mathcal S_{k}$.
We apply Lemma \ref{lem:compOne}: All  random variables are centered with matching variances, 
$
C^{-1}\leq \E[({\mathcal Y}_k)^2)]\leq C
$
and (we have $|X_p|<C p^{-1/2}$ deterministically)
$$
\sum_{e^{k-1}<\log p\leq e^k}{\E}[|X_p-{\E}(X_p)|^3]\leq C\sum_{e^{k-1}<\log p\leq e^k}p^{-3/2}\leq C e^{-c e^k},
$$
for some absolute constants $C,c>0$, so that
\begin{equation}\label{eqn:start}
{\mathbb P}\Big({\mathcal Y}_k \in A\Big)
=\mathbb P\Big(\mathcal N'_k\in A\Big)+\OO(e^{-c e^{k}}).
\end{equation}
Moreover, denote $\beta_k=\sum_{e^{k-1}<\log p\leq e^k}\left(\frac{1}{2p}+\frac{1}{8p^2}\right)$.
From Pinsker's inequality and (\ref{eqn: mertens}), we have that the total variation between the distribution of $\mathcal N_k$ and $\mathcal N_k'$ is
\begin{equation}\label{eqn:middle}
2{\rm TV}(\mathcal N_k,\mathcal N_k')^2\leq \int \Big(\log\frac{\dd\eta_{0,1/2}}{\dd\eta_{0,\beta_k}}\Big)\dd\eta_{0,1/2}
=\OO\Big(\big|\beta_k-\frac{1}{2}\big|\Big)=\OO(e^{-c e^{k/2}}).
\end{equation}
Equations (\ref{eqn:start}) and (\ref{eqn:middle}) conclude the proof.
\end{proof}

{\noindent A.3 \ \bf Moments of the Riemann zeta function.}
\begin{lemma}[Second moment of the Riemann zeta function] \label{le:secondmoment}
For all $h \in [-2,2]$, we have
  $$
  \mathbb{E}[|\zeta_{\tau}(h)|^2] \ll e^n.
  $$
\end{lemma}
\begin{proof}
  See \cite[Theorem 2.41]{HardyLittlewood}.
\end{proof}

\begin{lemma}[Fourth moment of the Riemann zeta function] \label{le:fourthmoment}
For all $h \in [-2,2]$, we have
  $$
  \mathbb{E}[|\zeta_{\tau}(h)|^4] \ll e^{4n}. 
  $$
  More generally, for real $|\sigma-1/2| \leq \tfrac{1}{100}$, we have 
  $$
  \mathbb{E}[|\zeta(\sigma + \ii \tau + \ii h)|^4 ] \ll \exp(1 + e^n (2 - 4 \sigma)) \, e^{4n}. 
  $$
\end{lemma}
\begin{proof}
  If $\sigma < \tfrac 12$ the functional equation yields
  $$
  \mathbb{E}[|\zeta(\sigma + \ii \tau + \ii h)|^4 ] \ll \exp(e^n (2 - 4 \sigma)) \mathbb{E}[|\zeta(1 - \sigma + \ii \tau + \ii h)|^4]. 
  $$
  Now uniformly in $\tfrac 12 \leq \sigma \leq \tfrac 34$ by \cite[Theorem D]{HardyLittlewood2} we have
  $$
  \mathbb{E}[|\zeta(\sigma + \ii \tau + \ii h)|^4 ] \ll e^{4n}.
  $$
  The result follows. 
\end{proof}

{\noindent A.4 \ \bf Some useful sums over primes.}
The first lemma justifies the approximation of $e^{-S_k}$ by mollifiers. Recall the definition of $\widetilde S_k$ in \eqref{eqn: complex S_k} and that $\re \ \widetilde S_k=S_k$.
\begin{lemma} \label{le:moliapprox}
  Let $\ell \geq 0$ and $k \in (n_{\ell - 1}, n_{\ell}]$.
  Suppose that $|\widetilde{S}_k(h) - \widetilde{S}_{n_{\ell - 1}}(h)| \leq 10^3 (n_{\ell} - n_{\ell - 1})$. 
  We have,
  \begin{align*}
    e^{-(S_k(h) - S_{n_{\ell - 1}}(h))} & \leq (1 + e^{-n_{\ell - 1}}) \, |\mathcal{M}_{\ell-1}^{(k)}(h)| + e^{-10^5 (n_{\ell} - n_{\ell - 1})} .
  \end{align*}
\end{lemma}
\begin{proof}
  Let
  $$
  R_k(h) := \sum_{\substack{e^{n_{\ell -1}} < \log p \leq e^k \\ \alpha \geq 3}} \frac{1}{\alpha} \, \re~p^{- \alpha (\tfrac 12 + \ii \tau + \ii h)}. 
  $$
  Notice that $|R_k(h)| \leq e^{-2 n_{\ell - 1}}$. As a result, we clearly have
  $$
  e^{-(S_k(h) - S_{n_{\ell-1}}(h))} \leq (1 + e^{-n_{\ell - 1}}) \, e^{-(S_k(h) - S_{n_{\ell - 1}}(h)) - R_k(h)}.
  $$
  Set $s := \tfrac 12 + \ii \tau + \ii h$. 
  Notice that, 
  $$
  e^{-(S_k(h) - S_{n_{\ell - 1}}(h)) - R_k(h)} = \Big | \prod_{p \in (T_{\ell - 1}, \exp(e^k)]} \Big (1 - \frac{1}{p^{s}} \Big ) \Big |.
    $$
Furthermore, setting $V := (n_{\ell} - n_{\ell - 1})^{10^5}$, 
    \begin{align*}
      \prod_{p \in (T_{\ell - 1}, \exp(e^k)]} \Big ( 1 - \frac{1}{p^s} \Big ) & = \sum_{\substack{p | n \implies p \in (T_{\ell - 1}, \exp(e^k)]}} \frac{\mu(n)}{n^s} \\ 
              & = \mathcal{M}_{\ell - 1}^{(k)}(h) + \sum_{\substack{p | n \implies p \in (T_{\ell - 1}, \exp(e^k)] \\ \Omega_{\ell - 1}(n) > V}} \frac{\mu(n)}{n^s}
            \end{align*}
    Therefore it remains to show that the second term is identically small. 
    Notice that we can re-write the second term as
            \begin{equation}\label{eq:tobound}
            \sum_{\ell > V} (-1)^{\ell} \Big ( \sum_{T_{\ell - 1} < p_1 < \ldots < p_{\ell} \leq \exp(e^k)} \frac{1}{(p_1 \ldots p_\ell)^s} \Big ).
              \end{equation}
              Furthermore, using the Girard-Newton identities (see for example Equation 2.14' in \cite{Mac95}), we can re-write the inner sum as follows,
              $$
              \sum_{T_{\ell - 1} < p_1 < \ldots < p_{\ell} < \exp(e^k)} \frac{1}{(p_1 \ldots p_{\ell})^s} = (-1)^{\ell} \sum_{\substack{m_1, \ldots, m_{\ell}, \ldots  \geq 0\\m_1 + 2 m_2 + \ldots + \ell m_{\ell} + (\ell + 1) m_{\ell + 1} + \ldots = \ell }} \prod_{1 \leq j} \frac{(-\mathcal{P}(\ell s))^{m_j}}{m_j!\ j^{m_j}}
              $$
              with
              $$
              \mathcal{P}(s) := \sum_{T_{\ell - 1} < p \leq \exp(e^k)} \frac{1}{p^s}.
              $$
              Using this we can bound the absolute value of \eqref{eq:tobound}, for any $\alpha > 0$, by
              \begin{align} \nonumber
              \sum_{m_1, \ldots, m_{\ell}, \ldots \geq 0} & \exp \Big ( - \alpha V + \alpha m_1 + 2 \alpha m_2 + \ldots + \ell \alpha m_{\ell} + \ldots \Big ) \prod_{1 \leq j} \frac{|\mathcal{P}(j s)|^{m_j}}{m_j! \ j^{m_j}} \\ \label{eq:to} & = e^{-\alpha V} \prod_{1 \leq j} \Big ( \sum_{m_j \geq 0} \frac{(e^{j \alpha} |\mathcal{P}(js)| / j)^{m_j}}{m_j!} \Big ) .
              \end{align}
              By assumption we have $|\mathcal{P}(s)| \leq 10^4 (n_{\ell} - n_{\ell - 1})$ and trivially we have $|\mathcal{P}(2s)| \leq n_{\ell} - n_{\ell - 1} + 2$ and $|\mathcal{P}(\ell s)| \leq 10^{-\ell}$ for $\ell \geq 3$. As a result \eqref{eq:to} is (for $\alpha = 1$) bounded by
              $$
   \ll \exp( - (n_{\ell} - n_{\ell - 1})^{10^5} + 10^5 (n_{\ell} - n_{\ell - 1})) 
              $$
              and the claim follows. 
  \end{proof}

   \begin{lemma} \label{le:bzn}
     Let $p > \exp(e^{1000})$ be a prime and $\alpha \geq 1$, an integer. 
     Given $\pi \mathbf{z} = (z_3, z_4, z_1, z_2)$, define
     $$
     B_{\pi \mathbf{z}}(p^{\alpha}) := \frac{\sum_{j = 0}^{\infty} \sigma_{z_1, z_2}(p^{j}) \sigma_{z_3, z_4}(p^{\alpha + j}) p^{-j}}{\sum_{j = 0}^{\infty} \sigma_{z_1, z_2}(p^j)\sigma_{z_3, z_4}(p^j) p^{-j}},
     $$
       where $\sigma_{z,w}(p^{\alpha}) := \sum_{n m = p^{\alpha}} n^{-z} m^{-w}$.
       Then, uniformly in $|z_i| \leq 3^4 / (\alpha \log p)$ for $i = 1,2,3,4$,    we have
     $$
       |B_{\mathbf{z}}(p^{\alpha}) - B_{\mathbf{0}}(p^{\alpha})| \leq e^{3000} \, \alpha^2 \log p\sum_{i = 1}^{4} |z_i| ,
     $$
     where $\mathbf{0} := (0,0,0,0)$.
     Furthermore
     $$
     B_{\mathbf{0}}(p^{\alpha}) = \Big (1 - \frac{1}{p^2} \Big )^{-1} \, \Big ( 1 + \alpha - \frac{2 \alpha}{p} + \frac{\alpha - 1}{p^2} \Big )
     $$
   \end{lemma}
   \begin{proof}
   We start  with the second claim, Lemma 6.9 of \cite{HughesYoung} (applied at $s=0$) implies
   \begin{equation}\label{eqn:HY}
   B_{\mathbf{z}}(p^{\alpha}) = \frac{B_{\mathbf{z}}^{(0)}(p^{\alpha}) - p^{-1} \, B_{\mathbf{z}}^{(1)}(p^{\alpha}) + p^{-2} \, B_{\mathbf{z}}^{(2)}(p^{\alpha})}{(p^{-z_3} - p^{-z_4}) \, (1 - p^{-2 - z_1 - z_2 - z_3 - z_4})} ,
   \end{equation}
   where
   \begin{align*}
     B_{\mathbf{z}}^{(0)}(p^{\alpha}) & = p^{-z_3 (\alpha + 1)} - p^{-z_4 (\alpha + 1)}, \\
     B_{\mathbf{z}}^{(1)}(p^{\alpha}) & = (p^{-z_1} + p^{-z_2}) \, p^{-z_3 - z_4} \, ( p^{-z_3 \alpha} - p^{-z_4 \alpha}), \\
     B_{\mathbf{z}}^{(2)}(p^{\alpha}) & = p^{-z_1 -z_2 - z_3 - z_4} \, (p^{-z_4 - z_3 \alpha}  - p^{-z_3 - z_4 \alpha}).
   \end{align*}
The second claims follows by estimating this at $\mathbf{z} = \mathbf{0}$.

    Note that, for $|w_i| \leq 200 / (\alpha \log p)$, (\ref{eqn:HY}) gives
     $
     |B_{\mathbf{w}}(p^{\alpha})| \leq e^{2000} \, \alpha
     $.
     Now, by Cauchy's theorem, we have 
     \begin{align*}
     |B_{(z_1, z_2, z_3, z_4)} (p^{\alpha}) - B_{(0,z_2, z_3, z_4)}(p^{\alpha})| & = \Big | \frac{1}{2\pi \ii} \oint_{|w| = 200 / (\alpha \log p)} B_{(w, z_2, z_3, z_4)}(p^{\alpha}) \, \frac{z_1 \dd w}{(w - z_1) w} \Big | \\ & \leq |z_1| \, e^{10} \, \alpha \log p \, \max_{|\mathbf{w}| = 200 / (\alpha \log p)} |B_{\mathbf{w}}(p^{\alpha})|,
     \end{align*}
     where $|\mathbf{w}| = C$ means that $|w_i| = C$ for $i = 1,2,3$, or $4$. Note also that the last bound is true by the maximum modulus principle, since $|z_i| \leq 200 / (\alpha \log p)$. 
     Now, iterating this on each variable $z_2, z_3, z_4$,   using the bound $|B_{\mathbf{w}}(p^{\alpha})| \leq e^{2000} \, \alpha$, and adding the results, we conclude that
     $$
     |B_{\mathbf{z}}(p^{\alpha}) - B_{\mathbf{0}}(p^{\alpha})| \leq e^{3000} \, \alpha^2 \log p  \, \sum_{i = 1}^{4} |z_i|.
     $$
     This proves the first claim. 
\end{proof}

\section{Ballot Theorem}\label{se:ballot}

Recall the definition  (\ref{eqn:trulygaussian}),
$$
\mathcal G_{k} = \sum_{1000 \leq \ell \leq k}  \mathcal{N}_{\ell}, 
$$
where the $\mathcal{N}_{\ell}$'s are centered, independent real Gaussian random variables, with variance $\frac{1}{2}$.
The main result in this section is the following Ballot theorem for the random walk $\mathcal G$, with Gaussian increments. It extends  \cite[Lemma 6.2]{Web11} from a linear to a curved barrier,  and our proof relies on this result.

\begin{prop}
\label{lem: ballot}
Uniformly in $n\geq 1$, $1\leq y\leq 2n$, $n/2\leq k\leq n$ and $m(k) + L_y(k) - 4 \leq w\leq m(k)+U_y(k)$ (see Equation \eqref{eq:barup}), we have for $r := \lceil y / 4 \rceil$, 
\begin{multline}
\label{eqn: ballot}
 \mathbb P(\{\mathcal G_{k}\in (w,w+1]\} \cap \{ \mathcal G_r - m(r) \in [L_y(r), U_y(r)] \} \cap_{r < j\leq k}\{\mathcal G_j<m(j)+ U_y(j) \})\\
 \ll (y+1)\, (U_y(k)+m(k)-w+1)\, k^{-3/2}\, e^{-\frac{w^2}{k}}.
 \end{multline}
\end{prop}

The above proposition is an immediate consequence of the following one.

\begin{prop}\label{lem:GaussianBallot}
For any fixed $c_1>0,0\leq \theta<1/2$, there exists $C$
such that the following holds. Consider arbitrary $k\geq 1$,$|\alpha|<c_1^{-1}$ and $g$ defined on $[0,k]$
satisfying $g(0)=g(k)=0$,
\begin{align}
|g'(x)|<c_1^{-1}\min(x+1,k-x+1)^{\theta-1},\ \ 0\leq x\leq k,\label{eqn:g1}\\
-c_1\min(x+1,k-x+1)^{\theta-2}<g''(x)\leq 0,\ \ 0\leq x\leq k.\label{eqn:g2}
\end{align}
Let $f_y(x)=g(x)+\alpha x+y$.
Then for any such $f_y$ and  $0<y<c_1^{-1}k$, $-c_1^{-1}k<w<f_y(k)$, we have 
\begin{equation}
\mathbb P\Big(\bigcap_{1\leq j\leq k}\Big\{\sum_{1\leq i\leq j} \mathcal{N}_i \leq f_y(j)\Big\} \cap \Big\{\sum_{1\leq i\leq k} \mathcal{N}_i\in (w,w+1]\Big\}\Big)\leq C \frac{(y+1)\,(f_y(k)-w+1)}{k^{3/2}} \,e^{-\frac{w^2}{k}}\ .\label{eqn:barr}
\end{equation}
\end{prop}

\begin{proof} We abbreviate $W_j=\sum_{i\leq j} \mathcal{N}_i$.
Let $\mathbb{P}_n^{w}$ denote the distribution of $(W_1,\dots,W_k)$ conditionally to $W_k=w$, and $\mathbb{E}^w_k$ the corresponding expectation. 
In our range of parameters for any $x\in[w,w+1)$ we have $e^{-x^2/k}\asymp e^{-w^2/k}$. 
It is therefore enough to prove that 
uniformly in the described $f_y$, $y,w$, we have
\begin{equation}\label{eqn:enough2}
{\mathbb P}_k^x\Big(\bigcap_{j\leq k}\Big\{W_j \leq f_y(j)\Big\}\Big)\ll (y+1)\, (f_y(k)-x+1)\, k^{-1}.
\end{equation}
By a linear change of variables eliminating $\alpha$, we have
\begin{equation}
{\mathbb P}_k^x\Big(\bigcap_{j\leq k}\Big\{W_j \leq f_y(j)\Big\}\Big)
=
{\mathbb P}_k^{x-(f_y(k)-y)}\Big(\bigcap_{j\leq k}\Big\{W_j \leq g(j)+y\Big\}\Big)
.
\label{eqn:int22}
\end{equation}
We denote $\bar x=x-(f_y(k)-y)$.
There is a constant 
$c(k)$ independent of all other parameters such that
the above right-hand side
is 
\begin{equation}
c(k)\int_{u_j<y+g(j)}e^{-\sum_{i=1}^{k} (u_{i}-u_{i-1})^2}\prod_{j=1}^{k-1}\dd u_j
=
c(k)\int_{v_j<y} e^{-\sum_{i=1}^{k} (v_{i}-v_{i-1}+g(i)-g(i-1))^2}\prod_{j=1}^{k-1}\dd v_j,
\label{eqn:int11}
\end{equation}
where we use the conventions $u_0=v_0=0$, $u_k=v_k=\bar x$.
From Equation \eqref{eqn:g1}, we have $|g(i)-g(i-1)|\leq c_1^{-1} \min(i,k-i+1)^{\theta-1}$, and Equation (\ref{eqn:g2}) gives $0\leq 2g(i)-g(i-1)-g(i+1)\leq 2 c_1^{-1}\min(i,k-i+1)^{\theta-2}$. These bounds in the expansion of the Hamiltonian together with the assumption $0\leq \theta<1/2$ give
\begin{align}
\notag&\sum_{i=1}^k (v_{i}-v_{i-1}+g(i)-g(i-1))^2\\
=&\sum_{i=1}^k (v_{i}-v_{i-1})^2 +\OO(1)\notag
+2\sum_{i=1}^{k} (v_{i}-i\frac{\bar x}{k}-v_{i-1}+(i-1)\frac{\bar x}{k})(g(i)-g(i-1))\\
\geq&
\sum_{i=1}^k (v_{i}-v_{i-1})^2 +\OO(1)
+\sum_{i=1}^{k} a_i \left(v_i-i\frac{\bar x}{k}\right)\label{eqn:interm2}
\end{align}
where the constants $a_i$ satisfy $0\leq a_i\leq  2 c_1^{-1}\min(i,k-i+1)^{\theta-2}$.
In the last line, we summed by parts to express the sum in the variables $v_i-i\frac{\bar x}{k}$ and in the difference $2g(i)-g(i-1)-g(i+1)$.
Let $\overline{W}_j=W_j-j\frac{\bar x}{k}$. 
With equations (\ref{eqn:int22}), (\ref{eqn:int11}) and (\ref{eqn:interm2}),  Equation (\ref{eqn:enough2}) follows once it is shown that
$$
\mathbb{E}_k^{\bar x}\big[e^{-\sum_{j=1}^{k-1} a_j \overline{W}_j}\1_{\cap_{j\leq k}\{W_j \leq y)\}}\big]\ll \frac{(y+1)(f_y(k)-x+1)}{k}.
$$
As $ab\leq (a^2+b^2)/2$, the above inequality will follow from
\begin{align}
&\mathbb{E}_k^{\bar x}\big[e^{-2\sum_{j\leq k/2} a_j \overline{W}_j}\1_{\cap_{j\leq k}\{W_j \leq y)\}}\big]\ll \frac{(y+1)(f_y(k)-x+1)}{k},\label{eqn:exp1}\\
&\mathbb{E}_k^{\bar x}\big[e^{-2\sum_{j>k/2} a_j \overline{W}_j}\1_{\cap_{j\leq k}\{W_j \leq y)\}}\big]\ll \frac{(y+1)(f_y(k)-x+1)}{k}.\label{eqn:exp2}
\end{align}
We start with (\ref{eqn:exp1}). 
Suppose without loss of generality that $-2\sum_{j\leq k/2} a_j \overline{W}_j>1$. (On the event that this is $<1$, we can bound the exponential term by a constant, the estimate then follows by a standard ballot theorem with constant barrier as in \eqref{eqn: ballot webb}.) 
Let $\varepsilon=(1/2-\theta)/2>0$. 
Note that
there exists a constant $\kappa=\kappa(c_1)>0$ such that for any $u>1$, $-2\sum_{j\leq k/2} a_j \overline{W}_j\in[u,u+1]$ implies that there exists $1\leq r\leq k/2$ such that 
$\overline{W}_r<-\kappa u r^{\frac{1}{2}+\varepsilon}$. This observation together with the union bound gives
\begin{align}
&\mathbb{E}_k^{\bar x}\big[e^{-2\sum_{j\leq k/2} a_j \overline{W}_j}\1_{\cap_{j\leq k}\{W_j \leq y)\}}\big]\notag\\
\ll& \sum_{u\geq 1,r\leq k/2,v\geq \kappa u r^{\frac{1}{2}+\varepsilon}}e^{u}\,\mathbb{P}_k^{\bar x}\left(\{-\overline{W}_r\in[v,v+1]\}\cap_{j\leq k}\{W_j \leq y)\}\right)\notag\\
\ll& \sum_{u\geq 1,r\leq k/2,v\geq \kappa u r^{\frac{1}{2}+\varepsilon}}e^{u}\,\mathbb{P}_k^{\bar x}\left(-\overline{W}_r\in[v,v+1]\right)\sup_{a\in[v,v+1]}\mathbb{P}_{k-r}^{\bar x+a-r\frac{\bar x}{k}}\left(\cap_{1\leq j\leq k-r}\{W_j \leq y-r\frac{\bar x}{k}+a)\}\right)\label{eqn:interm3}
\end{align}
where we used the Markov property for the second inequality. To bound the first probability above, note that under $\mathbb{P}_k^{\bar x}$, the random variable $\overline{W}_r$ is centered, Gaussian with variance $r-\frac{r^2}{k}\asymp r$. For the second probability, 
we will rely on \cite[Lemma 6.2]{Web11}, which can be rephrased  as follows: Uniformly in $m,z_1\geq 1$, $z_2\leq z_1$, we have
\begin{equation}
\label{eqn: ballot webb}
\mathbb{P}_m^{z_2}(\cap_{j\leq m}\{W_j \leq z_1))\ll \frac{(z_1+1)(z_1-z_2+1)}{m}.
\end{equation}
This allows to bound (\ref{eqn:interm3}) with
$$
\mathbb{E}_k^{\bar x}\big[e^{-2\sum_{j\leq k/2} a_j \overline{W}_j}\1_{\cap_{j\leq k}\{W_j \leq y)\}}\big]
\ll \sum_{u\geq 1,1\leq r\leq k/2,v>\kappa u r^{\frac{1}{2}+\varepsilon}}e^{u-c\frac{v^2}{r}}\cdot 
\frac{(y-r\frac{\bar x}{k}+v+1)(y-\bar x+1)}{k}
$$
for some absolute $c>0$. 
The above sum over $v$ and then $u$ is $\ll e^{-c'r^{2\varepsilon}}$ for some $c'>0$ depending on $c_1$.
We conclude that uniformly in our parameters,  (\ref{eqn:interm3}) is bounded with
$$\sum_{1\leq r\leq k/2}e^{-c  r^{2\varepsilon}}
\frac{(y-r\frac{\bar x}{k}+1)(y-\bar x+1)}{k}
\ll\frac{(y+\frac{|\bar x|}{k}+1)(y-\bar x+1)}{k}.
$$
It follows from our hypotheses that $\bar x/k$ is uniformly bounded, so that the above equation gives (\ref{eqn:exp1}). Equation 
(\ref{eqn:exp2}) can be proved the same way, with $r$ now chosen in $[k/2,k]$ and the barrier event between times  $0$ and $r$. This concludes the proof of
(\ref{eqn:enough2}), and the lemma.
\end{proof}

\section{Discretization}
\label{se:discretization}
The lemmas of this section allow to reduce the study the maximum of a Dirichlet polynomial of a given length on a typical interval  to a finite set of points.
\begin{lemma}\label{lem:basicd}
  Let $\varepsilon > 0$ be given. Let $V$ be a smooth function with $V(x) = 1$ for $0 \leq x \leq 1$ and compactly supported in $[-\varepsilon, 1 + \varepsilon]$. 
  Let $D(s)$ be a Dirichlet polynomial of length $N$. Then, for any $t, h_0 \in \mathbb{R}$,
  $$
  D(\tfrac 12 + \ii t + \ii h_0) = \frac{1}{2 + \varepsilon} \sum_{h \in \frac{2\pi \mathbb{Z}}{(2 + \varepsilon) \log N}} D \Big ( \tfrac 12 + \ii t + \ii h \Big ) \widehat{V} \Big ( \frac{(h - h_0) \log N}{2\pi} \Big ). 
  $$
\end{lemma}
\begin{proof}
  This proof is essentially a repetition of \cite[Proposition 2.7]{ArRadOui} with slight differences. 
    Let $G(x) = V(2\pi x / \log N)$, so that $\widehat{G}(x) := \frac{\log N}{2\pi} \widehat{V} (\frac{x \log N}{2\pi})$. 
 By Poisson summation, for any fixed $0\leq n\leq N$, we have
\begin{equation}\label{eqn:Pois}
    \sum_{k \in \mathbb{Z}}  n^{-\frac{2\pi \ii k}{(2 + \varepsilon) \log N}} \widehat{G} \Big ( \frac{2\pi k}{(2 + \varepsilon) \log N} - h_0 \Big ) = \sum_{\ell \in \mathbb{Z}} \int_{\mathbb{R}} n^{-\frac{2\pi \ii x}{(2 + \varepsilon) \log N}} \widehat{G} \Big ( \frac{2\pi x}{(2 + \varepsilon) \log N} - h_0 \Big ) e^{-2\pi \ii \ell x} \dd x.  
\end{equation}
For fixed $\ell$, by inverse Fourier transform the above integral is 
\begin{multline*}
  \frac{(2 + \varepsilon) \log N}{2\pi} \int_{\mathbb{R}} e^{-\ii x(\log n+(2+\varepsilon)\ell\log N)} \widehat{G}(x - h_0) {\rm d} x \\= \frac{(2 + \varepsilon) \log N}{2 \pi} e^{-\ii h_0(\log n+(2+\varepsilon)\ell\log N)}G\left(\frac{\log n+(2+\varepsilon)\ell\log N}{2\pi}\right).
\end{multline*}
From the compact support assumption on $V$, for $0\leq n\leq N$ the above right-hand side is nonzero only for $\ell=0$.
Equation (\ref{eqn:Pois}) can therefore be written as
$$
n^{-\ii h_0}=\frac{1}{2+\varepsilon}\sum_{h\in\frac{2\pi \mathbb{Z}}{(2 + \varepsilon) \log N}}
n^{-\ii h}
\widehat{V} \Big ( \frac{(h - h_0) \log N}{2\pi} \Big).
$$
This concludes the proof by linearity.
\end{proof}

The following is a particular case of  \cite[Corollary 2.8]{ArRadOui}.
\begin{lemma} \label{le:zetad}
  Let $\mathcal{T}_n$ be a set of $e^{-n - 100}$ well-spaced points in $[-2, 2]$ with $n=\log_2T$. There exists an absolute constant $C > 0$ such that for any $A>0$ and $v \geq 1$, 
  $$
  \mathbb{P}(\max_{|h| \leq 1} |\zeta(\tfrac 12 +\ii \tau + \ii h)| > v) \leq \mathbb{P}(\max_{h \in \mathcal{T}_n} |\zeta(\tfrac 12 +\ii \tau + \ii h)| > v / C) + \OO_A(e^{-A n})
  $$
\end{lemma}
\begin{proof}
  By \cite[Proposition 2]{BombieriFriedlander} for $t \in [T, 2T]$, the zeta function is well-approximated by a Dirichlet polynomial of length $T$:
  \begin{equation}
  \label{eqn: zeta approx}
  \zeta(\tfrac 12 + \ii t) = \sum_{n \leq T} \frac{1}{n^{1/2 + \ii t}} \, \Big ( 1 - \frac{\log n}{\log T} \Big )^{100} + \OO(T^{-100}) =: D(t) + \OO(T^{-100}). 
  \end{equation}
 Lemma \ref{lem:basicd} implies, for any $|h_0| \leq 1$, 
  $$
  |D(t + h_0)| \leq \sum_{h \in e^{-n - 100} \mathbb{Z}} |D(t + h)| \cdot \Big | \widehat{V} \Big ( \frac{(h - h_0) e^n}{2\pi} \Big ) \Big |,
  $$
  where $V$ is a smooth compactly supported function such that $V(x) = 1$ for $0 \leq x \leq 1$.  
  In particular, for any $|h_0| \leq 1$, we have
  $$
  |D(t + h_0)| \leq C \max_{h \in \mathcal{T}_n} |D(t + h)| + \mathcal{E}(t)
  $$
  with $C > 0$ an absolute constant and where
  $$
 \mathcal{E}(t) := \sum_{\substack{h \in e^{-n - 100} \mathbb{Z} \\ |h| > 2}} |D(t + h)| \cdot \Big | \widehat{V} \Big ( \frac{(h - h_0) e^n}{2\pi} \Big ) \Big |.
  $$
 Since $|h_0|\leq 1$, the $\widehat{V}$  term decays faster than any polynomial of $e^n$. Lemma \ref{lem: Transition} and the Cauchy-Schwarz inequality therefore give
 $$
 \mathbb{P}(\mathcal{E}(\tau) \geq 1) \leq \mathbb{E} [ \mathcal{E}(\tau) ] \ll_{A} e^{-A n} ,
 $$
 for any given $A > 0$. Putting it all together, we conclude that for $v \geq 1$, and all $T$ sufficiently large,
 $$
 \mathbb{P}(\max_{|h| \leq 1} |\zeta(\tfrac 12 + \ii\tau + \ii h)| > v) \leq \mathbb{P}(\max_{h \in \mathcal{T}_n} |\zeta(\tfrac 12 + \ii\tau + \ii h)| > v / (2C)) + \OO_A(e^{-A n})
 $$
 for any given $A > 0$.
\end{proof}

Lemma \ref{lem:basicd} implies the following discretization for the maximum of Dirichlet polynomials.  
\begin{lemma} \label{le:discretization}
  Let $\mathcal{I}$ be a finite set of indices. 
  Let $D_i$ with $i \in \mathcal{I}$ be a sequence of Dirichlet polynomial of length $ \leq N$.
  Then, for any $\ell\geq 1$, and any $A \geq 100$, 
 \begin{equation}
 \begin{aligned}
  \label{eqn: discretize}
  \max_{|h| \leq 2} \Big ( \sum_{i \in \mathcal{I}} |D_i(\tfrac 12 + \ii \tau + \ii h)|^2 \Big ) \ll_{A}
 & \sum_{|j| \leq 16 \log N} \Big ( \sum_{i \in \mathcal{I}} \Big | D_i \Big (\frac 12 +\ii \tau + \frac{2\pi \ii j}{8 \log N} \Big ) \Big |^{2} \Big ) \\ & + \sum_{|j| > 16 \log N} \frac{1}{1 + |j|^{A}} \Big ( \sum_{i \in \mathcal{I}} \Big |D_i \Big ( \frac 12 + \ii \tau + \frac{2\pi \ii j}{8 \log N} \Big ) \Big |^{2} \Big ).
\end{aligned}
\end{equation}
\end{lemma}
\begin{proof}
  We can apply Lemma \ref{lem:basicd}  to the Dirichlet polynomial $D_i^2$ (its proof  for Dirichlet polynomials of length at most $2N$ only requires minor changes in constants) we get, for any $|h| \leq 2$, 
  \begin{align*}
  |D_i(\tfrac 12 +\ii \tau + \ii h)|^2 \ll_{A}
 & \sum_{|j| \leq 16 \log N} \Big | D_i \Big (\frac 12 +\ii \tau + \frac{2\pi \ii j}{8 \log N} \Big ) \Big |^{2} \\ & + \sum_{|j| > 16 \log N} \frac{1}{1 + |j|^{A}} \Big ( \sum_{i \in \mathcal{I}} \Big |D_i \Big ( \frac 12 +\ii \tau + \frac{2\pi \ii j}{8 \log N} \Big ) \Big |^{2} \Big ),
 \end{align*}
  using the decay bound $\widehat{V}(x) \ll_{A} (1 + |x|)^{-A}$.
  Summing over $i \in \mathcal{I}$ and then taking the supremum over $|h| \leq 2$ yield the claim. 
\end{proof}
Since the zeta function is well-approximated by a Dirichlet polynomial of length $T$ as in \eqref{eqn: zeta approx}, Lemma \ref{le:discretization} can also be used to approximate the moments of the maximum of zeta on a short interval. We choose to prove this directly.
\begin{lemma}\label{le:zetada}
  We have
  $$
  \mathbb{E}[\max_{|h| \leq 1} |\zeta(\tfrac 12 + \ii \tau + \ii h)|^4] \ll e^{5n}.
  $$
\end{lemma}
\begin{proof}
As $\zeta$ is analytic, the function $|\zeta(\tfrac 12 + \ii t + \ii z)|^4$ is subharmonic in the region $|z| < \tfrac{1}{100}$. Therefore for $|h| \leq e^{-n}$ we have,
  $$
  |\zeta(\tfrac 12 + \ii t + \ii h)|^4 \ll e^{2n} \int_{|x|,|y| \leq 2 e^{-n}} |\zeta(\tfrac 12 + \ii t + x + \ii y)|^4 {\rm d} x {\rm d}y. 
  $$
  Summing the above over a grid of $e^{-n - 100}$ well-spaced point we conclude that,
  $$
  \max_{|h| \leq 1} |\zeta(\tfrac 12 + \ii \tau + \ii h)|^4 \ll e^{2n} \sum_{h \in \mathcal{T}_n} \int_{|x|, |y| \leq 2 e^{-n}} |\zeta(\tfrac 12 + \ii \tau + \ii h +  x + \ii y)|^4 {\rm d}x {\rm d}y.  
  $$
  Taking expectation on both sides and using Lemma \ref{le:fourthmoment} we obtain the desired result.
\end{proof}

\end{appendix}

\bibliographystyle{plain}
\bibliography{bib_ZetaMax}

\begin{thebibliography}{10}

\bibitem{arguin2016}
L.-P. Arguin.
\newblock {\em Extrema of Log-correlated Random Variables: Principles and
  Examples}, page 166–204.
\newblock Cambridge University Press, 2016.

\bibitem{ArgBelBou15}
L.-P. Arguin, D.~Belius, and P.~Bourgade.
\newblock Maximum of the characteristic polynomial of random unitary matrices.
\newblock {\em Comm. Math. Phys.}, 349(2):703--751, 2017.

\bibitem{ArgBelBouRadSou19}
L.-P. Arguin, D.~Belius, P.~Bourgade, M.~Radziwi\l\l, and K.~Soundararajan.
\newblock Maximum of the {R}iemann zeta function on a short interval of the
  critical line.
\newblock {\em Comm. Pure Appl. Math.}, 72(3):500--535, 2019.

\bibitem{ArgBovKis2013}
L.-P. Arguin, A.~Bovier, and N.~Kistler.
\newblock Genealogy of extremal particles of branching brownian motion.
\newblock {\em Communications on Pure and Applied Mathematics},
  64(12):1647--1676, 2011.

\bibitem{ArRadOui}
L-P. Arguin, F.~Ouimet, and M.~Radziwi{\l\l}.
\newblock Moments of the {Riemann} zeta function on short intervals of the
  critical line.
\newblock {\em arxiv:1901.04061}, 2019.

\bibitem{Bettin}
S.~Bettin, H.~M. Bui, X.~Li, and M.~Radziwi{\l\l}.
\newblock A quadratic divisor problem and moments of the {Riemann}
  zeta-function.
\newblock {\em J. Eur. Math. Soc.}, to appear, 2020.

\bibitem{BhaRan1976}
R.~N. Bhattacharya and R.~Ranga~Rao.
\newblock {\em Normal approximation and asymptotic expansions}.
\newblock John Wiley \& Sons, New York-London-Sydney, 1976.
\newblock Wiley Series in Probability and Mathematical Statistics.

\bibitem{BombieriFriedlander}
E.~Bombieri and J.~B. Friedlander.
\newblock Dirichlet polynomial approximations to zeta functions.
\newblock {\em Ann. Scuola Norm. Sup. Pisa Cl. Sci. (4)}, 22(3):517--544, 1995.

\bibitem{Bra78}
M.~Bramson.
\newblock Maximal displacement of branching brownian motion.
\newblock {\em Comm. Pure Appl. Math.}, 31(5):531--581, 1978.

\bibitem{BraDinZei2016}
M.~Bramson, J.~Ding, and O.~Zeitouni.
\newblock Convergence in law of the maximum of the two-dimensional discrete
  {G}aussian free field.
\newblock {\em Comm. Pure Appl. Math.}, 69(1):62--123, 2016.

\bibitem{LargeDeviations2}
N.~R. Chaganty and J.~Sethuraman.
\newblock Large deviation local limit theorems for arbitrary sequences of
  random variables.
\newblock {\em Ann. Probab.}, 13(1):97--114, 1985.

\bibitem{ChaMadNaj16}
R.~Chhaibi, T.~Madaule, and J.~Najnudel.
\newblock On the maximum of the {${\rm C}\beta {\rm E}$} field.
\newblock {\em Duke Math. J.}, 167(12):2243--2345, 2018.

\bibitem{DinZei2014}
J.~Ding and O.~Zeitouni.
\newblock Extreme values for two-dimensional discrete {G}aussian free field.
\newblock {\em Ann. Probab.}, 42(4):1480--1515, 2014.

\bibitem{FyodorovHiaryKeating}
Y.~V. Fyodorov, G.~A. Hiary, and J.~P. Keating.
\newblock Freezing transition, characteristic polynomials of random matrices,
  and the {Riemann} zeta function.
\newblock {\em Phys. Rev. Lett.}, 108:170601, Apr 2012.

\bibitem{FyodorovKeating}
Y.~V. Fyodorov and J.~P. Keating.
\newblock Freezing transitions and extreme values: random matrix theory, and
  disordered landscapes.
\newblock {\em Philos. Trans. R. Soc. Lond. Ser. A Math. Phys. Eng. Sci.},
  372(2007):20120503, 32, 2014.

\bibitem{HardyLittlewood}
G.~H. Hardy and J.~E. Littlewood.
\newblock Contributions to the theory of the {R}iemann zeta-function and the
  theory of the distribution of primes.
\newblock {\em Acta Math.}, 41(1):119--196, 1916.

\bibitem{HardyLittlewood2}
G.~H. Hardy and J.~E. Littlewood.
\newblock The {A}pproximate {F}unctional {E}quation in the {T}heory of the
  {Z}eta-{F}unction, with {A}pplications to the {D}ivisor-{P}roblems of
  {D}irichlet and {P}iltz.
\newblock {\em Proc. London Math. Soc. (2)}, 21:39--74, 1923.

\bibitem{Harper2nd}
A.~J. Harper.
\newblock On the partition function of the {Riemann} zeta function, and the
  {Fyodorov}--{Hiary}--{Keating} conjecture.
\newblock {\em arXiv:1906.05783}, 2019.

\bibitem{HarperSurvey}
A.~J. Harper.
\newblock The {R}iemann zeta function in short intervals [after {Najnudel}, and
  {Arguin}, {Belius}, {Bourgade}, {Radziwi\l\l}, and {Soundararajan}].
\newblock {\em arxiv:1904.08204, S\'eminaire Bourbaki, to appear in
  Ast\'erisque}, 2019.

\bibitem{HeaRadSou2019}
W.~Heap, M.~Radziwi{\l\l}, and K.~Soundararajan.
\newblock Sharp upper bounds for fractional moments of the {R}iemann zeta
  function.
\newblock {\em Q. J. Math.}, 70(4):1387--1396, 2019.

\bibitem{HughesYoung}
C.~P. Hughes and M.~P. Young.
\newblock The twisted fourth moment of the {R}iemann zeta function.
\newblock {\em J. Reine Angew. Math.}, 641:203--236, 2010.

\bibitem{InghamFourier}
A.~E. Ingham.
\newblock A {N}ote on {F}ourier {T}ransforms.
\newblock {\em J. London Math. Soc.}, 9(1):29--32, 1934.

\bibitem{LargeDeviations1}
C.~Joutard.
\newblock Asymptotic approximation for the probability density function of an
  arbitrary sequence of random variables.
\newblock {\em Statist. Probab. Lett.}, 90:100--107, 2014.

\bibitem{Mac95}
I.~G. Macdonald.
\newblock {\em Symmetric functions and {H}all polynomials}.
\newblock Oxford Mathematical Monographs. The Clarendon Press, Oxford
  University Press, New York, second edition, 1995.
\newblock With contributions by A. Zelevinsky, Oxford Science Publications.

\bibitem{MontgomeryVaughan}
H.~L. Montgomery and R.~C. Vaughan.
\newblock Hilbert's inequality.
\newblock {\em J. London Math. Soc. (2)}, 8:73--82, 1974.

\bibitem{MonVau2003}
H.~L. Montgomery and R.~C. Vaughan.
\newblock {\em Multiplicative number theory. {I}. {C}lassical theory},
  volume~97 of {\em Cambridge Studies in Advanced Mathematics}.
\newblock Cambridge University Press, Cambridge, 2007.

\bibitem{Naj18}
J.~Najnudel.
\newblock On the extreme values of the {R}iemann zeta function on random
  intervals of the critical line.
\newblock {\em Probab. Theory Related Fields}, 172(1-2):387--452, 2018.

\bibitem{PaqZei16}
E.~Paquette and O.~Zeitouni.
\newblock The maximum of the {CUE} field.
\newblock {\em Int. Math. Res. Not. IMRN}, pages 5028--5119, 2018.

\bibitem{Soundararajan}
Kannan Soundararajan.
\newblock Moments of the {R}iemann zeta function.
\newblock {\em Ann. of Math. (2)}, 170(2):981--993, 2009.

\bibitem{Explicit}
R.~Vanlalngaia.
\newblock Explicit {M}ertens sums.
\newblock {\em Integers}, 17:Paper No. A11, 18, 2017.

\bibitem{Web11}
C.~Webb.
\newblock Exact asymptotics of the freezing transition of a logarithmically
  correlated random energy model.
\newblock {\em J. Stat. Phys.}, 145(6):1595--1619, 2011.

\end{thebibliography}

\end{document}